\documentclass{amsart}
\usepackage{amsmath,amssymb, url,color}
\usepackage{amsmath,amssymb,amsthm,mathtools}
\usepackage[margin=3cm]{geometry}
\usepackage{tikz}

\newtheorem{lemma}{Lemma}[section]
\newtheorem{prop}[lemma]{Proposition}
\newtheorem{cor}[lemma]{Corollary}
\newtheorem{thm}[lemma]{Theorem}
\newtheorem{example}[lemma]{Example}
\newtheorem{thm?}[lemma]{Theorem?}

\newtheorem{remark}[lemma]{Remark}

\usepackage{enumitem}
\setenumerate[0]{label=\alph*)}

\begin{document}
\title{Torsion points and isogenies on CM elliptic curves}

\author{Abbey Bourdon}
\author{Pete L. Clark}

\newcommand{\etalchar}[1]{$^{#1}$}
\newcommand{\F}{\mathbb{F}}
\newcommand{\et}{\textrm{\'et}}
\newcommand{\ra}{\ensuremath{\rightarrow}}
\newcommand{\FF}{\F}
\newcommand{\ff}{\mathfrak{f}}
\newcommand{\Z}{\mathbb{Z}}
\newcommand{\N}{\Z^{\geq 0}}
\newcommand{\ch}{}
\newcommand{\R}{\mathbb{R}}
\newcommand{\PP}{\mathbb{P}}
\newcommand{\pp}{\mathfrak{p}}
\newcommand{\qq}{\mathfrak{q}}
\newcommand{\C}{\mathbb{C}}
\newcommand{\Q}{\mathbb{Q}}
\newcommand{\ab}{\operatorname{ab}}
\newcommand{\Aut}{\operatorname{Aut}}
\newcommand{\gk}{\mathfrak{g}_K}
\newcommand{\gq}{\mathfrak{g}_{\Q}}
\newcommand{\OQ}{\overline{\Q}}
\newcommand{\Out}{\operatorname{Out}}
\newcommand{\End}{\operatorname{End}}
\newcommand{\Gal}{\operatorname{Gal}}
\newcommand{\CT}{(\mathcal{C},\mathcal{T})}
\newcommand{\lcm}{\operatorname{lcm}}
\newcommand{\Div}{\operatorname{Div}}
\newcommand{\OO}{\mathcal{O}}
\newcommand{\rank}{\operatorname{rank}}
\newcommand{\tors}{\operatorname{tors}}
\newcommand{\IM}{\operatorname{IM}}
\newcommand{\CM}{\mathbf{CM}}
\newcommand{\HS}{\mathbf{HS}}
\newcommand{\Frac}{\operatorname{Frac}}
\newcommand{\Pic}{\operatorname{Pic}}
\newcommand{\coker}{\operatorname{coker}}
\newcommand{\Cl}{\operatorname{Cl}}
\newcommand{\loc}{\operatorname{loc}}
\newcommand{\GL}{\operatorname{GL}}
\newcommand{\PGL}{\operatorname{PGL}}
\newcommand{\PSL}{\operatorname{PSL}}
\newcommand{\Frob}{\operatorname{Frob}}
\newcommand{\Hom}{\operatorname{Hom}}
\newcommand{\Coker}{\operatorname{\coker}}
\newcommand{\Ker}{\ker}
\newcommand{\g}{\mathfrak{g}}
\newcommand{\sep}{\operatorname{sep}}
\newcommand{\new}{\operatorname{new}}
\newcommand{\Ok}{\mathcal{O}_K}
\newcommand{\ord}{\operatorname{ord}}
\newcommand{\mm}{\mathfrak{m}}
\newcommand{\Ohell}{\OO_{\ell^{\infty}}}
\newcommand{\cc}{\mathfrak{c}}
\newcommand{\ann}{\operatorname{ann}}
\renewcommand{\tt}{\mathfrak{t}}
\renewcommand{\cc}{\mathfrak{a}}
\renewcommand{\aa}{\mathfrak{a}}
\newcommand{\bb}{\mathfrak{b}}
\newcommand\leg{\genfrac(){.4pt}{}}
\renewcommand{\gg}{\mathfrak{g}}
\renewcommand{\O}{\mathcal{O}}
\newcommand{\Spec}{\operatorname{Spec}}
\newcommand{\rr}{\mathfrak{r}}
\newcommand{\rad}{\operatorname{rad}}
\newcommand{\SL}{\operatorname{SL}}
\def\hh{\mathfrak{h}}
\newcommand{\KK}{\mathcal{K}}

\begin{abstract}
Let $\OO$ be an order in the imaginary quadratic field $K$.  For positive integers $M \mid N$, we determine the least degree of an $\OO$-CM point on the modular curve $X(M,N)_{/K(\zeta_M)}$ and also on the modular curve $X(M,N)_{/\Q(\zeta_M)}$: that is, we treat both the case in which the complex multiplication is rationally defined and the case in which we do not assume that the complex 
multiplication is rationally defined.  To prove these results we establish several new theorems on rational cyclic isogenies of
CM elliptic curves.  In particular, we extend a result of Kwon \cite{Kwon99} that determines the set of positive
integers $N$ for which there is an $\OO$-CM elliptic curve $E$ admitting a cyclic, $\Q(j(E))$-rational $N$-isogeny.
\end{abstract}

\vspace*{-.5cm}

\maketitle

\vspace*{-.5cm}

\tableofcontents

\section{Introduction}
\noindent
An elliptic curve $E$ defined over a field $F$ of characteristic $0$ has \textbf{complex multiplication (CM)} if the geometric endomorphism algebra $\End( E_{/\overline{F}}) \otimes_{\Z} \Q$ is an imaginary quadratic field $K$; in this case the endomorphism ring 
$\OO \coloneqq \End E_{/\overline{F}}$ is an order in $K$. 
\\ \indent
In this work we continue our study of two closely related topics: torsion points 
on CM elliptic curves over number fields and CM points on modular curves.  Prior contributions to this program have been made by {Kronecker, Weber, Fricke, Hasse, Deuring, Shimura,}
Olson \cite{Olson74}, Silverberg \cite{Silverberg88}, \cite{Silverberg92}, Parish \cite{Parish89}, Aoki \cite{Aoki95}, \cite{Aoki06}, Ross 
\cite{Ross94}, Kwon \cite{Kwon99}, Prasad-Yogananda \cite{PY01}, Stevenhagen \cite{Stevenhagen01}, Breuer \cite{Breuer10}, 
Lombardo \cite{Lombardo17}, Lozano-Robledo \cite{LR18}, \cite{LR19}, Gaudron-R\'emond \cite{GR18} and by the present authors and our collaborators
\cite{CCS13}, \cite{CCRS14}, \cite{BCS17}, \cite{BCP17}, \cite{BP17}, \cite{CCM19}. 
\\ \indent
The present work is a sequel to \cite{BC18} by the same authors.  In \cite{BC18} we studied the modulo $N$ Galois 
representation attached to a CM elliptic curve defined over a number field $F$ containing the CM field $K$, with applications to torsion 
subgroups of CM elliptic curves defined over a number field $F \supset K$.  We also called for work on the following general problem: given a modular curve $X = X(\Gamma)$ attached to a congruence subgroup $\Gamma$ of $\SL_2(\Z)$\footnote{To be precise, following \cite[\S 4]{Deligne-Rapoport} and \cite[\S 2]{Mazur77} we begin with $N \in \Z^+$ and a subgroup $H$ of $\GL_2(\Z/N\Z)$.  Then we take $\Gamma$ to be the complete preimage of $H \cap \SL_2(\Z/N\Z)$ under the map $\SL_2(\Z) \ra
\SL_2(\Z/N\Z)$.  Let $\zeta_N$ be a primitive $N$th root of unity, so $(\Z/N\Z)^{\times} = \Aut(\Q(\zeta_N)/\Q)$.  Then the field
of definition of $X(\Gamma)$ is $\Q(\Gamma) \coloneqq \Q(\zeta_N)^{\det H}$.  We also put $K(\Gamma) \coloneqq K \Q(\Gamma)$.}, for each imaginary quadratic order $\OO$, determine the degrees of $\OO$-CM points on $X$.
The case $X = X(N) = X(\Gamma(N))$ as a curve over $K(\Gamma(N)) = K(\zeta_N)$ was treated in \cite{Stevenhagen01} and \cite[Thm. 1.1]{BC18}, yielding a generalization of the First Main Theorem of Complex Multiplication to arbitrary quadratic orders.  We also treated the case $X = X_1(N) = X(\Gamma_1(N))$ as a curve over $K(\Gamma_1(N)) = K$.
\\ \indent
This fits into the larger program of understanding points of low degree on modular curves.  For a nice curve $X$ defined over a 
number field $F$, a closed point $p$ on $X$ has \textbf{low degree} if the degree $d_p \coloneqq [F(p):F]$ of the residue field of $p$ is less than the $K$-gonality $\gamma_K(X)$, i.e., the least degree of a nonconstant $K$-morphism $\pi: X \ra \PP^1$.\footnote{It follows from Hilbert's Irreducibility Theorem that there are infinitely many closed points of degree $\gamma_K(X)$.}  The \textbf{sporadic points} $p$, for which there are only finitely many closed points $q$ with $d_q \leq d_p$, comprise a subclass of the points of low degree.  The second author's interest in the subject was kindled almost 15 years ago by the belief that for sufficiently large $N$, all noncuspidal points on $X_1(N)_{/\Q}$ 
of least degree should be CM points.  This would imply that for all sufficiently large $N$, the curves $X_1(N)_{/\Q}$ have 
sporadic CM points, and in \cite[Thm. 5, Thm. 7]{CCS13} it was shown that $X_1(N)_{/\Q}$ has sporadic CM points for all sufficiently large \emph{prime} $N$.  \\ \indent
It is of course also of interest to consider composite $N$.  More generally, the classification of torsion subgroups of elliptic curves over 
degree $d$ number fields is equivalent to classifying all points of degree dividing $\frac{d}{\varphi(M)}$ on the modular curves $X(M,N)_{/\Q(\zeta_M)}$.  In \cite{CCRS14} the authors used extensive computer calculations to classify all possible torsion subgroups of CM elliptic curves defined over a degree $d$ number field 
for all $1 \leq d \leq 13$, and in \cite{BP17} the authors extend the classification to include all odd $d$.
In the present paper we compute for each imaginary quadratic order $\OO$ and all positive integers $M \mid N$, the least degree of an $\OO$-CM point on the modular curve $X(M,N)_{/K(\zeta_M)}$ and on the modular curve $X(M,N)_{/\Q(\zeta_M)}$.  This brings us close 
to a classification of torsion subgroups of CM elliptic curves over number fields of \emph{all degrees}.  The remaining issue is that 
whereas in the case of $X(M,N)_{/K(\zeta_M)}$ every degree of a closed $\OO$-CM point {is divisible by the least degree}, in the case of 
$X(M,N)_{/\Q(\zeta_M)}$ this need not be the case: see Example \ref{LASTEXAMPLE}.  
\\ \\
We now discuss the results in more detail.  Our first main result generalizes \cite[Thm. 1.1]{BC18} from $X_1(N)$ to 
 $X(M,N) = X(\Gamma(M) \cap \Gamma_1(N))$ defined over $K(\Gamma) = K(\zeta_M)$.   To prove it, we analyze the action of the Cartan subgroup $(\OO/N\OO)^{\times}$ on pairs of points on an $\OO$-CM elliptic curve, generalizing the orbit analysis on points of order $N$ in \cite{BC18}. Here, $K(\ff)$ denotes the ring class field of $K$ of conductor $\ff$, which is equal to $K(j(E))$ for an elliptic curve $E$ with CM by the order in $K$ of conductor $\ff$.

\begin{thm} \label{introTHM}
Let $\OO$ be an imaginary quadratic order of conductor $\ff$, and let $M  \mid N $ be positive integers.  
There is a poitive integer $T(\OO,M,N)$, explicitly given in $\S4.1$, such that: for all positive integers $d$, there is a field extension $F / K(\ff)$ of degree $d$ and an $\OO$-CM elliptic curve $E_{/F}$ such that $\Z/M\Z \times \Z/N \Z \hookrightarrow E(F)$ iff $T(\OO,M,N) \mid d$.  

\end{thm}
Next we return to the case of $X = X_1(N)$ but \emph{without} the assumption that the ground field contains the CM field $K$. Write
$\Q(\ff)$ for $\Q(j(E)),$ where $E$ is an elliptic curve with CM by the order of conductor $\ff$ in the imaginary quadratic field $K$. Since $T(\OO,1,N)$ gives the least degree over $K(\ff)$ in which an $\OO$-CM elliptic curve has a rational point of order $N$, the least degree over $\Q(\ff)$ is either $T(\OO,1,N)$ or $2 \cdot T(\OO,1,N)$; it is a matter of whether we can save the factor of $2$. We determine this first in the case where $N=\ell^a$ is a prime power by establishing a relationship between rational cyclic $\ell^a$-isogenies and rational points of order $\ell^a$ in minimal degree.  If $A_{/F}$ is an abelian
variety defined over a number field and $C \subset A$ is an order $N$ cyclic \'etale $F$-subgroup scheme, then there is an abelian extension
$L/F$ of degree dividing $\frac{\varphi(N)}{2}$ and a quadratic twist $A^D$ of $A_{/L}$ such that $A^D(L)$ has a point of order $N$
\cite[Thm. 5.5]{BCS17}.  Moreover, by \cite[Thm. 6.2]{BC18}, if $\OO$ is an imaginary quadratic order of discriminant $\Delta < -4$ and
there is an $\OO$-CM elliptic curve $E$ defined over a number field $F \supset K(\ff)$ with an $F$-rational point of order $N$, then $\frac{\varphi(N)}{2} \mid [F:K(\ff)]$.  From this it follows immediately that if there is an $\OO$-CM elliptic curve defined over any number field $F$ {with an $F$-rational point of order $N$}, then $\frac{\varphi(N)}{2} \mid [F:\Q(\ff)]$.  So we see that working over either $\Q(\ff)$ or $K(\ff)$, the existence
of a rational cyclic $N$-isogeny yields an $\OO$-CM point of order $N$ in the lowest possible degree extension of $\Q(\ff)$ or $K(\ff)$.
Strikingly, when $N = \ell^a$ is a prime power the converse turns out to be true: for $F_0 = \Q(\ff)$ or $K(\ff)$, if we have an $\OO$-CM elliptic curve defined over an extension $F/F_0$ of degree $\frac{\varphi(\ell^a)}{2}$ with an $F$-rational point of order $N$, then {there is an $\OO$-CM elliptic curve defined over $F_0$ with} an $F_0$-rational cyclic $\ell^a$-isogeny.  Moreover, the least degree over
$\Q(\ff)$ in which an $\OO$-CM elliptic curve can have a point of order $\ell^a$ can be computed in terms of isogenies over $\Q(\ff)$ and $K(\ff)$.  

\begin{thm} \label{GenIsogThm}
Let $\OO$ be an imaginary quadratic order, let $\ell$ be a prime number, and let $a \in \Z^+$.  Let $m$ denote the maximum over all $i \in \Z^{\geq 0}$ such that there is an $\OO$-CM elliptic curve $E_{/\Q(\ff)}$ with a $\Q(\ff)$-rational cyclic $\ell^i$-isogeny, and let $M$ denote the supremum over all $i \in \Z^{\geq 0}$ such that there is an $\OO$-CM elliptic curve $E_{/K(\ff)}$ with a $K(\ff)$-rational cyclic $\ell^i$-isogeny.\footnote{{In particular, we allow $M=\infty$.}} The least degree over $\Q(\ff)$ in which there is an $\OO$-CM elliptic curve with a rational point of order $\ell^a$ is as follows:
\begin{enumerate}
\item If $a \leq m$, then the least degree is $T(\OO,\ell^a)$.
\item If $m<a \leq M$, then $\ell^a>2$ and the least degree is $2 \cdot T(\OO,\ell^a)$.
\item If $a>M=m$, then the least degree is $T(\OO,\ell^a)$.
\item If $a>M>m$, then $\ell=2$ and the least degree is  $2 \cdot T(\OO,2^a)$.
\end{enumerate}
The quantities $m$ and $M$ are explicitly computed in Propositions \ref{IsogenyProp} and \ref{lastIsog}, while $T(\OO,\ell^a)=T(\OO,1,\ell^a)$ is given in Theorem \ref{BIGMNTHM}.

\end{thm}

Thus our analysis of least degrees over $\Q(\ff)$ is heavily informed by the classification of rational
cyclic isogenies over both $\Q(\ff)$ and $K(\ff)$.  The classification over $\Q(\ff)$ is due to S. Kwon \cite{Kwon99} unless $K = \Q(\sqrt{-1})$ and $\Q(\sqrt{-3})$.  In fact his work holds verbatim as long as the discriminant of the CM
order is not $-4$ or $-3$ (in other words, it holds for all orders in these two fields except the maximal ones).  We will complete Kwon's
classification for these last two discriminants: Corollary \ref{LASTKWONCOR}.  Moreover we generalize half of Kwon's theorem, as follows (Theorem \ref{GENKWONTHM}): we show that when $F$ is any number field containing neither $K$ nor $\Q(\ell \ff)$ for any $\ell \mid N$, if no $\OO$-CM elliptic curve admits a $\Q(\ff)$-rational cyclic $N$-isogeny, then no $\OO$-CM elliptic curve admits an $F$-rational cyclic $N$-isogeny. Theorem \ref{GENKWONTHM} is a crucial ingredient in the proof of Theorem \ref{GenIsogThm}. 

In Section 7, we show that the classification of least degrees of CM points on $X_1(\ell^a)$ yields a natural criterion for when the least degree of an $\OO$-CM point on $X_1(N)$ is $T(\OO,N)=T(\OO,1,N)$ instead of $2 \cdot T(\OO,N)$. We let $T^{\circ}(\OO,N)$ denote the least degree over $\Q(\ff)$ in which there is an $\OO$-CM elliptic curve with a rational point of order $N$. If $N=\ell^a$ is a prime power, then $T^{\circ}(\OO,\ell^a)$ can be deduced from Theorem \ref{GenIsogThm} and is given explicitly in Theorem \ref{THM5.5}.  

\begin{thm} \label{introX1N}
Let $\OO$ be an imaginary quadratic order.  Let $N \in \Z^+$ have prime power decomposition 
$\ell_1^{a_1} \cdots \ell_r^{a_r}$ with $\ell_1 < \ldots < \ell_r$. The least degree over $\Q(\ff)$ in which there is an $\OO$-CM elliptic curve with a rational point of order $N$ is $T(\OO,N)$ if and only if $T^{\circ}(\OO,\ell_i^{a_i})=T(\OO,\ell_i^{a_i})$ for all $1 \leq i \leq r$. Otherwise the least degree is $2\cdot T(\OO,N)$.
\end{thm}
\noindent
In the last of our main results we return to the case of $X(M,N)$ but without the assumption that the ground field contains the CM field $K$.    For positive integers $M \mid N$, let 
$T^{\circ}(\OO,M,N)$ be the least degree of a field extension $F/\Q(\ff)$ for which there is an $\OO$-CM elliptic curve $E_{/F}$ and an injective 
group homomorphism $\Z/M\Z \times \Z/N\Z \hookrightarrow E(F)$.  The following result generalizes Theorem \ref{introX1N}:

\begin{thm}
Let $\OO$ be an imaginary quadratic order of discriminant $\Delta$. Let \[2 \leq M=\ell_1^{a_1} \cdots \ell_r^{a_r} \mid N=\ell_1^{b_1} \cdots \ell_r^{b_r} \text{ with } \ell_1 < \ldots < \ell_r. \]  The least degree $[F:\Q(\ff)]$ of a number field $F \supset \Q(\ff)$ for which 
there is an $\OO$-CM elliptic curve $E_{/F}$ and an injective group homomorphism $\Z/M\Z \times \Z/N\Z \hookrightarrow E(F)$ is 
$T(\OO,M,N)$ if and only if all of the following conditions hold: $M = 2$, $\Delta$ is even and $T^{\circ}(\OO,\ell_i^{a_i},\ell_i^{b_i}) = T(\OO,\ell_i^{a_i},\ell_i^{b_i})$ for all $1 \leq i \leq r$.  Otherwise the least degree is $2 \cdot T(\OO,M,N)$.
\end{thm}
\noindent
By \cite[Lemma 3.15]{BCS17} and \S 2.5, if $M \geq 3$ then for any $K$-CM elliptic curve defined over a number field $F$ we have 
$F(E[M]) \supset K$ and thus $T^{\circ}(\OO,M,N) = 2 T(\OO,M,N)$ unless $M = 2$.  Thus it suffices to compute $T^{\circ}(\OO,2,N)$, which is 
done in \S 8.  Again Theorem \ref{GENKWONTHM} plays an important role, as does the theory of CM elliptic curves over real number fields as developed in \cite{BCS17}. 
\\ \\
The quantities $T^{\circ}(\OO,2,2^b)$ and $T^{\circ}(\OO,1,\ell^{b_i})$ are computed explicitly in Theorem \ref{THM5.5}, Theorem \ref{Thm8.5}, and Proposition \ref{Prop8.6}.
\\ \\
Working with non-maximal imaginary quadratic orders brings various complications: for instance, if $\OO$ is maximal then every
fractional $\OO$-ideal is proper, and every finite $\OO$-submodule of $E(\C)[\tors]$ is cyclic.  Both of these can fail for non-maximal orders.  For these and other reasons, the study of Galois representations and torsion subgroups
is distinctly easier for $\OO_K$-CM elliptic curves. When $\OO$ has conductor $\ff > 1$, if $E_{/F}$ is any $\OO$-CM elliptic curve defined over a number field, there is a canonical $F$-rational
cyclic $\ff$-isogeny $\iota: E \ra E'$ with $\End E' = \OO_K$, and one can try to use $E'$ to study $E$.  For instance,
if $F \supset K$, then $\# E(F)[\tors] \mid \# E'(F)[\tors]$ \cite[Thm. 1.7]{BC18}.   In Theorem \ref{KKTHM}, we analyze the dual isogeny $\iota^{\vee}: E' \ra E$, computing
the intersection of its kernel $\mathcal{K}$ with any finite $\OO_K$-submodule of $E'$.  This result is used in the proof of Theorem \ref{BIGMNTHM}.  It also has the following consequence: \\ \indent
Let $\iota: E \ra E'$ be as above, defined over a number field $F$ containing $K$, and write
\[ E(F) \cong \Z/s\Z \times Z/e\Z, \ E'(F) \cong \Z/s' \Z \times \Z/e' \Z, \ s \mid e, \ s' \mid e'. \]
In \cite[Lemma 6.7]{BC18} we showed that $s \mid s'$: in other words, if $E$ has full $s$-torsion over $F$, then so does $E'$.  Here we apply the analysis of $\iota^{\vee}$ to show that $e' \mid e$: that is, the exponent of $E'(F)$ divides the exponent of $E(F)$.  This gives a contribution to the problem of how the torsion subgroup of an elliptic curve over a number field varies within a rational isogeny
class, as studied e.g. in \cite{Fujita-Nakamura07}.

\subsection{Acknowledgments} We thank Filip Najman for helpful comments on an earlier draft.

\section{Background}

\subsection{The morphism $X_1(N) \ra X_1(M)$}
For positive integers $M \mid N$, we have a map of modular curves $X_1(N) \ra X_1(M)$ defined over $\Q$.
\\ \\
The following result is well known, but for completeness we give the proof.

\begin{lemma}
\label{DEGREELEMMA}
Let $N \geq 2$.  We have
\begin{equation}
\label{DEGGAMMA1N}
 \deg(X_1(N) \ra X(1)) =  \begin{cases} \frac{N^2 \prod_{p \mid N} \left(1-\frac{1}{p^2} \right)}{2} & N \geq 3 \\ 3 & N = 2 \end{cases}.
\end{equation}
\end{lemma}

\begin{proof}
For $N \in \Z^+$, put
\[ \Gamma_1(N) \coloneqq \left\{ \left[ \begin{array}{cc} a & b \\ c & d \end{array} \right] \mid a,b,c,d \in \Z, \ a -1, b \equiv 0 \pmod{N} \} \subset \SL_2(\Z) \right\} \]
and
\[ \overline{\Gamma_1(N)} \coloneqq \Gamma_1(N)/ \{\pm 1\} \subset \PSL_2(\Z) \coloneqq \SL_2(\Z)/\{\pm 1\}. \]
For integers $M \mid N$, we have
\begin{equation}
\label{DGAMMAEQ1}
 \deg(X_1(N) \ra X_1(M)) = \deg(\overline{\Gamma_1(N)} \backslash \mathcal{H} \ra \overline{\Gamma_1(N)} \backslash
\mathcal{H}) = [\overline{\Gamma_1(M)}:\overline{\Gamma_1(N)}].
\end{equation}
Moreover we have
\begin{equation}
\label{DGAMMAEQ2}
 [\Gamma_1(1):\Gamma_1(N)] = N^2 \prod_{p \mid N} \left(1 - \frac{1}{p^2}\right)
\end{equation}
and
\begin{equation}
\label{DGAMMAEQ3}
[\overline{\Gamma_1(N)}:\Gamma_1(N)] = \begin{cases} 1 & N \leq 2 \\ 2 & N \geq 3 \end{cases}.
\end{equation}
From (\ref{DGAMMAEQ1}), (\ref{DGAMMAEQ2}) and (\ref{DGAMMAEQ3}), the desired result (\ref{DEGGAMMA1N}) follows.
\end{proof}
\noindent
We apply Lemma \ref{DEGREELEMMA} to give ``worst case scenarios'' on lifting torsion points on elliptic curves.

\begin{lemma}
\label{LITTLEDIVLEMMA}
Let $\ell$ be a prime number, and let $1 \leq a < b$ be integers.  Let $F$ be a field of characteristic $0$, let $E_{/F}$ be an elliptic
curve, and let $P \in E(F)$ be a point of order $\ell^a$.  
\begin{enumerate}
\item There is a field extension $L/F$ with $[L:F] \leq \ell^{2(b-a)}$ and a point $Q \in E(L)$ such that $\ell^{b-a} Q = P$ (and thus
$Q$ has order $\ell^b$). 
\item If $\ell^a = 2$, then there is a field extension $L/F$ with $[L:F] \leq 2^{2b-3}$ and an elliptic curve $E'_{/L}$ such that 
$j(E') = j(E)$ and $E'$ has an $L$-rational point $Q$ of order $2^b$.  
\end{enumerate}
\end{lemma}
\begin{proof}
a) The set $\{Q \in E(\overline{F}) \mid \ell^{b-a}Q = P\}$ is a principal homogeneous space for $E[\ell^{b-a}]$ and thus has cardinality $\ell^{2b-2a}$.  \\
b) The pair $(E,P)$ induces a point $p \in Y_1(2)(F)$.  Let $\pi: Y_1(2^b) \ra Y_1(2)$ be the natural modular map.   By (\ref{DEGGAMMA1N}) we have $\deg \pi = 2^{2b-3}$, so
there is a field extension $L/F$ with $[L:F] \leq 2^{2b-3}$ and a point $\widetilde{p} \in Y_1(2^b)(L)$ such that $\pi(\widetilde{p}) = p$.  
It follows that there is a pair $(E',Q)_{/L}$ with $E'$ an elliptic curve and $Q \in E'(L)$ of order $2^b$ that induces the point $\widetilde{p} \in Y_1(2^b)$.  
 (Since $b \geq 2$, the curve $Y_1(2^b)_{/\Q}$ is a fine moduli space and the pair $(E',Q)$ is unique up to isomorphism.  But in fact the existence of a structure defined over the field of moduli holds for all modular curves attached to congruence subgroups of $\SL_2(\Z)$ \cite[p. 274, Prop. VI.3.2]{Deligne-Rapoport}.)  Since $\pi(\widetilde{p}) = p$, we have $j(E') = j(E)$.  
\end{proof}

\begin{lemma}
\label{MODULARPULLBACKLEMMA}
Let $N \geq 3$, let $F$ be a field of characteristic $0$, and let $E_{/F}$ be an elliptic curve.  Then there is a field extension $L/F$ of
degree at most $\frac{N^2 \prod_{p \mid N} \left(1- \frac{1}{p^2} \right)}{\# \Aut E}$ and a twist $E'$ of $E_{/L}$ such that
$E'(L)$ has a point of order $N$.
\end{lemma}
\begin{proof}
Consider the natural modular map $\pi: X_1(N) \ra X(1)$, viewed as a morphsim of curves defined over $F$.  The elliptic curve $E_{/F}$ induces a degree $1$ divisor $[x]$ on $X(1)_{/F}$, so its pullback $D \coloneqq \pi^* [x]$ is an effective divisor on $X_1(N)_{/F}$ of degree
\[ \deg(X_1(N) \ra X(1)) = \frac{N^2 \prod_{p \mid N} \left(1-\frac{1}{p^2} \right)}{2}. \]
Case 1: Suppose $j(E) \neq 0,1728$, so $\# \Aut E = 2$.  Then there is some closed point $y$ in the support of $D$ of degree
at most $\frac{N^2 \prod_{p \mid N} \left(1- \frac{1}{p^2} \right)}{\# \Aut E}$.  Let $L = F(y)$.   By \cite[p. 274, Prop. VI.3.2]{Deligne-Rapoport}
there is an elliptic curve $E'_{/L}$ and an $L$-rational point $P$ of order $N$ on $E'$ such that the pair $(E',P)_{/L}$ induces the
point $y$ on $X_1(N)$.  Since $\pi(y) = \pi([E',P]) = x = [E]$, we have $j(E') = j(E)$ and thus $E'_{/L}$ is a twist of $E_{/L}$.  \\
Case 2: Suppose $j(E) = 1728$, so $\#\Aut E = 4$.  The map $X_1(N) \ra X(1)$ is ramified over $j = 1728$: more precisely
because $\overline{\Gamma_1(N)}$ has no nontrivial elements of finite order and $\overline{\Gamma(1)}$ has an element of order
$2$, we have $D = 2[y] + D'$ for a closed point $y$ and an effective divisor $D'$.  Again we take $L = F(y)$.  Then everything is as in Case 1 except we have the improved upper bound
\[[L:F] = [F(y):F] \leq \frac{\deg D}{2} = \frac{N^2 \prod_{p \mid N} \left(1- \frac{1}{p^2} \right)}{4}, \]
establishing the result in this case. \\
Case 3: Suppose $j(E) = 0$, so $\#\Aut E = 6$.  The map $X_1(N) \ra X(1)$ is ramified over $j = 0$: more precisely
because $\overline{\Gamma_1(N)}$ has no nontrivial elements of finite order and $\overline{\Gamma(1)}$ has an element of order
$3$, we have $D = 3[y] + D'$ for a closed point $y$ and an effective divisor $D'$.  Again we take $L = F(y)$, getting the improved upper bound
\[[L:F] = [F(y):F] \leq \frac{\deg D}{3} = \frac{N^2 \prod_{p \mid N} \left(1- \frac{1}{p^2} \right)}{6}. \qedhere \]
\end{proof}

\subsection{Orders in imaginary quadratic fields and complex multiplication.} Let $K$ be an imaginary quadratic field, with ring of integers $\OO_K$.  Throughout
this paper, $\OO$ denotes an arbitrary $\Z$-order in $K$, i.e., a subring of $K$ that is free of rank $2$ as a $\Z$-module such that
$\OO \otimes_{\Z} \Q = K$.  For such an order $\OO$, we define the \textbf{conductor}
\[ \ff = [\OO_K:\OO]. \]
For each positive integer $\ff$, there is a unique order $\OO(\ff)$ in $K$ of conductor $\ff$, namely
\[ \Z + \ff \OO_K. \]
We let $\Delta = \Delta(\OO)$ be the discriminant of $\OO$ (defined e.g. as the discriminant of the trace form, as for any
$\Z$-algebra that is finitely generated and free as a $\Z$-module).  Put $\Delta_K = \Delta(\OO_K)$; then
\[ \Delta(\OO(\ff)) = \ff^2 \Delta_K. \]
We put
\[\tau_K \coloneqq \frac{\Delta_K+ \sqrt{\Delta_K}}{2},
\]
so $\OO(\ff)=\Z[\ff\tau_K]$.  Also put
\[ w \coloneqq w(\OO) = \# \OO^{\times}, \ w_K \coloneqq \# \OO_K^{\times}. \]
Then we have
\[ w = \begin{cases} 6 & \text{if }\Delta = -3 \\
4 & \text{if }\Delta = -4 \\ 2 & \text{if } \Delta < -4
\end{cases}. \]
For an imaginary quadratic discriminant $\Delta$ 
(i.e., $\Delta \in \Z$ is negative and is $0$ or $1$ modulo $4$), let $H_{\Delta}(j) \in \Q[j]$ be the Hilbert class polynomial, 
the monic separable polynomial whose roots in $\C$ are the $j$-invariants of $\OO$-CM elliptic curves.  It has 
degree $\# \Pic \OO$ and is irreducible over $K$.   We take 
\[ \Q(\ff) \coloneqq \Q[j]/(H_{\Delta}(j)); \]
thus we view $\Q(\ff)$ as an abstract number field and not a subfield of $\C$.  It has $\# \Pic \OO$ different embeddings into $\C$, 
in particular the embedding in which $j$ maps to the $j$-invariant of $\C/\OO$, which is an embedding into $\R$.  We put 
\[ K(\ff) \coloneqq K \Q(\ff), \]
the $\ff$-ring class field of $K$.  For all $\ff \geq 2$ we have \cite[Cor. 7.24]{Cox89}
\begin{equation}
\label{RINGCLASSDEGEQ}
[K(\ff):K^{(1)}] = \frac{2}{w_K} \ff \prod_{p \mid \ff} \left( 1- \left(\frac{\Delta_K}{p}\right) \frac{1}{p} \right),
\end{equation}
where $K^{(1)}$ denotes the Hilbert class field of $K$. In general, for an ideal $I$ of $K$, we let $K^I$ denote the $I$-ray class field of $K$ and $K^{(N)}=K^{N\Ok}$.

\subsection{Galois representations of elliptic curves.} For a field $F$ of characteristic 0, let $F^{\sep}$ be a separable closure of $F$ and let $\gg_F = \Aut(F^{\sep}/F)$ be the absolute Galois group of $F$.
For an elliptic curve $E$ defined over $F$ and $N \in \Z^+$, let
\[ \rho_N: \gg_F \ra \Aut E[N] \cong \GL_2(\Z/N\Z) \]
be the mod $N$ Galois representation, and let
\[ \overline{\rho_N}: \gg_F \ra \Aut E[N] / \Aut(E) \cong \GL_2(\Z/N\Z) / \Aut(E) \]
be the reduced mod $N$ Galois representation. The advantage of the reduced Galois representation is that it is independent of
the chosen $F$-rational model of $E$; it depends only on $j(E)$.

For $N \in \Z^+$, we put
\[ C_N(\OO) = (\OO/N\OO)^{\times}, \]
the mod $N$ Cartan subgroup.  Let $q_N: \OO \ra \OO/N\OO$ be the natural map.  The \textbf{reduced Cartan subgroup} is
\[ \overline{C_N(\OO)} = C_N(\OO)/q_N(\OO^{\times}). \]
We have \cite[Lemma 2.2b)]{BC18}
\[
\#C_N(\OO)=N^2 \prod_{p \mid N} \left( 1- \left(\frac{\Delta}{p}\right) \frac{1}{p} \right)\left(1-\frac{1}{p}\right).
\]   The map $q_N^{\times}: \OO^{\times} \ra (\OO/N\OO)^{\times}$ is injective when $N \geq 3$, and when $N=2$ its kernel is $\{\pm 1\}$.  Thus we also know $\#\overline{C_N(\OO)}$.

If $E_{/F}$ has CM by the order $\OO$ in $K$ and if $F \supset K$, then the Galois action commutes with the $\OO$-action and thus
\[ \rho_N: \gg_F \ra (\OO/N\OO)^{\times},  \  \overline{\rho_N}: \gg_F \ra \overline{C_N(\OO)}. \]
If $F=K(\ff)$, work of Stevenhagen implies $\overline{\rho_N}(\gg_F)=\overline{C_N(\OO)}$. See \cite[$\S 4$]{Stevenhagen01} and \cite[Cor. 1.2]{BC18}.

\subsection{Weber functions and fields of moduli.} If $E$ is an elliptic curve defined over a field $F$ of characteristic $0$, we define
a Weber function $\mathfrak{h}$ to be the composition of the quotient map $E \rightarrow E/\Aut(E)$ with an $F$-isomorphism $E/\Aut(E) \cong \mathbb{P}^1$: thus a Weber function is uniquely specified up to an element of
$\operatorname{PGL_2(F)} = \Aut \mathbb{P}^1_{/F}$.  If $E_{/F}$ is given by a Weierstrass equation $y^2 = x^3 +Ax+B$
with $A,B \in F$, then for $P = (x,y) \in E(\overline{F})$, we may take \cite[Ex. II.5.5.1]{SilvermanII}
\[ \mathfrak{h}(P) = \begin{cases} x & AB \neq 0 \\ x^2 & B = 0 \\ x^3 & A = 0 \end{cases}. \]
We have $B = 0$ iff $j(E) = 1728$ iff $\End E$ is the imaginary quadratic order of discriminant $-4$, and $A = 0$ iff $j(E) = 0$
iff $\End E$ is the imaginary quadratic order of discriminant $-3$.
\\ \indent
For $P \in E(\overline{F})$, the Weber function field $F(\mathfrak{h}(P))$
is model-independent in the following sense: let $E'_{/F}$ be an elliptic curve such that there is an isomorphism $\psi: E_{/\overline{F}} \ra E'_{/\overline{F}}$.  Then $F(\mathfrak{h}(P)) = F(\mathfrak{h}(\psi(P)))$.  This can be seen either because $E'_{/F}$ is obtained
by twisting $E_{/F}$ by a cocycle $\eta \in Z^1(\gg_F,\Aut E)$ or by use of Weierstrass equations as above: cf. \cite[p. 107]{Shimura}.
\\ \indent
For $E_{/F}$ and $P \in E(\overline{F})$ of order $N$, we have $\Q(j(E),\mathfrak{h}(P)) \subset F(P)$.  In fact $
\Q(j(E), \mathfrak{h}(P))$ is the residue field $\Q(x)$ of the closed point $x =[(E,P)]$ on the modular curve $X_1(N)_{/\Q}$. Moreover, there is a model of $E_{/\Q(j(E),(\mathfrak{h}(P))}$ such that $P \in E(\Q(j(E),\mathfrak{h}(P))$ \cite[Prop. VI.3.2]{Deligne-Rapoport}.

\subsection{Addenda on $\R$-structures} In this section we will recall some results on $\OO$-CM elliptic curves defined over the real numbers established in \cite[\S3]{BCS17} and make some addenda.  
\\ \\
If $E_{/\C}$ is an elliptic curve, an \textbf{$\R$-structure} on $E$ is an $\R$-isomorphism class of $\R$-models of $E$.  An elliptic curve admits an $\R$-structure iff $j(E) \in \R$, in which case it admits exactly two $\R$-structures \cite[Lemma 3.3b)]{BCS17}.  We can understand $\R$-structures in terms of the uniformizing lattices in $\C$.   For a lattice $\Lambda \subset \C$, let $g_2(\Lambda)$ and $g_3(\Lambda)$ be the associated Eisenstein series, and put 
\[ E_{\Lambda}: y^2 = 4x^3 - g_2(\Lambda)x - g_3(\Lambda). \] A lattice $\Lambda \subset \C$ is \textbf{real} if $\overline{\Lambda} = \Lambda$.   Then an elliptic curve $E_{/\C}$ has $j(E) \in \R$ iff $E \cong E_{\Lambda}$ for a real lattice $\Lambda$ \cite[Lemma 3.2a)]{BCS17}.  Moreover, two real lattices $\Lambda_1$ and $\Lambda_2$ determine the 
same $\R$-structure on $E$ iff they are $\R$-homothetic: there is $\alpha \in \R^{\times}$ such that $\Lambda_2 = \alpha \Lambda_1$ 
\cite[Lemma 3.2b)]{BCS17}.  
\\ \indent
Let $\OO$ be an imaginary quadratic order, and let $I$ be an ideal of $\OO$.  We say that $I$ is \textbf{primitive} if $I$ is not contained 
in $a \OO$ for any integer $a > 1$; equivalently, if the additive group $(\OO/I,+)$ is cyclic.  We say that $I$ is \textbf{proper} if 
\[ (I:I) \coloneqq \{x \in K \mid x I \subset I\} = \OO. \]
By \cite[Lemma 3.1]{BCS17} an ideal of $\OO$ is proper iff it is projective and thus by \cite[Thm. 7.29]{Clark-CA} an ideal $I$ is proper 
iff it is locally principal: for all maximal ideals $\pp$ of $\OO$, the pushforward $I \OO_{\pp}$ of $I$ to the local ring $\OO_{\pp}$ is principal.  \\ \indent 
Let $\OO$ be an imaginary quadratic order, and let $E_{/\C}$ be an $\OO$-CM elliptic curve with $j(E) \in \R$, so as above 
there are two $\R$-homothety classes of real lattices $\Lambda$ such that $E \cong E_{\Lambda}$.  Each of these classes 
contains a unique primitive $\OO$-ideal $I$; moreover this $I$ is proper and real \cite[Lemma 3.6a)]{BCS17}.  \\ \indent Conversely, let 
$I$ be a primitive proper real $\OO$-ideal.  Then $E_{I}$ is a real $\OO$-CM elliptic curve. As in $\S2.2$, we view $\Q(\ff)$ as an abstract 
number field with no privileged complex embedding.  Then there is a unique field embedding $\iota: \Q(\ff) \ra \C$ 
such that $\iota(j) = j(E_{I})$, and we view $\Q(\ff)$ as a subfield of $\R$ via this embedding.  Let $(E_0)_{/\Q(\ff)}$ be any 
elliptic curve with $j$-invariant $j$.  Then the $\R$-structure on $(E_0)_{/\R}$ need not be the same as that of $E_{I}$, but we claim that there is some twist $E_1$ of $E_0$ over $\Q(\ff)$ such that $(E_1)_{/\R} \cong E_{I}$.   To see this, 
first suppose that $\Delta < -4$.  Then $(E_0)_{/\R}$ and $(E_{I})_{/\R}$ are quadratic twists of each other, and 
since the square classes in $\R$ are represented by $1$ and $-1$, we may take $E_1$ to be either $E_0$ or the quadratic twist of $(E_0)_{/\Q(\ff)}$ by $-1$.  Similar arguments work in the remaining two cases: when $\Delta = -4$ (resp. when $\Delta = -3$) then $(E_0)_{/\R}$ and $(E_{I})_{/\R}$ are quartic (resp. sextic) twists of each other, and $-1$ represents the unique nontrivial element of $\R^{\times}/\R^{\times 4}$ (resp. of $\R^{\times}/\R^{\times 6}$).  More concretely, when $\Delta = -4$ the elliptic curves 
\[ (E_{\pm})_{/\Q}:  y^2 = x^3 \pm x \]
give the two different $\R$-structures, and when $\Delta = -3$ the elliptic curves 
\[ (E_{\pm})_{/\Q}: y^2 = x^3 \pm 1 \]
give the two different $\R$-structures.  
\\ \\
Now let $E_{/\Q(\ff)}$ be an $\OO$-CM elliptic curve.  We wish to determine when the $2$-torsion field $\Q(\ff)(E[2])$ contains the 
CM field $K$.  When $\Delta = -3$, we have $\Q(\ff) = \Q$ and $E$ is given by $y^2 = x^3 + B$ for some $B \in \Q^{\times}$, and since the splitting field of 
$x^3+B$ contains $\Q(\zeta_3) = K$, necessarily we do have $\Q(\ff)(E[2]) \supset K$.  (This is asserted in the proof of 
\cite[Cor. 4.3]{BCS17}.  The explanation given there is not correct, but as we have seen the assertion is true.)  When $\Delta = -4$, we have $\Q(\ff) = \Q$ and $E$ is given 
by $y^2 = x^3 + Ax$ for some $A \in \Q^{\times}$, so $\Q(\ff)(E[2]) \supset K = \Q(\sqrt{-1})$ iff $A \in \Q^{\times 2}$.  Henceforth we assume $\Delta < -4$, in which case $\Q(\ff)(E[2]) = \Q(\ff)(\mathfrak{h}(E[2]))$, so the answer is independent of the chosen 
$\Q(\ff)$-model.  In this case, \cite[Thm. 4.2]{BCS17}a) asserts that $\Q(\ff)(E[2]) \supset K$ iff $\Delta$ is odd.  Once again the 
assertion is correct but the explanation requires some modification.  As noted in \emph{loc. cit.} it follows from \cite[Lemma 3.15]{BCS17} that if $\Delta$ is odd, then $\Q(\ff)(E[2]) \supset K$.  (In fact this applies even when $\Delta = -3$ and gives another explanation 
of that case -- this is probably what the authors of \cite{BCS17} had in mind.)  Conversely, suppose that $\Delta$ is even.  Since $\OO$ 
is a primitive, proper real $\OO$-ideal, as explained above there is an embedding $\iota: \Q(\ff) \hookrightarrow \C$ and an 
$\OO$-CM elliptic curve $E_{/\Q(\ff)}$ such that $E_{/\R} \cong (E_{\OO})_{/\R}$.  Since $\Delta$ is even we have $\OO = \Z 1 \oplus 
\Z \frac{ \sqrt{\Delta}}{2}$.  From this description it is clear that complex conjugation acts trivially on the $2$-torsion and thus 
that $\iota(\Q(\ff)(E[2])) \subset \R$.  Since $\Q(\ff)(E[2])$ has a real embedding, it cannot contain the imaginary quadratic field $K$.

\subsection{Isogenies of CM elliptic curves.} 

Let $\varphi: E \ra E'$ be an isogeny of $K$-CM elliptic curves over $\C$.  Let $\OO = \OO(\ff) = \End E$, $\OO' = \OO(\ff') = \End E'$, and suppose that $\ff' \mid \ff$.  We may represent $E$ as $\C/\aa$ for a proper fractional $\OO$-ideal $\aa$ and then $E' = \C/\Lambda$ for a proper fractional $\OO'$-ideal $\Lambda$ that contains $\aa$.  Thus also $\aa \OO' \subset \Lambda$.  Putting $E'' = \C/\aa \OO'$,
the isogeny $\varphi$ factors as
\[ E \stackrel{\iota_{\ff,\ff'}}{\ra} E'' \stackrel{\varphi'}{\ra} E'. \]
Recall that for an imaginary quadratic order $\OO$, a fractional $\OO$-ideal $\aa$ is proper iff it is projective \cite[Lemma 3.1]{BCS17}.
Thus $\aa$ is a projective $\OO$-module, so $\aa \OO' = \aa \otimes_{\OO} \OO'$
is a projective $\OO'$-module, hence
$\aa \OO'$ is a proper fractional $\OO'$-ideal, i.e., $\End E'' = \OO'$.  This shows that $\iota_{\ff,\ff'}$ is universal for isogenies
from an $\OO(\ff)$-CM elliptic curve to an $\OO(\ff')$-CM elliptic curve.  Moreover we have
\[ \Ker \iota_{\ff,\ff'} = \aa \OO' / \aa \cong \Z/ \frac{\ff}{\ff'} \Z. \]
(This is easy when $\aa = \OO$, but by \cite[Lemma 3.1]{BCS17}, for all primes $p$ we have $\aa \otimes \Z_p \cong \OO
\otimes \Z_p$, so the general case reduces to this.)  This shows that there is only one possible kernel for a $\frac{\ff}{\ff'}$-isogeny from an $\OO$-CM elliptic curve to an $\OO'$-CM elliptic curve, and thus $\iota_{\ff,\ff'}$ must be $\Q(\ff)$-rational.\footnote{This is a well known
result.  For other treatments, see \cite[\S 2]{Kwon99} and \cite[Prop. 2.2]{BP17}.}
\\ \\
Now suppose that $\varphi: E \ra E'$ is an isogeny of $K$-CM elliptic curves over $\C$ with $\End E = \End E' = \OO = \OO(\ff)$.  We may represent $\varphi$ as $\C/\Lambda \ra \C/\Lambda'$ where $\Lambda' \supset \Lambda$ are proper (thus invertible) fractional $\OO$-ideals.
Taking $\bb = (\Lambda')^{-1} \Lambda$, we have $\Lambda' = \Lambda \bb^{-1}$, so
\[ \Ker \varphi = \Lambda \bb^{-1} / \Lambda = E[\bb] \eqqcolon \{P \in E(\C) \mid \forall b \in \bb, bP = 0\}. \]
It follows by \cite[\S 3.3]{BCS17} that $\varphi$ is defined over a field $F \supset \Q(\ff)$ iff $F \supset K$ or $\bb$ is real.

\subsection{Classification of $K(\ff)$-rational cyclic isogenies}

\begin{thm}
\label{KFISOGTHM}
Let $\OO$ be an imaginary quadratic order of discriminant $\Delta$, and suppose that $\Delta < -4$.    Let $N \in \Z^+$.  The
following are equivalent:  
\begin{enumerate}
\item There is an $\OO$-CM elliptic curve $E_{/K(\ff)}$ with a $K(\ff)$-rational cyclic $N$-isogeny $\varphi: E \ra E'$.  
\item There is a point $P \in \OO/N\OO$ of order $N$ with $C_{N}(\OO)$-orbit of size $\varphi(N)$.  
\item We have that $\Delta$ is a square in $\Z/4N\Z$.
\end{enumerate}
\end{thm}
\begin{proof} Combine \cite[Thm. 6.18a)]{BC18} with \cite[Thm. 6.15]{BC18}.
\end{proof}

\section{The dual isogeny}
\noindent
Let $\iota: E \ra E'$ be the canonical cyclic $\ff$-isogeny to an $\OO_K$-CM elliptic curve, and let $\iota^{\vee}$ be the dual isogeny.
There is an embedding $K(\ff) \hookrightarrow \C$ such that $E_{/\C} \cong \C/\OO$, $E'_{/\C} \cong \C/\OO_K$,
$\iota$ is the quotient map and $\iota^{\vee}$ is $P + \OO_K \mapsto \ff P + \OO$.
\\ \\
In this section, we will compute $\iota^{\vee}(T')$ for any finite $\OO_K$-submodule $T' \subset E'[\tors]$.
\\ \\
The following result will be proved in the next section.

\begin{thm}
\label{KKTHM}
Put $\KK \coloneqq \Ker \iota^{\vee}: E' \ra E$ and $c \coloneqq \ord_{\ell}(\ff)$.  
\begin{enumerate}
\item Suppose $\left( \frac{\Delta_K}{\ell} \right) = 1$, and let $\pp_1,\pp_2$ be the two primes\footnote{The statement is not
symmetric in $\pp_1$ and $\pp_2$, but it still holds after interchanging $\pp_1$ and $\pp_2$.} of $\OO_K$ lying over $\ell$.
For $0 \leq a \leq b$ we have
\[ \KK \cap E'[\pp_1^a \pp_2^b] \cong_{\Z} \Z/\ell^{\min(a,c)} \Z \]
and
\[ \iota^{\vee} E'[\pp_1^a \pp_2^b] \cong_{\Z} \Z/\ell^{\max(a-c,0)} \Z \times \Z/\ell^b \Z. \]
\item Suppose $\left( \frac{\Delta_K}{\ell} \right) = 0$, and let $\pp$ be the prime of $\OO_K$ lying over $\ell$.  For all $d \in \Z^+$ we have
\[ \KK \cap E'[\pp^{d}] \cong_{\Z} \Z/\ell^{\min(c,\lfloor \frac{d}{2} \rfloor)} \Z\]
and
\[ \iota^{\vee} E'[\pp^d] \cong_{\Z} \Z/\ell^{\max(\lfloor \frac{d}{2} \rfloor-c,0)} \Z \times \Z/\ell^{\lceil \frac{d}{2} \rceil}\Z. \]
\item Suppose $\left( \frac{\Delta_K}{\ell} \right) = -1$.  For all $b \in \Z^+$ we have
\[ \KK \cap E[\ell^b] \cong_{\Z} \Z/\ell^{\min(b,c)} \Z \]
and
\[ \iota^{\vee} E'[\ell^b] \cong_{\Z} \Z/\ell^{\max(b-c,0)} \Z \times \Z/\ell^b \Z. \]
\end{enumerate}
\end{thm}
\noindent
From Theorem \ref{KKTHM} we deduce the following result.

\begin{thm}
\label{DUALEXPONENTTHM}
Let $F \supset K(\ff)$ be a number field, and let $\iota: E \ra E'$ be the canonical $K(\ff)$-rational cyclic $\ff$-isogeny to
an $\OO_K$-CM elliptic curve $E'$.  Write
\[ E(F)[\tors] \cong \Z/s \Z \times \Z/e\Z, \ E'(F)[\tors] \cong \Z/s'\Z \times \Z/e' \Z \]
with $s \mid e$ and $s' \mid e'$.  Then $e' \mid e$.
\end{thm}
\begin{proof}
It is enough to show that for all primes $\ell$ we have $\exp E'(F)[\ell^{\infty}] \mid \exp E(F)[\ell^{\infty}]$.  Since
$E'(F)[\ell^{\infty}]$ is a finite $\ell$-primary $\OO_K$-submodule of $E'[\tors]$, it is isomorphic to $E[I]$ for an ideal $I$ of
$\OO_K$ that is divisible only by prime ideals lying over $\ell$.  Recalling that when $\ell$ ramifies in $\OO_K$ we have
\[ E'[\pp^d] \cong_{\Z} \Z/\ell^{\lceil \frac{d}{2} \rceil} \Z \times \Z/\ell^{\lfloor \frac{d}{2} \rfloor} \Z, \]
we see that Theorem \ref{KKTHM} implies that in all cases
\[ \exp E'(F)[\ell^{\infty}] = \exp E'[I] \mid \exp \iota^{\vee} E'[I] \mid \exp E(F)[\ell^{\infty}]. \qedhere\]
\end{proof}

\begin{remark}
In \cite[Lemma 6.7]{BC18} it was shown that in the setting of Theorem \ref{DUALEXPONENTTHM} we have $s \mid s'$.  Theorem \ref{DUALEXPONENTTHM} is in some sense the ``dual divisibility.''
\end{remark}

\subsection{Inert case} We suppose that $\left(\frac{\Delta_K}{\ell} \right) = -1$. Put $c \coloneqq \ord_{\ell}(\ff)$.  In this case the only finite $\OO_K$-modules of $E'[\ell^{\infty}]$ are $E'[\ell^b]$ for $b \in \Z^+$.  This is an easy case: if $\mathcal{K} = \Ker \iota^{\vee}$
then $\mathcal{K} \cap E[\ell^b]$ is cyclic of order $\ell^{\min(b,c)}$  and $\iota^{\vee}(E[\ell^b])$ is the subgroup of
$\C/\OO$ generated by $\ell^{c} \cdot \frac{1}{\ell^b}$ and $\ell^{\cc} \cdot \frac{\tau_K}{\ell^b}$, and thus
\[ \iota^{\vee} E'[\ell^b] \cong \Z/\ell^{\max(b-c,0)} \Z \times \Z/\ell^b \Z. \]

\subsection{Split case} We suppose that $\left( \frac{\Delta_K}{\ell} \right) = 1$ and $\ell \mid \ff$. Let $\pp_1,\pp_2$
be the two primes of $\OO_K$ lying over $\ell$.  We have
\[ E'[\pp_1^a \pp_2^b] \cong E'[\pp_1^a] \oplus E'[\pp_2^b] \cong_{\Z} \Z/\ell^a \Z \times \Z/\ell^b \Z. \]
We have $\KK = \ker \iota^{\vee}$ is cyclic of order $\ff$, generated by $\frac{1}{\ff} + \OO_K$.  If $P \in \KK \cap E'[\pp_1^a \pp_2^b]$ has order $d$, then $d \mid \ell^b$ and $d \mid \ff$, so $d \mid \ell^c$.  Also the $\OO_K$-submodule $\langle \langle P \rangle \rangle$ generated by $P$ is $E'[d]$; since the largest $d$ such that $E'[d] \subset [\pp_1^a \pp_2^b]$ is $\ell^a$,
we get $d \mid \ell^a$ and thus $d \mid \ell^{\min(a,c)}$.  On the other hand, the unique cyclic subgroup of $\KK$ of
order $\ell^{\min(a,c)}$ is also contained in $E'[\pp_1^a \pp_2^b]$, so
\[ \KK \cap E'[\pp_1^a \pp_2^b] \cong_{\Z} \Z/\ell^{\min(a,c)} \Z. \]
It follows that $\KK \cap E'[\pp_2^b]$ is the trivial group, so \[\iota^{\vee}(E'[\pp_2^b]) \cong_{\Z} E'[\pp_2^b]
\cong_{\Z} \Z/\ell^b \Z, \]
and it follows that $\exp \iota^{\vee} E[\pp_1^a \pp_2^b] = \ell^b$.  Since
\[\# \iota^{\vee} E[\pp_1^a \pp_2^b] = \frac{ \# E[\pp_1^a \pp_2^b]}{\ell^{\min(a,c)}} = \ell^{a+b-\min(a,c)}, \]
it follows that
\[ \iota^{\vee} E[\pp_1^a \pp_2^b] \cong \Z/\ell^{\max(a-c,0)} \Z \times \Z/\ell^b \Z. \]

\subsection{Ramified case} We suppose that $\left(\frac{\Delta_K}{\ell} \right) = 0$ and $\ell \mid \ff$.
\\ \\
Step 1: Let $\underline{a} = \min(a,c)$.  Let $\pp$ be the unique prime of $\OO_K$ lying over $\ell$. Let $\iota: E \ra E'$ be the canonical $K(\ff)$-rational cyclic $\ff$-isogeny to an $\OO_K$-CM elliptic curve $E'_{/K(\ff)}$.  Let $\mathcal{K} = \Ker \iota^{\vee}$.  For $d \in \Z^+$,
we compute $E[\pp^d] \cap K$.  This intersection is generated by $\frac{x}{\ell^c}$ for some $x \in \Z$, while
$E[\pp^d] \cong \pp^{-d}/\OO_K$.  We have
\[\frac{x}{\ell^c} \in \pp^{-d} \iff x \in \ell^c \pp^{-d} = \pp^{2c-d}. \]
$\bullet$ If $d$ is even, then
\[ x \in (\ell)^{c-\frac{d}{2}}, \]
so $\ord_{\ell}(x) \geq c - \frac{d}{2}$.  If $c - \frac{d}{2} \leq 0$, this condition is vacuous and $E'[\pp^d] \cap \mathcal{K}$ is
generated by $\frac{1}{\ell^c}$ and thus is cyclic of order $\ell^{c}$.  If $c - \frac{d}{2} \geq 0$ then $E[\pp^d] \cap \mathcal{K}$
is cyclic of order $\ell^{d/2}$.  Compiling the two cases, we get that $E[\pp^d] \cap \mathcal{K}$ is cyclic of order $\ell^{\min(c,d/2)}$.  \\
$\bullet$ If $d$ is odd, then since $\ord_{\pp}(x)$ is even, the above gives $\ord_{\pp}(x) \geq 2c - (d-1)$, and running through the
argument as above, we get that $E[\pp^d] \cap \KK$ is cyclic of order $\ell^{\min(c,\frac{d-1}{2})}$. \\
$\bullet$ Thus in general we have
\[ \# E'[\pp^d] \cap \KK = \ell^{\min(c,\lfloor \frac{d}{2} \rfloor)}. \]
By \cite[\S7.3]{BC18}, we
have
\[ E'[\pp^{2b-1}] \cong \Z/\ell^b \Z \times \Z/\ell^{b-1} \Z, \ E'[\pp^{2b}] = E'[\ell^b] \cong \Z/\ell^b\Z \times \Z/\ell^b \Z. \]
We claim that in the first case we can take $\frac{1}{\ell^{b-1}}$ as a generator for the second invariant factor, and the
second case we can take $\frac{1}{\ell^b}$ as a generator for the second invariant factor.  In the latter case this is
clear: any element of order $\ell^b$ can be taken as a generator for the second invariant factor.  In the former case, the elements that can be taken as a generator of the second invariant factor are precisely those elements $x$ of order $\ell^{b-1}$ that generate a
\textbf{pure} subgroup of $E'[\pp^{2b-1}]$: for all $m,i \in \Z^+$, if there is $y \in E'[\pp^{2b-1}]$ such that
$mx = \ell^i y$, then $mx = \ell^i nx$ for some $n \in \Z$.  (For in a commutative group $G$ with $G = G[\ell^b]$ for some $b \in \Z^+$, every pure
subgroup is a direct summand \cite[\S4.3]{Robinson}.)  We apply this with $x = \frac{1}{\ell^{b-1}}$: if $mx = \ell^i y$ for $y \in \pp^{-(2b-1)}$,
then $m \ell^{-b-i+1} \in \pp^{-(2b+1)}$, so
\[m \in \ell^{b+i-1} \pp^{-(2b-1)} = \pp^{2b+2i-2} \pp^{-2b+1} = \pp^{2i-1}, \]
so $\ord_{\pp}(m) \geq 2i-1$.  Since $m \in \Z$, $\ord_{\pp}(m)$ is even, so $\ord_{\pp}(m) \geq 2i$ and thus $\ell^i \mid m$ and
we may take $n = \frac{m}{\ell^i}$.  It follows that
\[ \exp \iota^{\vee}(E[\pp^{2b-1}]) = \exp \iota^{\vee}(E[\pp^{2b}]) = \ell^b \]
and thus
\[ \exp \iota^{\vee}(E[\pp^d]) = \ell^{\lceil \frac{d}{2} \rceil}. \]
Since
\[ \# \iota^{\vee} E[\pp^d] = \frac{\# E[\pp^d]}{\ell^{\min(c,\lfloor \frac{d}{2} \rfloor)}} = \ell^{d-\min(c,\lfloor \frac{d}{2} \rfloor)}, \]
it follows that
\[ \iota^{\vee} E'[\pp^d] \cong_{\Z} \Z/\ell^{\max(\lfloor \frac{d}{2} \rfloor-c,0)} \Z \times \Z/\ell^{\lceil \frac{d}{2} \rceil}\Z. \]

\section{Degrees of CM points on $X(M,N)_{/K(\zeta_M)}$}

\subsection{Statement of the theorem}
Let $\OO$ be an order of conductor $\ff$ in the imaginary quadratic field $K$, and let $w = \# \OO^{\times}$ and $w_K = \# \OO_K^{\times}$.
For positive integers $M \mid N$, we define an
\textbf{(M,N)-pair} to be a pair $(P,Q)$ with $P \in \OO/N\OO$ of order $M$ and $Q \in \OO/N\OO$ of order $N$ such that the
$\Z$-module generated by $P$ and $Q$ is isomorphic to $\Z/M\Z \times \Z/N\Z$.  We denote by $\widetilde{T}(\OO,M,N)$ the least
size of a $C_N(\OO)$-orbit on the set of $(M,N)$-pairs. 

\begin{thm}
\label{BIGMNTHM}
Let $\OO$ be an imaginary quadratic order of conductor $\ff$, and let $M = \ell_1^{a_1} \cdots \ell_r^{a_r} \mid N = \ell_1^{b_1} \cdots \ell_r^{b_r}$ be positive integers.  
\begin{enumerate}
\item There is $T(\OO,M,N) \in \Z^+$ such that: for all $d \in \Z^+$, there is a number field $F \supset K(\ff)$ such that
$[F:K(\ff)] = d$ and an $\OO$-CM elliptic curve $E_{/F}$ such that $\Z/M\Z \times \Z/N \Z \hookrightarrow E(F)$ iff $T(\OO,M,N) \mid d$.  
\item 
If $N=2$ or 3, then $T(\OO,M,N)$ is as follows:
\begin{align*}
T(\OO, 1,2)&=\begin{cases} 3& \left( \frac{\Delta}{2} \right)=-1 \text{ and } \Delta \neq -3 \\ 1 & \text{otherwise} \end{cases},\\
T(\OO,1,3) &= \begin{cases} 8/w& \left( \frac{\Delta}{3} \right)=-1  \\ 1 & \text{otherwise} \end{cases},\\
 T(\OO,2,2) &= \frac{2(2-\left( \frac{\Delta}{2} \right))}{w},\\
 T(\OO,3,3) &= \frac{2(3-\left( \frac{\Delta}{3} \right))}{w}.
  \end{align*}

\item If $N=2$, then $\widetilde{T}(\OO,M,N)$ is as follows:
\begin{align*}
\widetilde{T}(\OO, 1,2)&=\begin{cases} 1& \left( \frac{\Delta}{2} \right)\neq-1  \\ 3 & \left( \frac{\Delta}{2} \right)=-1 \end{cases},\\
\widetilde{T}(\OO,2,2) &= 2-\left( \frac{\Delta}{2} \right).
  \end{align*}
\item Suppose $N \geq 3$ and $r = 1$, and write $M = \ell^a$, $N=\ell^b$ for $0 \leq a \leq b$. Put $c \coloneqq \ord_{\ell}(\ff)$. Then: 
\begin{enumerate}
\item[i)] If $\left( \frac{\Delta}{\ell} \right) = -1$, then
\[ \widetilde{T}(\OO,\ell^a,\ell^b) =\ell^{2b-2}(\ell^2-1). \]
\item[ii)] If $\left( \frac{\Delta}{\ell} \right) = 1$, then
\[ \widetilde{T}(\OO,\ell^a,\ell^b) = \begin{cases} \ell^{b-1}(\ell-1) & a = 0 \\ \ell^{a+b-2} (\ell-1)^2 & a \geq 1 \end{cases}. \]
\item[iii)] If $\ell \mid \ff $ and $\left( \frac{\Delta_K}{\ell} \right) = 1$, then
\[ \widetilde{T}(\OO,\ell^a,\ell^b) = \ell^{a+b-1}(\ell-1). \]
\item[iv)] If $\left( \frac{\Delta_K}{\ell} \right) = 0$, then
\[ \widetilde{T}(\OO,\ell^a,\ell^b) = \begin{cases} \ell^{a+b-1}(\ell-1) & b \leq 2c+1 \\
\ell^{\max(a+b-1,2b-2c-2)}(\ell-1) & b > 2c+1 \end{cases}. \]
\item[v)] If $\ell \mid \ff$ and $\left( \frac{\Delta_K}{\ell} \right) = -1$, then
\[ \widetilde{T}(\OO,\ell^a,\ell^b) = \begin{cases} \ell^{a+b-1}(\ell-1) & b \leq 2c \\
\ell^{\max(a+b-1,2b-2c-1)}(\ell-1) & b > 2c \end{cases}. \]
\end{enumerate}
\item Suppose $N \geq 4$.  Then we have
\[ T(\OO,M,N) = \frac{\prod_{i=1}^r \widetilde{T}(\OO,\ell_i^{a_i},\ell_i^{b_i})}{w}. \]
\end{enumerate}
\end{thm}

\subsection{Reducing to the case of prime powers.}
Let $\OO$ be an imaginary quadratic order, and let $M \mid N$ be positive integers.  Recall that an $(M,N)$-pair 
is a pair $(P,Q)$ with $P \in \OO/N\OO$ of order $M$ and $Q \in \OO/N\OO$ of order $N$ such that the
$\Z$-module generated by $P$ and $Q$ is isomorphic to $\Z/M\Z \times \Z/N\Z$. The Cartan subgroup $C_N(\OO) = (\OO/N\OO)^{\times}$ has a natural action on $(M,N)$ pairs:
\[ g \cdot (P,Q) = (gP,gQ). \]
We denote by $\widetilde{T}(\OO,M,N)$ the least
size of a $C_N(\OO)$-orbit on the set of $(M,N)$-pairs.

A \textbf{reduced (M,N)-pair} is an orbit of an $(M,N)$-pair $(P,Q)$ under the image of $\OO^{\times}$ in $\OO/N\OO$.  Unless
$\Delta = -4,-3$, this simply identifies $(P,Q)$ with $(-P,-Q)$.  The reduced Cartan subgroup $\overline{C_N(\OO)} = C_N(\OO)/q_N(\OO^{\times})$ has a natural action on the set of reduced $(M,N)$-pairs.  We denote by $T(\OO,M,N)$ the least size of a $\overline{C_N(\OO)}$-orbit on the set of reduced $(M,N)$-pairs.

\begin{lemma}
\label{COMPILE1}
Let $N \in \Z^+$.  The group $C_N(\OO)$ acts freely on the set of $(N,N)$-pairs, and thus every $C_N(\OO)$-orbit has size 
$\# C_N(\OO)$.
\end{lemma}
\begin{proof}
If $gP = P$ and $gQ = Q$ then $g \in (\OO/N\OO)^{\times}$ fixes all of $\OO/N\OO$ and thus $g = 1$.
\end{proof}
\noindent
The following result explains the relevance reduced $(M,N)$-pairs to the problem at hand.

\begin{prop}
\label{FIRSTCOMPILEPROP}
For an imaginary quadratic order $\OO$, positive integers $M \mid N$, and a positive integer $d$, the following are equivalent: 
\begin{enumerate}
\item There is a field $F \supset K(\ff)$ with $[F:K(\ff)] = d$, an $\OO$-CM elliptic curve $E_{/F}$ and an injection
\[ \Z/M\Z \times \Z/N\Z \hookrightarrow E(F)[\tors]. \]
\item The integer $d$ is divisible by the size of some $\overline{C_N(\OO)}$-orbit on the set of reduced $(M,N)$-pairs.
\end{enumerate}
\end{prop}
\begin{proof}
This is a direct consequence of the surjectivity of the reduced Galois representation
\[ \overline{\rho_N}: \gg_{K(\ff)} \ra \overline{C_N(\OO)}. \]
Indeed, it follows that for an $(M,N)$-pair $(P,Q) \in E(\overline{K(\ff)})$, the size of the $\overline{C_N(\OO)}$-orbit on the reduced
pair is the unique minimal degree of a field extension $F/K(\ff)$ over which there is a single character $\chi: \gg_F \ra \OO^{\times}$ such that for all $\sigma \in \gg_F$ we have $\sigma(P) = \chi(\sigma) P$ and $\sigma(Q) =
\chi(\sigma) Q$, and thus $(P,Q)$ both become rational on some twist of $E_{/F}$.
\end{proof}
\noindent
Thus we wish to find, for all $\OO$ and $M \mid N$, the set of all multiples of sizes of $\overline{C_N(\OO)}$-orbits on
reduced $(M,N)$-pairs.  We will show that every $\overline{C_N(\OO)}$-orbit has size a multiple of $T(\OO,M,N)$, with the consequence that condition (b) in Proposition \ref{FIRSTCOMPILEPROP} will simplify to: $d$ is a multiple of $T(\OO,M,N)$.
\\ \\
Although Proposition \ref{FIRSTCOMPILEPROP} gives us the answer in terms of orbits on reduced $(M,N)$-pairs, it is more natural
to work with orbits on $(M,N)$-pairs, especially with regard to the process of compiling across prime powers.  The following
result allows us to pass back and forth between them.

\begin{lemma}
\label{LEMMA4.10}
Let $N \geq 4$ and let $M \mid N$.  Let $(P,Q)$ be an $(M,N)$-pair in $\OO/N\OO$, and let $\overline{(P,Q)}$ be the corresponding
reduced $(M,N)$-pair. 
\begin{enumerate}
\item The size of the $C_N(\OO)$-orbit on $(P,Q)$ is $w$ times the size of the $\overline{C_N(\OO)}$-orbit
on $\overline{(P,Q)}$. 
\item We have $\widetilde{T}(\OO,M,N) = w T(\OO,M,N)$.
\item Every $C_N(\OO)$-orbit on an $(M,N)$-pair has size a multiple of $\widetilde{T}(\OO,M,N)$ iff every
$\overline{C_N(\OO)}$-orbit on a reduced $(M,N)$-pair has size a multiple of $T(\OO,M,N)$.
\end{enumerate}
\end{lemma}
\begin{proof}
(a) The assertion is equivalent to: $\OO^{\times}$ acts freely on $(M,N)$-pairs. When $M  = 1$, this follows from \cite[Lemma 7.6]{BC18}.  The general case follows. (b), (c) These follow immediately.
\end{proof}

\begin{prop}
\label{PROP4.11}
Let $M \mid N$ be positive integers.  Write
\[ M = \ell_1^{a_1} \cdots \ell_r^{a_r}, \ N = \ell_1^{b_1} \cdots \ell_r^{b_r}. \]
Let $(P,Q)$ be an $(M,N)$-pair in $\OO/N\OO$.  For $1 \leq i \leq r$, let $P_i$ (resp. $Q_i$) be the image of $P$ (resp. $Q$) in $\OO/\ell_i^{b_i} \OO$ .  Then: 
\begin{enumerate}
\item For all $1 \leq i \leq r$, we have that $(P_i,Q_i)$ is an $(\ell_i^{a_i},\ell_i^{b_i})$-pair in $\OO/\ell_i^{b_i} \OO$.  
\item The $C_N(\OO)$-orbit on the pair $(P,Q)$ is isomorphic as a $C_N(\OO)$-set to the direct product of
the $C_{\ell_i^{b_i}}(\OO)$-orbits on the pairs $(P_i,Q_i)$.
\item Suppose $N \geq 4$.  Then we have
\[ T(\OO,M,N) = \frac{\prod_{i=1}^r \widetilde{T}(\OO,\ell_i^{a_i},\ell_i^{b_i})}{w}. \]
\end{enumerate}
\end{prop}
\begin{proof}
Parts (a) and (b) are an immediate consequence of the canonical decompositions arising from the Chinese Remainder Theorem:
\[ C_N(\OO) = \prod_{i=1}^r C_{\ell_i^{b_i}}(\OO), \ \OO/N\OO = \prod_{i=1}^r \OO/\ell_i^{b_i}\OO.\]
Thus $\widetilde{T}(\OO,M,N)=\prod_{i=1}^r \widetilde{T}(\OO,\ell_i^{a_i},\ell_i^{b_i})$, and so (c) follows Lemma \ref{LEMMA4.10}(b).
\end{proof}

\begin{cor}
Suppose the $C_{\ell^b}(\OO)$-orbit on an $(\ell^a,\ell^b)$-pair has size a multiple of $\widetilde{T}(\OO,\ell^a,\ell^b)$ for all primes $\ell$ and all integers $0 \leq a \leq b$.  Then for any positive integers $M \mid N$ the following are equivalent:
\begin{enumerate}
\item There is a field $F \supset K(\ff)$ with $[F:K(\ff)] = d$, an $\OO$-CM elliptic curve $E_{/F}$ and an injection
\[ \Z/M\Z \times \Z/N\Z \hookrightarrow E(F)[\tors]. \]
\item The integer $d$ is a multiple of $T(\OO,M,N)$.
\end{enumerate}
\end{cor}

\begin{proof}
If $N \geq 4$, then this follows from Proposition \ref{FIRSTCOMPILEPROP}, Lemma \ref{LEMMA4.10}, and Proposition \ref{PROP4.11}. If $M=1$ and $N=2$ or 3, the statement follows from \cite[Thm. 7.2(a)]{BC18} with $T(\OO,1,N)=T(\OO,N)$. If $M=N=2$ or if $M=N=3$, then the statement follows from \cite[Thm. 1.1, Cor. 1.4]{BC18} with $ T(\OO,N,N)=\#\overline{C_N(\OO)}$.
\end{proof}

\subsection{Prime power case}
 For a prime $\ell$ and natural numbers $0 \leq a \leq b$, we will compute $\widetilde{T}(\OO, \ell^a,\ell^b)$ and show every $C_{\ell^b}(\OO)$-orbit on an $(\ell^a,\ell^b)$-pair has size a multiple of $\widetilde{T}(\OO,\ell^a,\ell^b)$.  When $a = b$, by Lemma \ref{COMPILE1} every $C_{\ell^b}(\OO)$-orbit has size $\# C_{\ell^b}(\OO)$, accomplishing both goals in this case, so we may assume 
$a < b$.  
\\ \\
If $\ell^b <4$, then the values of $T(\OO, \ell^a,\ell^b),$ $\widetilde{T}(\OO, \ell^a,\ell^b)$ follow from Lemma \ref{COMPILE1} and \cite[Thm. 1.1, Cor. 1.4, Thm. 7.2, Rmk. 7.3]{BC18}. The following result handles the divisibility claim.
 
 \begin{lemma} 
 \begin{enumerate}
 \item[]
\item Every $C_2(\OO)$-orbit on a $(1,2)$-pair has size a multiple of $\widetilde{T}(\OO,1,2)$.  
\item Every $C_3(\OO)$-orbit on a $(1,3)$-pair has size a multiple of $\widetilde{T}(\OO,1,3)$.
\end{enumerate}
\end{lemma}

 \begin{proof} a) We have
  \[
\widetilde{T}(\OO,1,2)=\widetilde{T}(\OO,2)=\begin{cases} 1& \left( \frac{\Delta}{2} \right)\neq-1,  \\ 3 & \left( \frac{\Delta}{2} \right)=-1. \end{cases}
  \] 
The divisibility claim holds trivially if $\left( \frac{\Delta}{2} \right)\neq-1$, so suppose $\left( \frac{\Delta}{2} \right)=-1$. Then 
$\# C_2(\OO) = 3 = \# (\OO/2\OO \setminus \{0\})$, so the $(1,2)$-pairs form a single $C_2(\OO)$-orbit.  \\
b) This follows from \cite[Thm. 7.8 and $\S7.4$]{BC18}. 
  \end{proof}
\noindent
Henceforth we assume $\ell^b \geq 4$. The following preliminary result strengthens \cite[Lemma 2.3]{BCP17}.

\begin{prop}
\label{1.1}
Let $\ell^b \geq 4$ be a prime power, and let $E$ be an $\OO$-CM elliptic curve defined over a number field $F \supset K$. Suppose
\[ \Z/\ell^a \Z \times \Z/\ell^b \Z \hookrightarrow E(F). \]
\begin{enumerate}
\item If $\left( \frac{\Delta}{\ell} \right) = -1$, then $a = b$ and $\ell^{2b-2} (\ell^2-1) \mid w [F:K(\ff)]$.  
\item If $\left( \frac{\Delta}{\ell}\right) = 1$ and $a = 0$, then $\ell^{b-1} (\ell-1) \mid w [F:K(\ff)]$.  
\item If $\left( \frac{\Delta}{\ell}\right) = 1$ and $a \geq 1$, then $\ell^{a+b-2}(\ell-1)^2 \mid w [F:K(\ff)]$. 
\item If $\left( \frac{\Delta}{\ell}\right) = 0$, then $\ell^{a+b-1}(\ell-1) \mid w [F:K(\ff)]$.
\end{enumerate}
\end{prop}
\begin{proof}
By the results of \cite[Thm. 1.1]{BC18}, we have $F(E[\ell^b]) \supset K^{(\ell^b)}K(\ff \ell^b)$, so
\[ \# \overline{C_{\ell^b}(\OO)} =  [K^{(\ell^b)}K(\ff \ell^b):K(\ff)] \mid
[F(E[\ell^b]):K(\ff)] =[F(E[\ell^b]):F][F:K(\ff)]. \]
(a) If $\left(\frac{\Delta}{\ell} \right) = -1$, then $C_{\ell^b}(\OO)$ acts transitively on points of order $\ell^b$ \cite[Lemma 6.12(b)]{BC18}, so
$a = b$ and $F = F(E[\ell^b])$. Thus
\[  \frac{\ell^{2b-2}(\ell^2-1)}{w} = \# \overline{C_{\ell^b}(\OO)} \mid [F:K(\ff)]. \]
(b) In this case, by \cite[Lemma 2.2]{BCP17}, we have $[F(E[\ell^b]):F] \mid \ell^{b-1} (\ell-1)$, so
\[ \frac{\ell^{2b-2} (\ell-1)^2}{w} = \# \overline{C_{\ell^b}(\OO)} \mid [F(E[\ell^b]):F][F:K(\ff)] \mid
\ell^{b-1} (\ell-1) [F:K(\ff)], \]
and the result follows. \\
(c),(d) In both of these cases, by \cite[Lemma 2.2]{BCP17} we have $[F(E[\ell^b]):F] \mid \ell^{b-a}$, so
\[ \frac{ \# \overline{C_{\ell^b}(\OO)}}{\ell^{b-a}} \mid [F:K(\ff)]. \]
Since
\[ \# \overline{C_{\ell^b}(\OO)} = \begin{cases} \frac{\ell^{2b-2} (\ell-1)^2}{w} & \left(\frac{\Delta}{\ell} \right) = 1 \\
\frac{\ell^{2b-1} (\ell-1)}{w} & \left( \frac{\Delta}{\ell} \right) = 0 \end{cases} \] the result follows.
\end{proof}

\begin{thm}
\label{1.3}
Let $\ell^b \geq 4$ be a prime power.   Suppose $\left( \frac{\Delta}{\ell} \right) = -1$.   
\begin{enumerate}
\item We have
\[
 w[K^{(\ell^b)}K(\ff \ell^b):K(\ff)] = \ell^{2b-2}(\ell^2-1). \]
\item  There is an $\OO$-CM elliptic curve
$E_{/K^{(\ell^b)}K(\ff \ell^b)}$ such that
\[ E(K^{(\ell^b)}K(\ff \ell^b))[\ell^{\infty}] \cong \Z/\ell^b \Z \times \Z/\ell^b \Z. \]
\item We have $\widetilde{T}(\OO, \ell^a, \ell^b)=\ell^{2b-2}(\ell^2-1)$.  Moreover every $C_{\ell^b}(\OO)$-orbit on an $(\ell^a,\ell^b)$-pair has size a multiple of $\widetilde{T}(\OO,\ell^a,\ell^b).$
\end{enumerate}
\end{thm}
\begin{proof}
By \cite[Thm. 1.1]{BC18} we have
\[ [K^{(\ell^b)}K(\ff \ell^b):K(\ff)] = \# \overline{C_{\ell^b}(\OO)} = \frac{\# C_{\ell^b}(\OO)}{w} = \frac{\ell^{2b-2}(\ell^2-1)}{w}, \]
establishing (a).  By \cite[Cor. 1.4]{BC18} there is an $\OO$-CM elliptic curve $E_{/K^{(\ell^b)}K(\ff \ell^b)}$
with $\Z/\ell^b \Z \times \Z/\ell^b \Z \hookrightarrow E(K^{(\ell^b)}K(\ff \ell^b))$.  If the $\ell$-primary torsion subgroup were any larger, then because $\left( \frac{\Delta}{\ell} \right) = -1$ there would be full $\ell^{b+1}$-torsion, which would imply that $K^{(\ell^b)}K(\ff \ell^b) \supset K^{(\ell^{b+1})} K(\ff \ell^{b+1})$.  But since
$[K^{(\ell^b)}K(\ff \ell^b):K(\ff)] = \frac{\ell^{2b-2} (\ell^2-1)}{w}$ is an increasing function of $b$, this is absurd. 
Combining (b) with Proposition \ref{1.1} gives that $T(\OO,\ell^a, \ell^b)=\frac{ \ell^{2b-2}(\ell^2-1)}{w}$ and this quantity divides the size of every $\overline{C_{\ell^b}(\OO)}$-orbit on a reduced $(\ell^a,\ell^b)$-pair. So (c) follows from Lemma \ref{LEMMA4.10}.
\end{proof}

\begin{thm}
\label{THM1.4}
Suppose $\left( \frac{\Delta}{\ell} \right) = 1$.  Let $\pp_1, \pp_2$ be the two primes of $\OO_K$ lying over
$\ell$.  Let $a,b \in \N$ with $a \leq b$ and $\ell^b \geq 4$.  
\begin{enumerate}
\item There is a number field $F \supset K(\ff)$ with \[w[F:K(\ff)] =  \begin{cases} \ell^{b-1}(\ell-1)& a = 0 \\
\ell^{a+b-2}(\ell-1)^2 & a \geq 1 \end{cases}\] and an $\OO$-CM elliptic curve $E_{/F}$ such that 
\[
E(F)[\ell^{\infty}] \cong  \begin{cases} \Z/\ell^a \Z \times \Z/\ell^b \Z& \ell \neq 2 \\
\Z/\ell^{\max\{1,a\}} \Z \times \Z/\ell^b \Z & \ell=2 \end{cases} .
\]  
\item If $\Delta_K \notin \{-4,-3\}$ or if $\ff = 1$, we may take $F = K^{\pp_1^a \pp_2^b}K(\ff)$. 
\item If $\ff \neq 1$ and $\Delta_K \in \{-4,-3\}$, we may take $F$ to be an extension
of $K^{\pp_1^a \pp_2^b}K(\ff)$ of degree $\frac{w_K}{2}$.
\item Every $C_{\ell^b}(\OO)$-orbit on an $(\ell^a,\ell^b)$-pair has size a multiple of 
\[
\widetilde{T}(\OO,\ell^a,\ell^b)=\begin{cases} \ell^{b-1}(\ell-1)& a = 0, \\
\ell^{a+b-2}(\ell-1)^2 & a \geq 1. \end{cases}
\]
\end{enumerate}
\end{thm}

\begin{proof}
First suppose $\ff = 1$, i.e., $\OO = \OO_K$.  By the results of \cite[\S 7.3]{BC18}, there is an $\OO_K$-CM elliptic curve
$E_{/K^{\pp_1^a \pp_2^b}}$ with $\Z/\ell^a\Z \times \Z/\ell^b \Z \hookrightarrow E(K^{\pp_1^a \pp_2^b})$.   By \cite[Lemma 2.10]{BC18} we have
\begin{equation}
\label{THM1.4EQ1}
w_K [K^{\pp_1^a\pp_2^b}:K^{(1)}] = \begin{cases} \ell^{b-1}(\ell-1) & a = 0 \\
\ell^{a+b-2}(\ell-1)^2 & a \geq 1 \end{cases}.
\end{equation}
By Proposition \ref{1.1} we must have $E(K^{\pp_1^a \pp_2^b})[\ell^{\infty}] \cong \Z/\ell^a \Z \times \Z/\ell^b \Z$ if $\ell>2$ or if $\ell=2$ and $a\geq 1$. If $\ell=2$, then $K(\ff)=K(2\ff)$ (see $\S \ref{coincidences}$). In this case, $K\neq\Q(i), \Q(\sqrt{-3})$, so $K(2\ff)$ is the projective 2-torsion point field of an $\OO$-CM elliptic curve \cite[Thm. 4.1]{BC18}, and $E$ has full $2$-torsion over any extension containing a rational point of order 2. This completes
the proof in this case.  \\ \indent
Next suppose $\ff > 1$ and $\Delta_K \notin \{-4,-3\}$.   Let $E_{/K(\ff)}$ be an $\OO$-CM elliptic curve, and let $\iota: E \ra E'$ be the canonical $K(\ff)$-rational cyclic $\ff$-isogeny to an $\OO_K$-CM elliptic curve $E'$. Then \[T' \coloneqq E'[\pp_1^a \pp_2^b] \cong \Z/\ell^a \Z \times \Z/\ell^b \Z, \] and
$K^{(1)}(\hh(E'[\pp_1^a \pp_2^b])) = K^{\pp_1^a \pp_2^b}$. Let $F=K^{\pp_1^a \pp_2^b}K(\ff)$.  Since $w_K= 2$ and $T'$ is cyclic as an $\Ok$-module \cite[Lemma 2.4]{BC18},
it follows that for all $\sigma \in \gg_F$ there is $\epsilon(\sigma) \in\{ \pm 1\}$ such that for all $P' \in T'$ we have $\sigma P' = \epsilon(\sigma) P'$.  Let $\iota^{\vee}_{/K(\ff)}: E' \ra E$
be the dual isogeny, also cyclic of order $\ff$.  Put $T \coloneqq \iota^{\vee}(T')$.  Since $\gcd(\ff,\ell) =1$,
the map $\iota^{\vee}: T' \ra T$ is an isomorphism of $\gg_F$-modules, and thus for all $P \in T$, we have
$\sigma P = \epsilon(\sigma) P$.  Thus $\epsilon$ is a (possibly trivial) quadratic character on $\gg_F$,
and twisting $E$ by $\epsilon$ we get an elliptic curve $E_{/F}$ such that
\[\Z/\ell^a \Z \times \Z/\ell^b \Z \cong T \subset E(F). \]
Since $w_K = 2$, (\ref{THM1.4EQ1}) implies
\[[F:K(\ff)] \leq \begin{cases} \frac{\ell^{b-1}(\ell-1)}{2} & a = 0 \\
\frac{\ell^{a+b-2}(\ell-1)^2}{2} & a \geq 1 \end{cases}. \]  By Proposition \ref{1.1} we must have equality.  Similarly, we must have
$E(F)[\ell^{\infty}] \cong \Z/\ell^a \Z \times \Z/\ell^b \Z$ if $\ell>2$ or if $\ell=2$ and $a \geq1$, because if the $\ell$-primary torsion subgroup
were any larger, it would contradict Proposition \ref{1.1}.  As above, if $\ell=2$, then $K(\ff)=K(2\ff)$, and $E$ has full $2$-torsion over any extension containing a rational point of order 2. \\ \indent
Finally suppose $\ff > 1$ and $\Delta_K \in \{-4,-3\}$: thus $w_K$ is $4$ or $6$ while $w = 2$.   Then over the field $F_0 \coloneqq K^{\pp_1^a \pp_2^b} K(\ff)$ the action of $\gg_{F_0}$
on $T'$ is now by a character with values in $\OO_K^{\times}$.  There is thus a field extension $F/F_0$
of degree $\frac{w_K}{2}$ over which the action of $\gg_F$ on $T'$ is given by a quadratic character,
and the argument proceeds as above with this choice of $F$.  Notice in particular that $[F_0:K(\ff)]$ is smaller
than $[F:K(\ff)]$ in the previous case by a factor of $\frac{w_K}{2}$; since $[F:F_0] = \frac{w_K}{2}$, these
factors cancel out and $[F:K(\ff)]$ is unchanged.

Thus by Proposition \ref{1.1} we have
\[
T(\OO,\ell^a, \ell^b)=\begin{cases} \frac{\ell^{b-1}(\ell-1)}{w}& a = 0 \\
\frac{\ell^{a+b-2}(\ell-1)^2}{w} & a \geq 1 \end{cases}
\] and this quantity divides the size of every $\overline{C_{\ell^b}(\OO)}$-orbit on a reduced $(\ell^a,\ell^b)$-pair. So (d) follows from Lemma \ref{LEMMA4.10}.
\end{proof}

\begin{thm}
\label{THM1.5}
Suppose $\ell \mid \ff$ and $\left( \frac{\Delta_K}{\ell} \right) = 1$, and let $\pp_1,\pp_2$ be the two
primes of $\OO_K$ lying over $\ell$.  Let $a,b \in \N$ with $0 \leq a \leq b$
and $\ell^b \geq 4$. 
\begin{enumerate}
\item There is a number field $F \supset K(\ff)$ with
\[  [F:K(\ff)] = \frac{\ell^{a+b-1} (\ell-1)}{2} \]
and an $\OO$-CM elliptic curve $E_{/F}$ such that $E(F)[\ell^{\infty}] \cong \Z/\ell^a \Z \times \Z/\ell^b \Z$.  
\item We may take $F$ to be an extension of $K^{\pp_2^b}K(\ell^a \ff)$ of degree $\frac{w_K}{2}$.
\item Every $C_{\ell^b}(\OO)$-orbit on an $(\ell^a,\ell^b)$-pair has size a multiple of 
\[
\widetilde{T}(\OO,\ell^a,\ell^b)=\ell^{a+b-1}(\ell-1). \]
\end{enumerate}
\end{thm}
\begin{proof}
Let $\iota: E \ra E'$ be the canonical $K(\ff)$-rational cyclic $\ff$-isogeny to an $\OO_K$-CM elliptic curve
$E'$.  We put $T' \coloneqq E'[\pp_1^a \pp_2^b]$ and $T \coloneqq \iota^{\vee}(T')$.  Let us first assume
that $K \neq \Q(\sqrt{-1}),\Q(\sqrt{-3})$, so $w_K = 2$. \\
Case 1: Suppose $a = 0$.  Let $F=K^{\pp_2^b}K(\ff).$ In this case, by Theorem \ref{KKTHM} the map $\iota^{\vee}: T' \ra T$ is an isomorphism of $\gg_{F}$-modules, so after twisting $E$ by a unique
quadratic character $\epsilon$ we get an elliptic curve $E_{/F}$ with $T \cong \Z/\ell^b \Z \hookrightarrow
E(F)$. Combining with Proposition \ref{1.1} we get $w[F:K(\ff)] = \ell^{b-1}(\ell-1)$ and $E(F)[\ell^{\infty}]
\cong \Z/\ell^b \Z$, completing the proof in this case.  \\
Case 2: Suppose $a \geq 1$, and let $F_0 \coloneqq K(\ff)K^{\pp_2^b}$.  By Case 1 there is an
$\OO$-CM elliptic curve $E_{/F_0}$ with an $F_0$-rational point of order $\ell^b$.  By \cite[Thm. 4.1]{BC18},
after base extension to $K^{\pp_2^b}K(\ell^a \ff)$ the mod $\ell^a$ Galois representation is given by
scalar matrices, but since we also have a rational point of order $\ell^b$, the mod $\ell^a$ Galois
representation is trivial.  Since $[K(\ell^a \ff):K(\ff)] = \ell^a$ we must have
\[ 2[K^{\pp_2^b}K(\ell^a \ff):K(\ff)] \leq \ell^{a+b-1}(\ell-1). \]
By Proposition \ref{1.1} we must have
\[ [K^{\pp_2^b}K(\ell^a \ff):K(\ff)] = \frac{\ell^{a+b-1}(\ell-1)}{2}. \]  Taking $F = K^{\pp_2^b}K(\ell^a \ff)$ completes the proof in this case. \\ \indent
If $K = \Q(\sqrt{-1})$ or $\Q(\sqrt{-3})$ then we modify the above argument as in the proof of Theorem
\ref{THM1.4}: namely we make an extension of $K^{\pp_2^b}$ of degree $\frac{w_K}{2}$ so as to
ensure that that Galois action on $T'$ is given by a \emph{quadratic} character.  Once again, the degree comes out the same.

Thus $T(\OO,\ell^a, \ell^b)=\frac{\ell^{a+b-1}(\ell-1)}{2},$ and this quantity divides the size of every $\overline{C_{\ell^b}(\OO)}$-orbit on a reduced $(\ell^a,\ell^b)$-pair by Proposition \ref{1.1}. So (c) follows from Lemma \ref{LEMMA4.10}.
\end{proof}

\begin{thm}
\label{THM1.6}
Suppose $\left( \frac{\Delta_K}{\ell} \right) = 0$.  Let $c \coloneqq \ord_{\ell}(\ff)$.  Let $a,b \in \N$ with $a \leq b$ and $\ell^b \geq 4$.  
\begin{enumerate}
\item Suppose $b \leq 2c+1$.  Then there is a number field $F \supset K(\ff)$ with
\[ [F:K(\ff)] = \frac{\ell^{a+b-1}(\ell-1)}{w} \]
and an $\OO$-CM elliptic curve $E_{/F}$ such that $E(F)[\ell^{\infty}] \cong \Z/\ell^a \Z \times \Z/\ell^b \Z$.   
\item Suppose $b > 2c+1$.  Then: 
\begin{enumerate}
\item[i)] If for a number field $F \supset K(\ff)$ we have $E(F)[\ell^{\infty}] \cong \Z/\ell^a\Z \times \Z/\ell^b\Z$, then
$a \geq b-2c-1$.
\item[ii)] If $a \geq b-2c-1$, there is a number field $F \supset K(\ff)$ with $[F:K(\ff)] = \frac{\ell^{a+b-1}(\ell-1)}{w}$ and an $\OO$-CM elliptic curve
$E_{/F}$ with $E(F)[\ell^{\infty}] = \Z/\ell^a \Z \times \Z/\ell^b \Z$.
\end{enumerate}
\item Every $C_{\ell^b}(\OO)$-orbit on an $(\ell^a,\ell^b)$-pair has size a multiple of 
\[
\widetilde{T}(\OO,\ell^a,\ell^b)=\begin{cases} \ell^{a+b-1}(\ell-1)& b \leq 2c+1, \\
\ell^{\max(a+b-1,2b-2c-2)}(\ell-1) & b>2c+1. \end{cases}
\]
\end{enumerate}
\end{thm}
\begin{proof}
Let $\pp$ be the unique prime of $\OO_K$ lying over $\ell$.  \\ \indent
First suppose that $\ff = 1$, so $c = 0$.  In this case, for any number
field $F \supset K^{(1)}$ and any $\OO_K$-CM elliptic curve $E_{/F}$, the subgroup $E(F)[\ell^{\infty}]$ is an $\OO_K$-submodule
of $E(F)$ and thus is isomorphic to $E[\pp^d]$ for some $d \in \N$ (see the proof of Theorem 7.8 in \cite{BC18}).  By the First Main Theorem of Complex Multiplication, we have
\[ K^{(1)}(\mathfrak{h}(E[\pp^d])) = K^{\pp^d}, \]
and since $E[\pp^d]$ is cyclic there is an $\OO_K$-CM elliptic curve $E_{/K^{\pp^d}}$ with $E[\pp^d] \subset E(K^{\pp^d})$, so
\[ \frac{\ell^{d-1}(\ell-1)}{w} = [K^{\pp^d}:K^{(1)}].\]
 Moreover, by the proof of Theorem 7.8 in \cite{BC18}, we have
\[ E[\pp^d] \cong  \Z/\ell^{\lfloor \frac{d}{2} \rfloor} \Z \times \Z/\ell^{\lceil \frac{d}{2} \rceil}  \Z, \]
which implies that if $a \leq b$ are the natural numbers such that $E(F)[\ell^{\infty}] \cong \Z/\ell^a \Z \times \Z/\ell^b \Z$,
then either $a = b-1$ or $a = b$.  If $a = b-1$ then $d = 2a+1$, while if $a = b$ then $d = 2a$.  Either way we have
\[ \frac{\ell^{a+b-1}(\ell-1)}{w} = \frac{\ell^{d-1}(\ell-1)}{w} = [K^{\pp^d}:K^{(1)}], \]
completing the result in this case.  Henceforth we suppose that $\ff > 1$, so $w = 2$.  \\
(a) Suppose $b \leq 2c+1$.  Then by \cite[Thm. 7.2]{BC18}, there is an extension $F_0/K(\ff)$ of degree $\ell^{b-1}(\ell-1)/2$ and $\OO$-CM elliptic curve $E_{/F_0}$ with an $F_0$-rational point of order $\ell^b$. By \cite[Thm. 7]{CP15} and its proof,
there is an extension $F/F_0$ with $[F:F_0] \mid \ell^a$ such that $E_{/F}$ has full $\ell^a$-torsion.  Thus
\[ \Z/\ell^a \Z \times \Z/\ell^b \Z \hookrightarrow E(F)[\ell^{\infty}] \]
and by Proposition \ref{1.1} we must have equality.  \\
(b) Suppose $b > 2c+1$, let $F \supset K$, and suppose $E(F)[\ell^{\infty}] \cong \Z/\ell^a \Z \times \Z/\ell^b \Z$.  Let $\iota: E \ra E'$ be the usual canonical isogeny to an $\OO_K$-CM elliptic curve.  If $E(F)$ has a point of order $\ell^b$
then $E'(F)$ has a point of order $\ell^{b-c}$ and thus there is a subgroup $T' \subset E'(F)$ with $T' \cong \Z/\ell^{b-c-1}\Z \times \Z/\ell^{b-c} \Z$.  Let $T = \iota^{\vee}(T')$, and write $T \cong \Z/\ell^A\Z \times \Z/\ell^B \Z$ with $0 \leq A \leq B$.  Since $\#T \geq \frac{\# T'}{\ell^c}$ and $B \leq b-c$, we must have $A \geq b-c-1-c = b-2c-1$, so $a \geq A \geq b-2c-1$.  Now Proposition
\ref{1.1} implies
\[ \ell^{2b-2c-2}(\ell-1) \mid 2 [F:K(\ff)]. \]
By \cite[Thm. 7.2]{BC18}, there is a field extension $F_0/K(\ff)$ of degree $\frac{\ell^{2b-2c-2}(\ell-1)}{2}$ and an $\OO$-CM elliptic curve $E_{/F_0}$ with an $F_0$-rational point of order $\ell^b$. As we have shown, this forces $\Z/\ell^{b-2c-1}\Z \times \Z/\ell^b \Z \hookrightarrow
E(F_0)[\ell^{\infty}]$, and then Proposition \ref{1.1} implies
\[ [F_0:K(\ff)] = \frac{\ell^{2b-2c-2}(\ell-1)}{2} \]
and
\[ E(F_0)[\ell^{\infty}] \cong \Z/\ell^{2b-2c-1} \Z \times \Z/\ell^b \Z. \]
Finally, suppose $b-2c-1 \leq a \leq b$.  By \cite[Thm. 7]{CP15} there is a field extension $F/F_0$ of degree dividing
$\ell^{a-b+2c+1}$ over which $E$ has full $\ell^a$-torsion.  Thus
\[ [F:K(\ff)] \mid \frac{\ell^{a-b+2c+1+2b-2c-2}(\ell-1)}{2} = \frac{\ell^{a+b-1}(\ell-1)}{2} \]
and
\[ \Z/\ell^a \Z \times \Z/\ell^b \Z \hookrightarrow E(F)[\ell^{\infty}]. \]
Once again, by Proposition \ref{1.1} the degree divisibility and the group inclusion are each equalities.
Statement (c) now follows from Lemma \ref{LEMMA4.10}.
\end{proof}

\begin{thm}
\label{THM1.7}
Suppose $\ell \mid \ff$ and $\left( \frac{\Delta_K}{\ell} \right) = -1$. Let $c \coloneqq \ord_{\ell}(\ff)$. Let $a,b \in \N$ with $a \leq b$ and $\ell^b \geq 4$. 
\begin{enumerate}
\item Suppose $b \leq 2c$.  Then there is a number field $F \supset K(\ff)$ with
\[ [F:K(\ff)] = \frac{\ell^{a+b-1}(\ell-1)}{2} \]
and an $\OO$-CM elliptic curve $E_{/F}$ such that $E(F)[\ell^{\infty}] \cong \Z/\ell^a \Z \times \Z/\ell^b \Z$.  
\item Suppose $b > 2c$.  Then: 
\begin{enumerate}
\item[i)] If for a number field $F \supset K(\ff)$ we have $E(F) \cong\Z/\ell^a \Z \times \Z/\ell^b \Z$, then $a \geq b-2c$. 
\item[ii)] If $a \geq b-2c$, there is a number field $F \supset K(\ff)$ with $[F:K(\ff)] = \frac{\ell^{a+b-1}(\ell-1)}{2}$ and an
$\OO$-CM elliptic curve $E_{/F}$ with $E(F)[\ell^{\infty}] = \Z/\ell^a \Z \times \Z/\ell^b \Z$.
\end{enumerate}
\item Every $C_{\ell^b}(\OO)$-orbit on an $(\ell^a,\ell^b)$-pair has size a multiple of 
\[
\widetilde{T}(\OO,\ell^a,\ell^b)=\begin{cases} \ell^{a+b-1}(\ell-1)& b \leq 2c, \\
\ell^{\max(a+b-1,2b-2c-1)}(\ell-1) & b>2c. \end{cases}
\]
\end{enumerate}
\end{thm}
\begin{proof}
(a) Suppose that $b \leq 2c$.  Then by \cite[Thm. 7.2]{BC18}, there is an extension $F_0/K(\ff)$ of degree $\ell^{b-1}(\ell-1)/w$ and an $\OO$-CM elliptic curve $E_{/F_0}$ with an $F_0$-rational point of order $\ell^b$. The remainder of the proof of part (a) is identical to the proof of part (a) of Theorem \ref{THM1.6}. \\
(b) Suppose $b > 2c$, let $F \supset K$, and suppose $E(F)[\ell^{\infty}] \cong \Z/\ell^a \Z \times \Z/\ell^b \Z$.
Let $\iota: E \ra E'$ be the usual canonical isogeny to an $\OO_K$-CM elliptic curve.  If $E(F)$ has a point of order $\ell^b$
then $E'(F)$ has a point of order $\ell^{b-c}$ and thus -- since $\left( \frac{\Delta_K}{\ell} \right) = -1$ -- we have
$E'[\ell^{b-c}] \subset E'(F)$.  As in the proof of part (b) of Theorem \ref{THM1.6}, it follows that $E[\ell^{b-2c}] \subset \iota^{\vee}(E'(F))$, so $a \geq b-2c$.  Now Proposition \ref{1.1} implies
\[ \ell^{2b-2c-1}(\ell-1) \mid 2[F:K(\ff)]. \]
The argument now proceeds exactly as in the proof of part (b) of Theorem \ref{THM1.6}.

As above, statement (c) follows from Lemma \ref{LEMMA4.10}.
\end{proof}

\section{A generalization of Kwon's theorem}

\subsection{Kwon's theorem}
For an order $\OO = \OO(\ff)$ in an imaginary quadratic field $K$ and a positive integer $N$, we use the shorthand $I(\OO,N)$ to mean there is a $\Q(\ff)$-rational cyclic $N$-isogeny $\varphi: E \ra E'$ between elliptic curves with $\End E = \OO$. Recall $\Delta=\ff^2 \Delta_K$.

\begin{thm}(Kwon \cite[Cor. 4.2]{Kwon99})
\label{KWONTHM}
Let $\OO = \OO(\ff)$ be an order in an imaginary quadratic field $K \neq \Q(\sqrt{-1}),\Q(\sqrt{-3})$.  Let $N \in \Z^+$.
\begin{enumerate}
\item Suppose one of the following holds:
\begin{enumerate}
\item[i)] $2 \mid \ff$ and $2$ is ramified in $K$ or
\item[ii)] $4 \mid \ff$.
\end{enumerate} Then $I(\OO,N)$ holds iff $N \mid \frac{\Delta}{4}$.
\item  Suppose one of the following holds:
\begin{enumerate}
\item[i)] $\ff \equiv 2 \pmod{4}$ and $2$ is unramified in $K$ or
\item[ii)] $\ff$ is odd and $\left(\frac{\Delta_K}{2} \right) \neq -1$.
\end{enumerate}  Then $I(\OO,N)$ holds iff either $N$ or $\frac{N}{2}$ is an odd integer dividing $\Delta$.
\item Suppose $2 \nmid \ff$ and $\left( \frac{\Delta_K}{2} \right) = -1$.  Then $I(\OO,N)$ holds iff $N \mid \Delta$.
\end{enumerate}
\end{thm}

\begin{remark}\label{KWONRMK}
Let $K = \Q(\sqrt{-1})$ or $\Q(\sqrt{-3})$.  Suppose that $\ff > 1$.  It follows from Theorem \ref{GENKWONTHM} and Proposition \ref{C1PROP} that if
$I(\OO,N)$ holds then $N$ must satisfy the same numerical conditions as in Theorem \ref{KWONTHM}.  The arguments of
\cite[pp. 954-955]{Kwon99} hold verbatim to show that if $N$ satisfies these numerical conditions then $I(\OO,N)$ holds.  Thus
in fact Theorem \ref{KWONTHM} holds verbatim for every imaginary quadratic order of discriminant $\Delta < -4$.  \\ \indent
In Corollary \ref{LASTKWONCOR} we will complete Kwon's theorem by determining all $N \in \Z^+$ such that $I(\OO,N)$ holds when $\Delta = -4$ or $\Delta = -3$.
\end{remark}

\subsection{Statement of the generalization of Kwon's theorem} Half of Theorem \ref{KWONTHM} gives necessary conditions on $N$ for the existence
of a $\Q(\ff)$-rational cyclic $N$-isogeny on an $\OO(\ff)$-CM elliptic curve.  The following result extends this half of the result
by showing that the conclusions continue to hold for $F$-rational cyclic $N$-isogenies for a certain class of number fields $F$ that 
contain $\Q(\ff)$.

\begin{thm}
\label{GENKWONTHM}
Let $N, \ff \in \Z^+$.  Let $F/\Q(\ff)$ be a number field.  We suppose that there is an $\OO(\ff)$-$\operatorname{CM}$ elliptic curve $E_{/F}$ admitting an $F$-rational cyclic $N$-isogeny.  We also suppose $F$ does not contain $K$, and for all primes $\ell \mid N$, $F$ does not contain $\Q(\ell \ff)$.
\begin{enumerate}
\item There is a positive integer $d \mid \gcd(\ff,N)$ and a
primitive, proper, real $\OO(\frac{\ff}{d})$-ideal of index $\frac{N}{d}$.  
\item It follows that $N \mid \Delta = \ff^2 \Delta_K$.  Moreover: 
\begin{enumerate}
\item[i)] Suppose that $16 \mid \Delta$.  Then $N \mid \frac{\Delta}{4}$.   
\item[ii)] Suppose $\ff \equiv 2 \pmod{4}$ and $2$ is unramified in $K$.
Then either $N$ or $\frac{N}{2}$ is an odd divisor
of $\Delta$. 
\item[iii)] Suppose $2 \nmid \ff$ and $2$ is ramified in $K$.  Then either $N$ or $\frac{N}{2}$ is an odd divisor of $\Delta$.
\end{enumerate}
\end{enumerate}
\end{thm}

\begin{remark}
The hypotheses on $F$ are natural.   On the one hand, \cite[Thm. 6.18]{BC18} gives the classification of $N \in \Z^+$ such that
some $\OO(\ff)$-$\operatorname{CM}$ elliptic curve admits a $K(\ff)$-rational cyclic $N$-isogeny.  On the other: suppose that $F$ is a number field
containing $\Q(N \ff)$.  Then there is an $\OO(N\ff)$-$\operatorname{CM}$ elliptic curve
$\tilde{E}_{/F}$ and a canonical $F$-rational cyclic $N$-isogeny $\iota_{N\ff,\ff}: \tilde{E} \ra E$ with $\End E = \OO(\ff)$, and thus $\iota_{N\ff,\ff}^{\vee}: E \ra \tilde{E}$ is a cyclic $N$-isogeny.
\end{remark}

\subsection{A preliminary lemma}

\begin{lemma}
\label{ISOGENYLEMMA1}
Let $N \in \Z^+$.  
\begin{enumerate}
\item Let $\varphi: E \ra E'$ be a degree $N$ isogeny of $K$-CM elliptic curves, with $\End E=\OO(\ff)$ and $\End E'=\OO(\ff')$. Then
\[ \ff \mid N \ff', \ \ff' \mid N \ff. \]
\item Let $F$ be a number field that does not contain $\Q(\ell \ff)$ for any $\ell \mid N$, let $E_{/F}$ be an $\OO(\ff)$-CM elliptic curve, and let $\varphi: E \ra E'$ be an $F$-rational $N$-isogeny.
Then the conductor $\ff'$ of $\End E'$ divides $\ff$.
\end{enumerate}
\end{lemma}
\begin{proof}
Every $N$-isogeny factors as a product of $\ell$-isogenies, so it is enough to treat the case $N = \ell$.   Over $\C$
we may view $\varphi$ as $\C/\Lambda \ra \C/\Lambda'$ where $\Lambda' \supset \Lambda$ and $\Lambda'/\Lambda \cong \Z/\ell \Z$.
Let $\alpha \in \OO(\ff)$.  For $\lambda' \in \Lambda'$ we have
\[ (\ell \alpha) \lambda' = \alpha (\ell \lambda'), \]
but $\ell \lambda' \in \Lambda$, so
\[ \alpha (\ell \lambda') \in \alpha \Lambda \subset \Lambda \subset \Lambda'. \]
Thus $\ell \alpha \in \{ x \in K \, | \, x \Lambda' \subset \Lambda'\} = \OO(\ff')$.   Applying this with $\alpha = \ff \tau_K$, we get that $\ff \ell \tau_K \in  \OO(\ff')$,
so $\ff' \mid \ell \ff$.  The same argument applied to the dual isogeny $\varphi^{\vee}: E' \ra E$ shows $\ff \mid \ell \ff'$, so
$\frac{\ff}{\ff'} \in \{\ell^{-1},1,\ell\}$.  \\
b) By part a), if $\ff'$ does not divide $\ff$, then there is some prime $\ell \mid N$ such that $\ord_{\ell}(\ff') > \ord_{\ell}(\ff)$; seeking a contradiction, we fix such a prime. We may factor $\varphi$ over $F$ as $\varphi_2 \circ \varphi_1$, where $\deg \varphi_1 = \ell^{\ord_{\ell}(N)}$ and $\deg \varphi_2 = \frac{N}{\ell^{\ord_{\ell}(N)}}$.  By part a), the $\ell$-primary part of the conductor is unchanged under $\varphi_2$, so if $\ff_1$ is the conductor of $E/\Ker \varphi_1$ then we must have $\ff \ell \mid \ff_1$ and thus
$F \supset \Q(\ff \ell)$, contradicting our hypothesis.
\end{proof}

\subsection{Classification of primitive, proper real ideals}

\begin{lemma}
\label{KWONLEMMA3.1}
Let $\OO$ be an imaginary quadratic order of discriminant $\Delta = \ff^2 \Delta_K$.  
\begin{enumerate}
\item If there is a primitive, proper real $\OO$-ideal of index $N$, then $N \mid \Delta$.  
\item Let $N = \ell_1^{a_1} \cdots \ell_r^{a_r}$.  There is a primitive, proper real $\OO$-ideal $I$ such that $[\OO:I] = N$ iff for all $1 \leq i \leq r$, there is a primitive,
proper real $\OO$-ideal $I_i$ such that $[\OO:I_i] = \ell_i^{a_i}$.  
\item Let $\ell > 2$, and let $a \in \Z^+$.  There is a primitive, proper real $\OO$-ideal $I$ such that $[\OO:I] = \ell^a$ iff
$a = \ord_{\ell} (\Delta)$.  
\item Let $\ell = 2$, and let $a \in \Z^+$. 
\begin{enumerate}
\item[i)] Suppose $16 \mid \Delta$.  Then there is a primitive, proper real $\OO$-ideal $I$ such that $[\OO:I] = 2^a$ iff
$a = 2$ or $a = \ord_2(\Delta)-2$.  
\item[ii)] Suppose $2 \mid \Delta$ and $16 \nmid \Delta$.  Then there is a primitive, proper real $\OO$-ideal $I$ such that $[\OO:I] =2^a$ iff
$a = 1$.
\end{enumerate}
\end{enumerate}
\end{lemma}
\begin{proof}
a) See \cite[\S 3]{Kwon99} or \cite[Cor. 3.8]{BCS17}. \\
b) Let $\ell_1  < \ldots < \ell_r$ be prime numbers and $a_1,\dots,a_r$ be positive integers.  For $1 \leq i \leq r$, let $I_i$ be a primitive 
proper real ideal of index $\ell_i^{a_i}$.  Then the $I_i$'s are pairwise comaximal, so by the Chinese Remainder Theorem we have
\[ I \coloneqq \bigcap_{i=1}^r I_i = \prod_{i=1}^r I_i\]
and $\OO/I \cong \prod_{i=1}^r \OO/I_i$, so $I$ has index $\ell_1^{a_1} \cdots \ell_r^{a_r}$.  Since each $I_i$ is primitive we have 
$\OO/I_i \cong_{\Z} \Z/\ell_i^{a_i} \Z$.  Thus $\OO/I \cong \prod_{i=1}^r \Z/\ell_i^{a_i} \Z \cong \Z/\prod_{i=1}^r \ell_i^{a_i} \Z$ and $I$ is primitive.  An ideal of $\OO$ is proper iff it is locally principal, and $I = \prod_{i=1}^r I_i$ is a product of locally principal 
ideals and thus locally principal, hence proper.  Finally, as an intersection of real ideals, $I$ is real.  \\ \indent
Let $I$ be a primitive, proper real ideal of $\OO$ of index $\ell_1^{a_1} \cdots \ell_r^{a_r}$.  Since $I$ is a nonzero ideal in a 
one-dimensional Noetherian domain, there are pairwise comaximal primary ideals $\qq_1,\ldots,\qq_s$ such that 
\[ I = \bigcap_{j=1}^s \qq_j = \prod_{j=1}^s \qq_j \]
and the $\qq_j$'s are unique up to ordering \cite[\S 10.6]{Clark-CA}.  For each $1 \leq j \leq s$ the ring $\OO/\qq_j$ finite local, hence of prime power 
order.  Since $\OO/I \cong \prod_{j=1}^s \OO/\qq_j$, the order of each $\OO/\qq_j$ is a power of some $\ell_i$, and for all $i$ 
there is at least one $\qq_j$ such that $[\OO:\qq_j]$ is a power of $\ell_i$.  For $1 \leq i \leq r$, let $S_i$ be the subset of 
$\{1,\ldots,s\}$ consisting of all $j$ such that $[\OO:\qq_j]$ is a power of $i$, and put 
\[ I_i \coloneqq \bigcap_{j \in S_i} \qq_j = \prod_{j \in S_i} \qq_j. \]
Then $I_i$ is the ``$\ell_i$-primary part of $I$'': that is, we have 
\[ \OO/I \cong \prod_{i=1}^r \OO/I_i, \ [\OO:I_i] = \ell_i^{a_i},\]
and it follows from the uniqueness of the primary decomposition that that $I_i$'s are the unique ideals containing $I$ with these 
properties.  \\ \indent
Since $I$ is primitive, $\OO/I \cong \prod_{i=1}^r \OO/I_i$ is a cyclic $\Z$-module, hence so is each $\OO/I_i$, hence each $I_i$ 
is primitive.  Since $I$ is real, we have 
\[ I = \overline{I} = \prod_{i=1}^r \overline{I_i}, \]
and by the uniqueness of the $I_i$'s we must have $\overline{I_i} = I_i$ for all $i$.  As above, to show that each $I_i$ is projective 
it is equivalent to show that each $\qq_j$ is locally principal.  For $1 \leq j \leq s$, let $\pp_j$ be the radical of $\qq_j$.  If $\pp \in \operatorname{MaxSpec} \OO \setminus \{ \pp_j \}$ we have $\pp_j \OO_{\pp} = \OO_{\pp}$, while 
\[ \pp_j \OO_{\pp_j} = \pp_1 \cdots \pp_s \OO_{\pp_j} = I \OO_{\pp_j}, \]
so $\pp_j$ is locally principal hence projective.
\\
c), d) This is \cite[Prop. 3.1]{Kwon99}.
\end{proof}

\subsection{Proof of Theorem \ref{GENKWONTHM}}
a) Let $F$ be a number field, let $E_{/F}$ be an $\OO = \OO(\ff)$-CM elliptic curve, and let $\varphi: E \ra E'$ be an $F$-rational cyclic $N$-isogeny.  Suppose $F$ contains neither $K$ nor $\Q(\ell \ff)$ for any $\ell \mid N$.   By Lemma \ref{ISOGENYLEMMA1},
the conductor $\ff'$ of $\OO' \coloneqq \End E'$ divides $\ff$.  As in $\S2.6$, we may factor $\varphi$ as
\[ E \stackrel{\iota_{\ff,\ff'}}{\ra} E'' \stackrel{\varphi'}{\ra} E'. \]
Thus $\varphi'$ is a cyclic $\frac{N}{\ff/\ff'}$-isogeny of $\OO'$-CM elliptic curves, so by $\S2.6$ we have $\Ker \varphi' = E''[\bb]$ for a primitive
(since $\varphi'$ is cyclic) proper $\OO'$-ideal $\bb$.   Let
\[ \mathcal{K} \coloneqq \Ker \varphi, \ \mathcal{K}'' \coloneqq \Ker \iota_{\ff,\ff'}, \ \mathcal{K}' \coloneqq \ker \varphi'. \]
Since $\mathcal{K}$ and $\mathcal{K}''$ are both $F$-rational group schemes and
\[ \mathcal{K'} = \mathcal{K}/\mathcal{K}'', \]
we have that $\mathcal{K}'$ is $F$-rational; equivalently, $\varphi'$ is defined over $F$.  Since $F$ does not contain $K$, the $F$-rationality of $E[\bb]$ implies that $\bb$
is a real ideal and $\# \OO'/\bb = \# E[\bb] = \frac{N}{\ff/\ff'}$: see \cite[\S 3.3]{BCS17} and \cite[Lemma 2.4]{BC18}.
\\
b) Suppose that there is an $\OO = \OO(\ff)$-CM elliptic curve $E_{/F}$ admitting an $F$-rational cyclic $N$-isogeny.  By part a) there is there is some $\ff' \mid \ff$ such that
$N = \frac{\ff}{\ff'} \cdot i$, where $i$ is the index of a primitive, proper real $\OO(\ff')$-ideal.   By Lemma \ref{KWONLEMMA3.1}(a)
we have $i \mid \Delta(\OO(\ff')) = \ff'^2 \Delta_K$, so
\begin{equation}
\label{GENKWONEQ1}
N = \frac{\ff}{\ff'} i \mid \frac{\ff}{\ff'} \ff'^2 \Delta_K  = \ff \ff' \Delta_K \mid \ff^2 \Delta_K.
\end{equation}
$\bullet$ Suppose that $2 \mid \ff$ and $16 \mid \Delta$.  We wish to show that $N \mid \frac{\ff^2 \Delta_K}{4}$.
In view of what we have already shown it is enough to show that if $N = 2^a$ then $a \leq \ord_2(\ff^2 \Delta_K) - 2$, and this
follows easily from Lemma \ref{KWONLEMMA3.1}(d). \\
$\bullet$ Suppose that $\ff \equiv 2 \pmod{4}$ and $2$ is unramified in $K$.  We wish to show that if $2^a \mid N$ then $a \leq 1$.  This follows easily from Lemma \ref{KWONLEMMA3.1}. \\
$\bullet$ Suppose that $2 \nmid \ff$ and $2$ ramifies in $K$.  We wish to show that if $2^a \mid N$ then $a \leq 1$.  This
follows easily from Lemma \ref{KWONLEMMA3.1}.

\subsection{Supplements to Kwon's theorem}\label{coincidences}
The hypotheses of Theorem \ref{GENKWONTHM} are never satisfied in the presence of ``ring class field coincidences'': namely, when
$\Q(\ff) = \Q(\ell \ff)$ for some prime $\ell \mid N$.  Using (\ref{RINGCLASSDEGEQ}), one
sees that the instances of $\Q(\ff) = \Q(\ell \ff)$ are precisely as follows: \\ \\
(C1) If $\left( \frac{\Delta_K}{2} \right) = 1$ and $\ff$ is odd, then $\Q(\ff) = \Q(2 \ff)$.  \\
(C2) If $K = \Q(\sqrt{-1})$, then $\Q(1) = \Q(2)$.  \\
(C3) If $K = \Q(\sqrt{-3})$, then $\Q(1) = \Q(2) = \Q(3)$.
\\
In particular we must have $\ell \in \{2,3\}$.
\\ \\
Concerning (C1):

\begin{lemma}
\label{LITTLEPROJPROP}
Let $\OO(\ff)$ be an imaginary quadratic order of discriminant $\Delta < -4$, and let $E_{/F}$ be an $\OO(\ff)$-CM elliptic curve.
For a positive integer $N$, if the image $\rho_N(\gg_F)$ of the mod $N$ Galois representation on $E$ consists only of
scalar matrices, then $\Q(N\ff) \subset F$.
\end{lemma}
\begin{proof}
This is really the observation that part of the proof of \cite[Thm. 4.1]{BC18} goes through over $\Q(\ff)$ rather than $K(\ff)$, but for
the convenience of the reader we will recap the argument.  \\ \indent
There is a field embedding $F \hookrightarrow \C$ such that $E_{/\C} \cong \C/\OO(\ff)$.  The map $z \mapsto Nz$ induces a
cyclic $N$-isogeny $\varphi: \C/\OO(\ff) \ra \C/\OO(N\ff)$.  Let $C = \Ker \varphi = \langle P \rangle$ for some point $P \in E(\overline{F})$ of order $N$.  Our assumption on $\gg_F$ implies that Galois acts by scaling $P$ and thus $C$ is $F$-rational.
So $E \ra E/C$ is an $F$-rational isogeny and $\End E/C = \OO(N\ff)$, so $\Q(N \ff) \subset F$.
\end{proof}

\begin{prop}
\label{C1PROP}
Let $K$ be an imaginary quadratic field in which $2$ splits.  Let $\ff$ be an odd integer, and put $\OO = \OO(\ff)$.  Let $F$ be a number field that does not contain $K$.  The following are equivalent: 
\begin{enumerate}
\setcounter{enumii}{4}
\item Every $\OO$-CM elliptic curve defined over $F$ admits an $F$-rational cyclic $4$-isogeny. 
\item There is an $\OO$-CM elliptic curve defined over $F$ admitting an $F$-rational cyclic $4$-isogeny.  
\item The field $F$ contains $\Q(4\ff)$.
\end{enumerate}
\end{prop}
\begin{proof}
The hypotheses prevent $K$ from being $\Q(\sqrt{-1})$ or $\Q(\sqrt{-3})$, so for any $\OO$-CM elliptic curve $E$ we have
$\Aut E = \{ \pm 1 \}$. Thus the existence of rational isogenies is independent of the chosen $F$-rational model: (a) $\iff$ (b).  
\\
(b) $\implies$ (c): Let $E_{/F}$ be an $\OO$-CM elliptic curve, and let $\varphi: E \ra E'$ be an $F$-rational cyclic $4$-isogeny.
Seeking a contradiction, we suppose that $F$ does not contain $\Q(4\ff)$.  

The isogeny $\varphi$ factors as $E \stackrel{\iota}{\ra} E'' \stackrel{\varphi'}{\ra} E'$, where $\iota$ and $\varphi'$ are both $F$-rational $2$-isogenies.  Thus there is $P \in E(F)$ of order $2$ such that -- up to an isomorphism on $E''$ -- if $C = \langle P \rangle$, then $\iota: E \ra E/C$.  We claim that $\End E/C \cong \OO(2 \ff)$, so that -- again, up to an isomorphism on $E''$ -- we have 
$\iota = \iota_{2\ff,\ff}^{\vee}$.  Indeed, there are three order $2$ subgroups of $E$.  Since $2$ splits in $K$ and $\ff$ is odd, 
we have $2 \OO = \pp_1 \pp_2$ with $\pp_2 = \overline{\pp_1} \neq \pp_1$.  Thus $E[\pp_1]$ and $E[\pp_2]$ are distinct 
subgroups of order $2$.  However, since $F$ does not contain $K$ and $\pp_1$ and $\pp_2$ are not real ideals, the subgroups 
$E[\pp_1]$ and $E[\pp_2]$ are not $F$-rational, whereas $C$ is, so this gives all three order $2$ subgroups of $E$.  Because $\pp_1$ 
and $\pp_2$ are proper $\OO$-ideals we have $\End E/E[\pp_1] \cong \End E/E[\pp_2] \cong \End E \cong \OO$.  On the other hand, 
there is $2$-isogeny $\psi$ from an $\OO(2\ff)$-CM elliptic curve $\tilde{E}$ to $E$: e.g. we may embed $\Q(j(E)) \hookrightarrow \C$ 
such that $E_{/\C} \cong C/\OO$ and then the isogeny is $\C/\OO(2\ff) \ra \C/\OO$ and thus $\Ker \psi^{\vee}$ is an order $2$ 
subgroup of $E$ such that $\End E/\Ker \psi^{\vee} \cong \OO(2\ff)$.  Therefore we must have $\Ker \psi^{\vee} = C$.

Now $\iota^{\vee}: E'' \ra E$ is also an $F$-rational $2$-isogeny, hence is of the form $\langle Q \rangle$ for a point $Q \in E''(F)$ 
of order $2$.  Lemma \ref{LITTLEPROJPROP} gives $F(E''[2]) \supset \Q(4\ff)$; since $F$ does not contain $\Q(4\ff)$ it follows 
that $Q$ is the only element of $E''(F)$ of order $2$.  It follows that -- up to an isomorphism on $E'$ -- we have $\varphi' = \iota^{\vee}$ -- but then $\Ker \varphi = \Ker \iota^{\vee} \circ \iota = E[2]$, so $\varphi$ is not cyclic, a contradiction.
\\
(c) $\implies$ (b): If $F$ contains $\Q(4\ff)$ then there is an $\OO(4\ff)$-CM elliptic curve $\tilde{E}_{/F}$ and a
cyclic $4$-isogeny $\iota_{4\ff,\ff}$ with target an $\OO(\ff)$-CM elliptic curve.
\end{proof}
\noindent
Concerning (C2):

\begin{prop}
\label{C2PROP}
Let $K = \Q(\sqrt{-1})$, and let $\OO = \OO_K$ be the ring of integers of $K$, so $\Delta = \Delta_K = -4$, $\ff= 1$ and $\Q(\ff) = \Q$.
For $a \in \Z^+$,
we have that $I(\OO,2^a)$ holds iff $a \leq 2$.
\end{prop}
\begin{proof}
Step 1: We show that there is an $\OO$-CM elliptic curve $E_{/\Q}$ admitting a $\Q$-rational cyclic $4$-isogeny.  Since $\varphi(4)/2 = 1$, by \cite[Thm. 5.5]{BCS17}, this holds iff there is an $\OO$-CM elliptic curve such that $E(\Q)$ has an element of order $4$.  That the latter holds is well known \cite[p. 196]{Olson74}.  However, we will give a ``principled'' proof.  Namely, let $\OO'$ be the order of
conductor $2$ in $K$.  Since $\Q(2) = \Q(1) = \Q$, there is an $\OO'$-CM elliptic curve $E'_{/\Q}$ and thus a $\Q$-rational $2$-isogeny
$\iota: E' \ra E$ where $\End E = \OO$.  Let $\pp$ be the unique prime ideal of norm $2$ in $\OO$; by uniqueness, $\pp$ is real, so we have a $\Q$-rational $2$-isogeny $\psi: E \ra E/E[\pp]$.  Since $\pp$ is a proper $\OO$-ideal, we have $\End(E/E[\pp]) = \OO$.
Let $\varphi = (\psi \circ \iota)^{\vee}: E/E[\pp] \ra E'$.  Then $\varphi$ is a $\Q$-rational $4$-isogeny; if it were not cyclic, then
it would be (up to an isomorphism on the target) $[2]$, but this is impossible as $\OO = \End E/E[\pp] \neq \End E' = \OO'$.  \\
Step 2: By \cite[Thm. 6.18]{BC18}, no $\OO$-CM elliptic curve has a $K$-rational cyclic $8$-isogeny, so certainly no $\OO$-CM
elliptic curve has a $\Q$-rational cyclic $8$-isogeny.
\end{proof}
\noindent
Concerning (C3):

\begin{prop}
\label{C3PROP}
Let $K = \Q(\sqrt{-3})$, and let $\OO = \OO_K$ be the ring of integers of $K$, so $\Delta = \Delta_K = -3$, $\ff = 1$ and $\Q(\ff) = \Q$.  
\begin{enumerate}
\item For $a \in \Z^+$, we have that $I(\OO,2^a)$ holds iff $a = 1$.  
\item For $a \in \Z^+$, we have that $I(\OO,3^a)$ holds iff $a \leq 2$.
\end{enumerate}
\end{prop}
\begin{proof}
a) Let $\OO'$ be the order in $K$ of conductor $2$.  Because $\Q(2) = \Q(1)$, there is an $\OO'$-CM elliptic curve
$E'_{/\Q}$ and thus a canonical $\Q$-rational $2$-isogeny $\iota: E' \ra E$ with $\End E = \OO$.  Then $\iota^{\vee}: E \ra E'$
is a $\Q$-rational cyclic $2$-isogeny.  By \cite[Thm. 6.18]{BC18}, no $\OO$-CM elliptic curve has even a $K$-rational cyclic $4$-isogeny.  (Or: if there were an $\OO$-CM elliptic curve with a $\Q$-rational cyclic $4$-isogeny, then
since $\frac{\varphi(4)}{2} = 1$, by \cite[Thm. 5.5]{BCS17} there would be an $\OO$-CM elliptic curve $E_{/\Q}$ with a $\Q$-rational
point of order $4$, which is not the case: \cite[p. 196]{Olson74}, \cite[Cor. 9.4]{Aoki95} or \cite[Thm. 5.1c)]{BCS17}.)  \\
b) Step 1: We construct a $\Q$-rational cyclic $9$-isogeny $\varphi: E \ra E'$ with $\End E = \OO$ in exactly the same way as
in the proof of Proposition \ref{C2PROP}: $\Q(3) = \Q(1) = \Q$, and there is a unique ideal $\pp$ of $\OO$ of norm $3$.
\\
Step 2: By \cite[Thm. 6.18]{BC18}, no $\OO$-CM elliptic curve has a $K$-rational cyclic $27$-isogeny, so certainly no $\OO$-CM
elliptic curve has a $\Q$-rational cyclic $27$-isogeny.
\end{proof}
\noindent
The following result completes Kwon's theorem by determining all $N \in \Z^+$ such that $I(\OO,N)$ holds for the orders $\OO$
of discriminants $-3$ and $-4$.

\begin{cor}
\label{LASTKWONCOR} 
\textbf{}
 \begin{enumerate}
\item Let $\OO$ be the imaginary quadratic order with $\Delta(\OO)=-4$.  Then $I(\OO,N)$ holds iff $N \in \{1,2,4\}$.  
\item Let $\OO$ be the imginary quadratic order with $\Delta(\OO)=-3$.  Then $I(\OO,N)$ holds iff $N \in \{1,2,3,6,9\}$.
\end{enumerate}
\end{cor}
\begin{proof}
a) By Theorem \ref{GENKWONTHM}, $I(\OO,\ell)$ does not hold for any odd prime $\ell$.  The result now follows from Proposition
\ref{C2PROP}.  \\
b) By Theorem \ref{GENKWONTHM}, $I(\OO,\ell)$ does not hold for any prime $\ell \geq 5$, so by Proposition \ref{C3PROP},
if $I(\OO,N)$ holds then $N \in \{1,2,3,6,9,18\}$.  Proposition \ref{C3PROP} also shows that $I(\OO,N)$ holds for $N \in \{1,2,3,9\}$.
In particular there is an $\OO$-CM elliptic curve $E_{/\Q}$ admitting a $\Q$-rational $2$-isogeny; let $C_2$ be
its kernel.  Let $\pp$ be the unique $\OO$-ideal of norm $3$.  Then $E \ra E/(C_2 \times E[\pp])$ is a $\Q$-rational $6$-isogeny,
so $I(\OO,6)$ holds.  By \cite[Thm. 6.18]{BC18}, no $\OO$-CM
elliptic curve has even a $K$-rational cyclic $18$-isogeny.
\end{proof}

\begin{remark}
a) As seen in the proof of Proposition \ref{C1PROP}, for an order $\OO$ of discriminant $\Delta < -4$,
whether an $\OO$-CM elliptic curve $E_{/F}$ admits an $F$-rational cyclic $N$-isogeny is independent of the $F$-model, so if
$\gcd(M,N) = 1$, then $I(\OO,MN)$ holds iff $I(\OO,M)$ holds and $I(\OO,N)$ holds. However
Corollary \ref{LASTKWONCOR} shows that when $\Delta = -3$, $I(\OO,2)$ holds and $I(\OO,9)$ holds but $I(\OO,18)$ does not hold. \\
b) It would be interesting to know whether Proposition \ref{C2PROP} continues to hold over any number field $F$ containing neither $K$ nor $\Q(4)$, and similarly for Proposition \ref{C3PROP}. \\
c) Combining Proposition \ref{C3PROP} with \cite[Thm. 5.5]{BCS17}, we find that if $\OO$ is the imaginary quadratic order of
discriminant $-3$, there is an $\OO$-CM elliptic curve defined over a cubic number field $F$ with an $F$-rational point of order $9$.
In fact \cite[Table II]{BCS17} gives an explicit such curve: we may take $F = \Q[b]/(b^3-15b^2-9b-1)$,
\[ E_{/F}: y^2 + \left(1-\left(\frac{b^2}{4}+\frac{5b}{2}+\frac{1}{4}\right)\right)xy -\left(\frac{b^2}{4}+\frac{5b}{2}+\frac{1}{4}\right)y = x^3 - bx^2, P = (0,0). \]
\end{remark}

\section{Least degrees of CM points on $X_1(\ell^b)_{/\Q}$}
\noindent
For an imaginary quadratic order $\OO$ and positive integers $M \mid N$, we let $T(\OO, M,N)$ be as defined in Theorem \ref{BIGMNTHM}. We write $T(\OO,N)$ in place of $T(\OO,1,N)$. In all cases the least degree of an $\OO$-CM point on $X_1(N)_{/\Q}$ is $T(\OO,N)$ or $2 \cdot T(\OO,N)$.

\subsection{Inert case}

\begin{thm} \label{InertThm}
Let $\OO$ be an imaginary quadratic order of discriminant $\Delta$, let $\ell$ be a prime that is inert in $\OO$, and let $b \in \Z^+$.
Then:
\begin{enumerate}
\item If $E_{/L}$ is an $\OO$-CM elliptic curve defined over a number field such that $E(L)$ has a point
of order $\ell^b$, then $T(\OO, \ell^b) \mid [L:\Q(\ff)]$.
\item Conversely, there is an extension $L \supset \Q(\ff)$ with $[L:\Q(\ff)] =T(\OO, \ell^b)$ and
an $\OO$-CM elliptic curve $E_{/L}$ such that $E(L)$ has a point of order $\ell^b$.
\end{enumerate}
\end{thm}

\begin{proof}
Part (a) follows from Theorem \ref{BIGMNTHM}. If $\ell^b \geq 3$, then part (b) follows from Lemma \ref{MODULARPULLBACKLEMMA}. If $\ell^b=2$ and $\Delta \neq -3$, then part (b) is clear: for any elliptic curve $E$ over a field $F$ of characteristic $0$ and any point $P$ of order $2$ on $E$, we have $[F(P):F] \leq 3=T(\OO, 2)$
because there are precisely $3$ points of order $2$. Finally, if $\ell^b = 2$ and $\Delta = -3$, then \cite[p. 196]{Olson74} shows there is an $\OO$-CM elliptic curve $E_{/\Q}$ with a rational point of order 2.  (One can take $E_{/\Q}: y^2 = x^3-1$.)
\end{proof}
\noindent
In Theorem \ref{InertThm}, when $\Delta < -4$ and $\ell^b \geq 3$ then we have  \[T(\OO,\ell^b) = \frac{\ell^{2b-2}(\ell^2-1)}{2} = \deg(X_1(\ell^b) \ra X(1)), \]
so it is as difficult for an $\OO$-CM elliptic to acquire a point of order $\ell^b$ as 
any elliptic curve defined over a number field.  

\subsection{Split case}

\begin{thm} \label{SplitThm} Let $\OO$ be an imaginary quadratic order of discriminant $\Delta$, and let $\ell$ be a prime that splits in $\OO$. For any $b \in \Z^+$ we have the following:
\begin{enumerate}
\item Suppose $\ell^b \geq 3$.
\begin{enumerate}
\item[i)] If $E$ is an $\OO$-CM elliptic curve defined over a number field $L$ such that $E(L)$ has a point of order $\ell^b$, then $2 \cdot T(\OO,\ell^b) \mid [L:\Q(\ff)]$.
\item[ii)] Conversely, there is an extension $L \supset \Q(\ff)$ with $[L:\Q(\ff)] =2 \cdot T(\OO,\ell^b)$ and an $\OO$-CM elliptic curve
$E_{/L}$ such that $E(L)$ has a point of order $\ell^b$.
\end{enumerate}
\item Suppose $\ell^b=2$. Then there is an $\OO$-CM elliptic curve $E_{/\Q(\ff)}$ with a point of order 2 in $E(\Q(\ff))$.
\end{enumerate}
\end{thm}
\begin{proof}
(a) Suppose $\ell^b \geq 3$. If $L \supset K$, then Theorem \ref{BIGMNTHM} implies $T(\OO,\ell^b) \mid
[L:K(\ff)]$ and so $2 \cdot T(\OO,\ell^b) \mid [L:\Q(\ff)]$.  So suppose that we have an $\OO$-CM elliptic curve $E$ defined over a number field $L$ \emph{not} containing $K$
such that $E(L)$ has a point of order $\ell^b$. By \cite[Thm. 4.8]{BCS17}, this implies that $\Z/\ell^b \Z \times \Z/\ell^b \Z
\hookrightarrow E(KL)$, and thus by \cite[Cor. 1.2, Lemma 2.2]{BC18} that
\[ 2\cdot T(\OO,\ell^b)=\frac{2\ell^{b-1}(\ell-1)}{w} \mid \ell^{b-1} (\ell-1)  \left( \frac{\ell^{b-1}(\ell-1)}{w} \right)
= \# \overline{C_{\ell^b}(\OO)} \mid [KL:K(\ff)] = [L:\Q(\ff)]. \]
As for existence: by Theorem \ref{BIGMNTHM} there is an $\OO$-CM elliptic curve
defined over an extension $L/K(\ff)$ with $[L:\Q(\ff)] = 2[L:K(\ff)] =  2 \cdot T(\OO,\ell^b)$ such that $E(L)$ has a point of order $\ell^b$. \\
(b) If $\ell^b=2$, the result follows from Theorem \ref{KWONTHM} and Remark \ref{KWONRMK} (or from \cite[Thm. 4.2b)]{BCS17}).
\end{proof}
\noindent
In fact, \cite[Thm. 4.8]{BCS17} has the following additional consequence:
\begin{prop}
Suppose that $\ell$ splits in $\OO$ and let $\ell^b \geq 3$.  For $d \in \Z^+$, suppose there is an $\OO$-CM elliptic curve
$E$ defined over an extension field $F/\Q(\ff)$ of degree $d$. 
\begin{enumerate}
\item  If $F$ does not contain $K$, then
\[ \frac{\varphi(\ell^b)^2}{w} \mid d. \]
\item Suppose $[F:\Q(\ff)] = \varphi(\ell^b)$. 
\begin{enumerate}
\item[i)] If $\Delta < -4$ and $\ell^b \geq 5$ then $F \supset K$.  
\item[ii)] If $\Delta = -4$ and $\ell^b \geq 9$ then $F \supset K$.  
\item[iii)] If $\Delta = -3$ and $\ell^b \geq 11$ then $F \supset K$.
\end{enumerate}
\end{enumerate}
\end{prop}
\begin{proof}
a) By \cite[Thm. 4.8]{BCS17} we have $E[\ell^b] \subset E(FK)$, so by \cite[Thm. 1.1]{BC18} we have
\[ \frac{\varphi(\ell^b)^2}{w} = \frac{\# C_{\ell^b}(\OO)}{w} = \# \overline{C_{\ell^b(\OO)}} \mid [FK:K(\ff)] \mid [F:\Q(\ff)] = d. \]
b) Taking $d = \varphi(\ell^b)$ in part a), we get that if $F$ does not contain $K$ then $\varphi(\ell^b) \mid w$.  If $\Delta < -4$ then
$w = 2$ and $\varphi(\ell^b) \mid 2 \implies \ell^b \leq 4$.  If $\Delta = -4$ then $w = 4$ and $\varphi(\ell^b) \leq 4 \implies
\ell^b \leq 8$.  If $\Delta = -3$ then $w= 6$ and $\varphi(\ell^b) \leq 6$ implies $\ell^b \leq 9$.
\end{proof}

\subsection{Ramified case}

Let $\OO$ be an order of conductor $\ff$ in the imaginary quadratic field $K$. Let $\ell$ be a prime ramified in $\OO$ and put $c \coloneqq \ord_\ell(\ff)$.
\\ \\
The least degree in which an $\OO$-CM elliptic curve has a rational point of order $\ell^b$ is tied to the existence of rational isogenies over $\Q(\ff)$ and $K(\ff)$, both of which have been classified. Work of Kwon concerns isogenies over $\Q(\ff)$ (see Theorem \ref{KWONTHM} and Remark \ref{KWONRMK}), and the following result determines
isogenies over $K(\ff)$.

Let $m$ be the maximum of all $b \in \Z$ such that there is an $\OO$-CM elliptic curve $E_{/\Q(\ff)}$ with a $\Q(\ff)$-rational cyclic
$\ell^b$-isogeny. Let $M$ be the supremum over all $b \in \Z$ such that there is an $\OO$-CM elliptic curve $E_{/K(\ff)}$ with a $K(\ff)$-rational
cyclic $\ell^b$-isogeny.

\begin{prop}\label{IsogenyProp}
Suppose that $\ell \mid \Delta$.  Recall that $c \coloneqq \ord_{\ell}(\ff)$.
\begin{enumerate}
\item Suppose $\ell \mid \ff$ and $\left( \frac{\Delta_K}{\ell}\right) = 1$.  Then we have:
 \begin{align*}
m&=\begin{cases}
1 &  \ell=2, c=1\\
2c-2 &\ell=2, c\geq2\\
2c &\ell \text{ is odd}
\end{cases}\\
M &= \infty
 \end{align*}
  \item Suppose $\ell \mid \ff$ and $\left( \frac{\Delta_K}{\ell}\right) = -1$.  Then we have:
 \begin{align*}
 m&=\begin{cases}
1 &  \ell=2, c=1\\
2c-2 &\ell=2, c\geq2\\
2c &\ell \text{ is odd}
\end{cases}\\
M &= 2c
 \end{align*}
 \item Suppose $\Delta \neq -3, -4$ and $\left( \frac{\Delta_K}{\ell}\right) = 0$.  Then we have:
 \begin{align*}
 m&=\begin{cases}
1 &  \ell=2, c=0\\
2c &\ell=2, c\geq1, \ord_2(\Delta_K)=2\\
2c+1 &\ell=2, c\geq1, \ord_2(\Delta_K)=3\\
2c+1 &\ell \text{ is odd}
\end{cases}\\
M &= 2c+1
 \end{align*}
 \item Suppose $\Delta=-3$ and $\ell=3$. Then we have $m=M=2$.
 \item Suppose $\Delta=-4$ and $\ell=2$. Then we have $m=M=2$.
 \end{enumerate}
 \end{prop}

\begin{proof}
Suppose $\Delta \neq -3,-4$. Theorem \ref{KWONTHM} and Remark \ref{KWONRMK} gives the value of $m$ in all cases. By Theorem \ref{KFISOGTHM}, $M$
is the maximum of all $b \in \Z^+$ such that $\Delta$ is a square modulo $\Z/4\ell^b \Z$.  The explicit determination of $M$ in terms of $c$ and $\ell$ is given in \cite[\S 7.4]{BC18}.  (In that section we have the running hypothesis that $\ord_{\ell}(\ff) \geq 1$, but the
calculation done in Case 3 is valid even when $\ord_{\ell}(\ff) = 0$.)
For $\Delta = -3,-4$, see Propositions \ref{C2PROP} and \ref{C3PROP} and their proofs.
\end{proof}

\begin{thm}\label{IsogenyTheorem} \label{THM6.5}
Suppose that $\ell \mid \Delta$. The least degree over $\Q(\ff)$ in which there is an $\OO$-CM elliptic curve with a rational point of order $\ell^b$ for $b \in \Z^+$ is as follows:
\begin{enumerate}
\item If $b \leq m$, then the least degree is $T(\OO,\ell^b)$.
\item If $m<b \leq M$, then $\ell^b>2$ and the least degree is $2 \cdot T(\OO,\ell^b)$.
\item If $b>M=m \geq 1$, then the least degree is $T(\OO,\ell^b)$.
\item If $b>M>m \geq 1$, then $\ell=2$ and the least degree is $2 \cdot T(\OO,\ell^b)$.
\end{enumerate}
\end{thm}

\begin{proof}
We will consider each case separately:

(a) Suppose $b \leq m$. \\
$\bullet$ Suppose $\Delta \neq -3,-4$. If $\ell^b = 2$, then by Proposition \ref{IsogenyProp} there is a $\Q(\ff)$-rational $2$-isogeny,
hence a $\Q(\ff)$-rational point of order $2$.  If $\ell^b > 2$, then by \cite[Thm. 5.5]{BCS17}, there is an extension $L/\Q(\ff)$ of degree $\varphi(\ell^b)/2$ and a twist $E'$ of $E_{/L}$ such that $E'(L)$ has a rational point of order $\ell^b$.  This is the least possible degree by \cite[Thm. 6.2]{BC18}.\\
$\bullet$ Suppose $\Delta=-3$.  Then $\ell=3$ and $m=2$ and the claim follows as above.  \\
$\bullet$ Suppose $\Delta = -4$.  Then $\ell = 2$ and $m = 2$, and again the claim follows as above.

(b) Suppose $m<b \leq M$. \\
By Proposition \ref{IsogenyProp}, we may assume $\Delta \neq -3, -4$. As above, we have $\ell^b > 2$, and by  \cite[Thm. 5.5]{BCS17}, there is an extension $L/K(\ff)$ of degree $\varphi(\ell^b)/2$ and an $\OO$-CM elliptic curve $E_{/L}$ with a rational point of order $\ell^b$. Furthermore, by \cite[Thm. 6.2]{BC18} if $E_{/L}$ is an $\OO$-CM elliptic curve with an $L$-rational point of order $\ell^b$,
then  \[\varphi(\ell^b)/2 \mid [L:K(\ff)], \] so the least degree over $\Q(\ff)$ in which there is an $\OO$-CM elliptic curve with a rational point of order $\ell^b$ is either $\varphi(\ell^b)/2$ or $\varphi(\ell^b)$.

Suppose for the sake of contradiction that there is a number field $L/\Q(\ff)$ of degree $\varphi(\ell^b)/2$ and an $\OO$-CM elliptic curve $E_{/L}$ such that $E(L)$ contains a rational point $P$ of order $\ell^b$. Since $\varphi(\ell^b)/2 \mid [KL:K(\ff)]$, it follows that  $[KL:K(\ff)]=\varphi(\ell^b)/2$ and $K \not \subset L$.  Since $b>m$, Theorem \ref{GENKWONTHM} implies $\Q(\ell \ff) \subset L$, and so $K(\ell\ff) \subset KL$. Recall that $K(\ell \ff)$ is the projective $\ell$-torsion point field of an $\OO$-CM elliptic curve (see \cite[Prop.3]{Parish89} and \cite[Thm. 4.1]{BC18}). Thus the image of the mod $\ell$ Galois representation of $E_{/KL}$ consists of scalar matrices. Since $E(KL)$ has a point of order $\ell$, it therefore has full $\ell$-torsion, so $\Z/\ell\Z \times \Z/\ell^b\Z \hookrightarrow E(KL)$, contradicting Theorem \ref{BIGMNTHM}.

(c) Suppose $b>M=m \geq 1$. \\ \indent
First suppose that $\ell$ is odd.  By (a), there is an extension $F/\Q(\ff)$ of
degree $\frac{\varphi(\ell^m)}{2}$ and an $\OO$-CM elliptic curve $E_{/F}$ with an $F$-rational point $P$ of order $\ell^m$.
By Lemma \ref{LITTLEDIVLEMMA}a), there is a field extension $L/F$ of degree at most $\ell^{2(b-m)}$ and $Q \in E(L)$ of
order $\ell^b$, and thus
\[ [L:\Q(\ff)] = [L:F][F:\Q(\ff)] \leq \ell^{2(b-m)} \frac{\varphi(\ell^m)}{2}= T(\OO, \ell^b).  \]
It follows from Theorem \ref{BIGMNTHM} that this is the least possible degree for an $\OO$-CM elliptic curve to have a rational point of
order $\ell^b$.  \\ \indent
Now suppose that $\ell =2$.  Then the assumptions hold in two cases: (i) $\left( \frac{\Delta_K}{2}\right) = 0$ and $c=0$; or (ii) $\left( \frac{\Delta_K}{2}\right) = 0$, $c\geq 1$, and  $\ord_2(\Delta_K)=3$. In both cases Theorem \ref{BIGMNTHM} implies that  $T(\OO,2^b)=2^{2(b-m)}\cdot \varphi(2^m)/2$ is the least possible degree, and we shall construct a point of at most this degree.  \\ \indent
Suppose (i): $\left(\frac{\Delta_K}{2} \right) = 0$ and $c = 0$.  Let us first assume that $\Delta \neq -4$.  Then $m = M = 1$,
so by (a) there is an $\OO$-CM elliptic curve defined over $\Q(\ff)$ with a $\Q(\ff)$-rational point of order $2$, and thus by Lemma \ref{LITTLEDIVLEMMA}b) there is a field extension $L/\Q(\ff)$ of degree at most $2^{2b-3} = 2^{2(b-m)} \frac{ \varphi(\ell^m)}{2}$,
an $\OO$-CM elliptic curve $E_{/L}$ and a point $Q \in E(L)$ of order $2^b$.  If $\Delta = -4$, then $m = M = 2$.   By (a) there is an $\OO_K$-CM elliptic curve $E_{/\Q}$ with a $\Q$-rational point of order $4$.  By Lemma \ref{LITTLEDIVLEMMA}a) there is a field extension $L/\Q(\ff)$ of degree at most $2^{2b-4} =
2^{2(b-m)} \frac{ \varphi(\ell^m)}{2}$ and $Q \in E(L)$ of order $2^b$.  \\ \indent
Suppose (ii): $\left( \frac{\Delta_K}{2}\right) = 0$, $c\geq 1$, and  $\ord_2(\Delta_K)=3$.  By (a), there is an extension $F/\Q(\ff)$
of degree $\frac{\varphi(2^m)}{2}$ and an $\OO$-CM elliptic curve $E_{/F}$ with an $F$-rational point of order $2^m$.  By
Lemma \ref{LITTLEDIVLEMMA}a) there is a field extension $L/F$ of degree at most $2^{2(b-m)}$ and $Q \in E(L)$ of order
$2^b$, so
\[ [L:\Q(\ff)] = [L:F][F:\Q(\ff)] \leq 2^{2(b-m)} \frac{ \varphi(2^m)}{2}. \]
\indent (d) Suppose $b>M>m\geq 1$. This case is only possible when $\ell=2$ and $2 \mid \ff$. Theorem \ref{BIGMNTHM} gives $T(\OO,2^b)=2^{2b-M-2}$, so the least degree in which there is an $\OO$-CM elliptic curve with a rational point of order $2^b$ is either $2^{2b-M-2}$ or $2^{2b-M-1}$.

Suppose for the sake of contradiction that there is a number field $L/\Q(\ff)$ of degree $2^{2b-M-2}$ and an $\OO$-CM elliptic curve $E_{/L}$ such that $E(L)$ contains a rational point $P$ of order $2^b$.  Then we must have $L=\Q(\ff)(\mathfrak{h}(P))$. Since $2^{2b-M-2} \mid [K(\ff)(P):K(\ff)]$, it follows that  \[[K(\ff)(P):K(\ff)]=[K(\ff)(\mathfrak{h}(P)):K(\ff)]=2^{2b-M-2}\] and $K \not \subset L$.
The point $2^{b-M} P$ has order $2^M$.   Suppose for the sake of contradiction that the orbit of $C_{2^M}(\OO)$ on $2^{b-M} P$ has more than $\varphi(2^M)$ elements. By \cite[\S 7.4]{BC18}, the $C_{2^b}(\OO)$-orbit on $P$ has size larger than
\[
2^{2(b-M)}\varphi(2^M)=2^{2b-M-1}.
\]
Since $[K(\ff)(\mathfrak{h}(P)):K(\ff)]$ is equal to the size of the orbit of $\overline{C_{2^b}(\OO)}$ on $\overline{P}$ (see \cite[\S7.1]{BC18}), Lemma 7.6 in \cite{BC18} implies that
$[K(\ff)(\mathfrak{h}(P)):K(\ff)]>2^{2b-M-2}$, which is a contradiction. Thus the orbit of $C_{2^M}(\OO)$ on $2^{b-M} P$ has $\varphi(2^M)$ elements, and $[K(\ff)(\mathfrak{h}(2^{b-M} P)):K(\ff)]= \frac{\varphi(2^M)}{2}$. We have the following diagram of fields:

\begin{center}
\begin{tikzpicture}[node distance=2cm]
\node (Qj)                  {$\Q(\ff)$};
\node (F) [above of=Qj, node distance=2.1cm] {$\Q(\ff)(\mathfrak{h}(2^{b-M}P))$};
\node (Kj) [above right of =Qj, node distance=2.3 cm] {$K(\ff)$};
\node (FK)  [above of=Kj, node distance=2.1 cm]   {$K(\ff)(\mathfrak{h}(2^{b-M}P))$};
\node (QjP) [above of=F, node distance=2.1cm] {$L =\Q(\ff)(\mathfrak{h}(P))$};
\node (KjP)  [above of=FK, node distance=2.1 cm]   {$K(\ff)(\mathfrak{h}(P))$};

 \draw[-] (Qj) edge node[auto] {} (F);
 \draw[-] (Qj) edge node[auto] {$2$} (Kj);
 \draw[-] (Kj) edge node[right] {$\frac{\varphi(2^M)}{2}$} (FK);
 \draw[-] (F) edge node[auto] {$2$} (FK);
 \draw[-] (F) edge node[auto] {} (QjP);
 \draw[-] (FK) edge node[right] {$2^{2(b-M)}$} (KjP);
  \draw[-] (QjP) edge node[auto] {$2$} (KjP);

\end{tikzpicture}
\end{center}
It follows that $[L:\Q(\ff)(\mathfrak{h}(2^{b-M}P))]=2^{2{b-M}}$ and thus
\[
[\Q(\ff)(\mathfrak{h}(2^{b-M}P)):\Q(\ff)]=2^{M-2}.
\]
But then we have a point of order $2^M$ in an extension of $\Q(\ff)$ of degree $2^{M-2}$, contradicting (b).
\end{proof}

\noindent
From Proposition \ref{IsogenyProp} and Theorem \ref{IsogenyTheorem} we deduce:

\begin{thm}
\label{THM5.5}
Let $\ell$ be a prime with $\ell \mid \Delta$. Then the least degree over $\Q(\ff)$ in which there is an $\OO$-CM elliptic curve with a rational point of order $\ell^b$ for $b \in \Z^+$  is as follows:
\begin{enumerate}
\item Suppose $\ell \mid \ff$ and $\left( \frac{\Delta_K}{\ell}\right) = 1$.
\begin{enumerate}
\item[i)] If $\ell=2$ and $c=1$, the least degree is $\begin{cases}
1 & b=1,\\
2 \cdot T(\OO,2^b) & b >1.
\end{cases}$
\item[ii)] If $\ell=2$ and $c \geq 2$, the least degree is $\begin{cases}
T(\OO,2^b)  &  b \leq 2c-2,\\
2 \cdot T(\OO,2^b) & b >2c-2.
\end{cases}$
\item[iii)] If $\ell$ is odd, the least degree is $\begin{cases}
T(\OO,\ell^b) & b \leq 2c,\\
2 \cdot T(\OO,\ell^b) & b >2c.
\end{cases}$
\end{enumerate}
\item Suppose $\ell \mid \ff$ and $\left( \frac{\Delta_K}{\ell}\right) = -1$.
\begin{enumerate}
\item[i)] If $\ell=2$ and $c=1$, the least degree is $\begin{cases}
1 & b=1,\\
2 \cdot T(\OO,2^b)  & b\geq 2.
\end{cases}$
\item[ii)] If $\ell=2$ and $c \geq 2$, the least degree is $\begin{cases}
T(\OO, 2^b)  & b \leq 2c-2,\\
2 \cdot T(\OO, 2^b)  & 2c-2<b 
\end{cases}$
\item[iii)] If $\ell$ is odd, the least degree is $T(\OO,\ell^b)$.
\end{enumerate}
\item Suppose $\left( \frac{\Delta_K}{\ell}\right) = 0$.
\begin{enumerate}
\item[i)] If $\ell=2$ and $c=0$, the least degree is $T(\OO, 2^b)$.
\item[ii)] If $\ell=2$,  $c \geq 1$, and $\ord_2(\Delta_K)=2$, the least degree is $\begin{cases}
T(\OO,2^b)  & b \leq 2c,\\
2 \cdot T(\OO,2^b) & b > 2c.
\end{cases}$
\item[iii)] If $\ell=2$,  $c \geq 1$, and $\ord_2(\Delta_K)=3$, the least degree is $T(\OO,2^b)$.
\item[iv)] If $\ell$ is odd, the least degree is $T(\OO,\ell^b)$.
\end{enumerate}

\end{enumerate}
\end{thm}

\subsection{An example}

\begin{example}
\label{LASTEXAMPLE}
We place ourselves in the setting of Theorem \ref{THM5.5}(a) with a prime $\ell > 2$.  Then there is a number field
$F \supset \Q(\ff)$ of degree $\frac{\ell-1}{2}$ and an $\OO$-CM elliptic curve $E_{/F}$ with an $F$-rational point $P$ of order $\ell$.  We observe that for any $b \in \Z^+$, there is an extension $L/F$ such that $[L:F]$ is odd and $E(L)$ has a point of order $\ell^b$: indeed, the $\gg_F$-set $\{Q \in E(\overline{F}) \mid \ell^{b-1} Q = P\}$ has odd order $\ell^{2b-2}$, and thus contains at least one
$\gg_F$-orbit of odd cardinality.  Overall we get a point of order $\ell^b$ over an extension $L/F$ with $\ord_2 [L:F] = \ord_2 \frac{\ell-1}{2}$.  On the other hand, when $b > 2c$ the \emph{least} degree of an extension field $F/\Q(\ff)$ for which
there is an $\OO$-CM elliptic curve with an $F$-rational point of order $\ell^b$ is $\ell^{b-1} (\ell-1)$.  Since
$\ord_2 (\ell^{b-1} (\ell-1)) = \ord_2([L:F]) + 1$, it is not the case that every degree of an extension field $F$ of $\Q(\ff)$ for which
some $\OO$-CM elliptic curve admits an $F$-rational point of $\ell^b$ is a multiple of the least such degree.  This is in distinct
contrast to Theorem \ref{BIGMNTHM}, which works over $K(\ff)$.
\end{example}

\subsection{Proof of Theorem \ref{GenIsogThm}}
Let $\OO$ be an imaginary quadratic order of discriminant $\Delta$. If $\ell \mid \Delta$, then the result follows from Theorem \ref{THM6.5}. So suppose $\ell \nmid \Delta$. As above, we let $m$ denote the maximum of all $b \in \Z$ such that there is an $\OO$-CM elliptic curve $E_{/\Q(\ff)}$ with a $\Q(\ff)$-rational cyclic
$\ell^b$-isogeny, and let $M$ be the supremum over all $b \in \Z$ such that there is an $\OO$-CM elliptic curve $E_{/K(\ff)}$ with a $K(\ff)$-rational
cyclic $\ell^b$-isogeny. Theorem \ref{GenIsogThm} can be deduced from Theorem \ref{InertThm}, Theorem \ref{SplitThm}, and the following proposition.

\begin{prop} \label{lastIsog}
Let $\OO$ be an imaginary quadratic order of discriminant $\Delta$.
\begin{enumerate}
\item Suppose $\left( \frac{\Delta}{\ell}\right) = 1$. Then we have $m=1$ if $\ell=2$ and $m=0$ otherwise. In either case $M=\infty$.
\item Suppose $\left( \frac{\Delta}{\ell}\right) = -1$. Then if $\Delta \neq -3$ or $\ell>2$, we have $m=M=0$. If $\Delta =-3$ and $\ell=2$, then $m=M=1$.
\end{enumerate}
\end{prop}
\begin{proof}
If $\Delta=-3$ or $-4$, then $m$ is given by Corollary \ref{LASTKWONCOR} and $M$ is given by Theorem 6.18 in \cite{BC18}. Thus we may assume $\Delta <-4$. The values of $m$ follow from Theorem \ref{KWONTHM} and Remark \ref{KWONRMK}. By Theorem \ref{KFISOGTHM}, there is an $\OO$-CM elliptic curve with a $K(\ff)$-rational cyclic $\ell^b$-isogeny iff there is a point $P \in \OO/\ell^b\OO$ of order $\ell^b$ with a $C_{\ell^b}(\OO)$-orbit of size $\varphi(\ell^b)$. The values of $M$ can thus be deduced from Lemma 7.1 and Theorem 7.2 in \cite{BC18}.
\end{proof}

\section{Least degrees of CM points on $X_1(N)_{/\Q}$} 
\noindent
For an order $\OO$ of conductor $\ff$ in the imaginary quadratic field $K$ and $N \in \Z^+$, let $T^{\circ}(\OO,N)$ be the least degree over $\Q(\ff)$ in which there is an $\OO$-CM elliptic curve with a rational point of order $N$.  As above, $T(\OO,N)$ denotes the least degree over $K(\ff)$ in which there is an $\OO$-CM elliptic curve with a rational point of order $N$.  
Thus: \\ \\
$\bullet$ In all cases we have $T^{\circ}(\OO,N) \in \{ T(\OO,N), 2 \cdot T(\OO,N) \}$.  \\
$\bullet$ For all $\OO$ and $N$, the quantity $T(\OO,N)=T(\OO,1,N)$ is computed in Theorem \ref{BIGMNTHM}.  \\
$\bullet$ For all $\OO$ and all prime powers $\ell^b$, the quantity $T^{\circ}(\OO,N)$ is computed in \S 6.
\\ \\
In this section we will compute $T^{\circ}(\OO,N)$ for all $\OO$ and $N$.

\begin{thm}\label{Thm7.1}
Let $\OO$ be an imaginary quadratic order of discriminant $\Delta$.  Let $N \in \Z^+$ have prime power decomposition 
$\ell_1^{b_1} \cdots \ell_r^{b_r}$ with $\ell_1 < \ldots < \ell_r$. 
\begin{enumerate}
\item If we have $T^{\circ}(\OO,\ell_i^{b_i}) = T(\OO,\ell_i^{b_i})$ for all $1 \leq i \leq r$, then $T^{\circ}(\OO,N) = T(\OO,N)$. 
\item If for some $1 \leq i \leq r$, we have $T^{\circ}(\OO,\ell_i^{b_i}) = 2 \cdot T(\OO,\ell_i^{b_i})$, then $T^{\circ}(\OO,N) = 2 \cdot T(\OO,N)$.
\end{enumerate}
\end{thm}
\begin{proof}
Put $w \coloneqq \# \OO^{\times}$. \\
a) Suppose that $T^{\circ}(\OO,\ell_i^{b_i}) = T(\OO,\ell_i^{b_i})$ for all $1 \leq i \leq r$. \\ 
Case 1: Suppose that $\ell_1^{b_1} > 2$. If $\Delta=-3$, we further assume that $\ell_1^{b_1} \neq 3$.  Then Theorem \ref{BIGMNTHM} gives
\[ T(\OO,N) = w^{r-1} \prod_{i=1}^r T(\OO,\ell_i^{b_i}). \]
Let $E_{/\C}$ be an $\OO$-CM elliptic curve.  
By \S 2.4, for all $1 \leq i \leq r$ there is $P_i \in E(\C)$ of order $\ell_i^{b_i}$ such that 
\[ [\Q(\ff)(\mathfrak{h}(P_i)):\Q(\ff)] = T(\OO,\ell_i^{b_i}). \]
There is a model of $E$ over $\Q(\ff)(\mathfrak{h}(P_1))$ such that $P_1 \in E(\Q(\ff)(\mathfrak{h}(P_1)))$.  
Put $F \coloneqq \prod_{i=1}^r \Q(\ff)(\mathfrak{h}(P_i))$.  Since $\hh(P_2),\ldots,\hh(P_r) \in F$, there is an extension $L/F$
with $[L:F] \leq w^{r-1}$ such that $P_1,\ldots,P_r \in E(L)$, and 
\[ [L:\Q(\ff)] = [L:F][F:\Q(\ff)] \leq
w^{r-1} \prod_{i=1}^r T(\OO,\ell_i^{b_i}) \] \[ = w^{r-1} \prod_{i=1}^r \frac{\widetilde{T}(\OO,\ell_i^{b_i})}{w} = 
\frac{ \prod_{i=1}^r \widetilde{T}(\OO,\ell_i^{b_i})}{w} = T(\OO,N). \]
This shows that $T^{\circ}(\OO,N) \leq T(\OO,N)$ and thus we have $T^{\circ}(\OO,N) = T(\OO,N)$.  \\ 
Case 2: Suppose $\Delta=-3$ and $\ell_1^{b_1}=3$. If $r=1$, the claim holds trivially, so we may assume $r \geq 2$. By Case 1, there is a field extension $L/\Q(\ff)$ of degree $T(\OO, \frac{N}{3})$ and an $\OO$-CM elliptic curve $E_{/L}$ with a point of order $\frac{N}{3}$ in $E(L)$. Let $\pp$ be the (unique, hence real) ideal of $\OO$ of norm $3$.  Then $E[\pp]$ is an $L$-rational cyclic subgroup scheme of 
order $3$. (See \cite[p.11]{BCS17}.) The corresponding isogeny character trivializes over a degree $2$ extension $M/L$, so $E(M)$ has a point of order $N$. Thus 
\[ T^{\circ}(\OO,N) \leq 2 \cdot T(\OO,\tfrac{N}{3}) = T(\OO,N). \]\\
Case 3: We suppose that $\ell_1^{b_1} = 2$ and $\Delta < -4$.    We have 
\[ T^{\circ}(\OO,2) = T(\OO,2) = \begin{cases} 1 & \left( \frac{\Delta}{2} \right) \neq -1 \\ 3 & \left( \frac{\Delta}{2} \right) = -1 \end{cases} \]
and (since $\Delta < -4$) for any point $P$ of order $2$ we have $\Q(\ff)(\mathfrak{h}(P)) = \Q(\ff)(P)$. If $r=1$, we are done. Otherwise, running the above argument 
for $\frac{N}{2} = \ell_2^{b_2} \cdots \ell_r^{b_r}$ we get $T^{\circ}(\OO,\frac{N}{2}) = T(\OO,\frac{N}{2})$, and then by Theorem \ref{BIGMNTHM} we have 
\[ T^{\circ}(\OO,N) \leq T(\OO,2) T(\OO,\tfrac{N}{2}) = T(\OO,N), \]
which again implies $T^{\circ}(\OO,N) = T(\OO,N)$.  \\ 
Case 4: We suppose that $\Delta = -4$ and $\ell_1^{b_1} = 2$.  Every $\OO$-CM elliptic curve $E$ defined over a field $F$ 
of characteristic $0$ has an $F$-rational point of order $2$, as can be seen from the Weierstrass equation 
$y^2 = x^3 + Ax$ or because $E \ra E/E[\pp]$ is an $F$-rational $2$-isogeny, where $\pp$ is the (unique, hence real) ideal in $\OO$ 
of norm $2$.  Now we get the result by applying Case 1 with $\frac{N}{2}$ in place of $N$.  \\
Case 5: We suppose that $\Delta = -3$, $\ell_1^{b_1} = 2$ and $\ell_2^{b_2} \neq 3$. If $r=1$ we are done, so we may assume $r \geq 2$. By Case 1 above we have 
\[ T^{\circ}(\OO,\tfrac{N}{2}) = T(\OO,\tfrac{N}{2}), \]
so there is a field extension $F/\Q(\ff)$ of degree $T(\OO,\frac{N}{2}) = 6^{r-2} \prod_{i=2}^r T(\OO,\ell_i^{b_i})$ and an $\OO$-CM elliptic curve $E_{/F}$ with an $F$-rational 
point of order $\frac{N}{2}$.   We have $\widetilde{T}(\OO,2) = 3$.   Certainly there is an extension $L/F$ of degree $3$ such that $E(L)$ has a point of order $2$, and 
thus 
\[ T^{\circ}(\OO,N) \leq 3 \cdot T(\OO,\tfrac{N}{2}) = \frac{\widetilde{T}(\OO,2)}{6} \prod_{i=2}^r \widetilde{T}(\OO,\ell_i^{b_i}) = T(\OO,N). \]
Case 6: We suppose that $\Delta = -3$, $\ell_1^{b_1} = 2$ and $\ell_2^{b_2} = 3$. If $r = 2$ then $N = 6$ and $T^{\circ}(\OO,N) = 1$ by \cite{Olson74}.  So suppose that $r \geq 3$. By Case 2, there is a field extension $F/\Q(\ff)$ of degree $T(\OO,\frac{N}{2})$ and an $\OO$-CM elliptic curve $E_{/F}$ with an $F$-rational 
point of order $\frac{N}{2}$. As above, there is an extension $L/F$ of degree $3$ such that $E(L)$ has a point of order $2$. Thus 
\[ T^{\circ}(\OO,N) \leq 3 \cdot T(\OO,\tfrac{N}{2}) = T(\OO,N). \]
b) Fix $1 \leq I \leq r$ such that $T^{\circ}(\OO,\ell_I^{b_I}) = 2 \cdot T(\OO,\ell_I^{b_I}) = \frac{\widetilde{T}(\OO,\ell_I^{b_I})}{w/2}$. 
Since we have $T^{\circ}(\OO,2) = T(\OO,2)$, we must have $\ell_I^{b_I} > 2$.    Moreover, if $\Delta = -3$ then since 
$T^{\circ}(\OO,3) = 1 = T(\OO,3)$, we must have $\ell_I^{b_I} > 3$.  Seeking a 
contradiction, we suppose there is a field extension $F/\Q(\ff)$ of degree $T(\OO,N) = \frac{\prod_{i=1}^r \widetilde{T}(\OO,\ell_i^{b_i})}{w}$, an $\OO$-CM elliptic curve $E_{/F}$ and a point $P \in E(F)$ of order $N$.  Put $P_I \coloneqq \frac{N}{\ell_I^{b_I}} P$.   By \cite[Lemma 7.6 and Prop. 7.7]{BC18} we have
\[ [K(\ff)(\mathfrak{h}(P_I)):K(\ff)] = \frac{\widetilde{T}(\OO,\ell_I^{b_I})}{w} =  \frac{T^{\circ}(\OO,\ell_I^{b_I})}{2}, \]
which implies that $K \subset \Q(\ff)(\mathfrak{h}(P_I)) \subset F$.  But then 
\[ [F:\Q(\ff)] = [F:K(\ff)][K(\ff):\Q(\ff)] = 2 [F:K(\ff)] \geq 2 \cdot T(\OO,N), \] 
a contradiction.
\end{proof}

\section{Least degrees of CM points on $X(M,N)_{/\Q(\zeta_M)}$}
\noindent
For an imaginary quadratic order $\OO$ of discriminant $\Delta$ and conductor $\ff$ and positive integers $M \mid N$, we denote by $T^{\circ}(\OO,M,N)$ the least $d \in \Z^+$ such that there is a number field $F \supset \Q(\ff)$ with $[F:\Q(\ff)] = d$, an $\OO$-CM 
elliptic curve $E_{/F}$ and an injective group homomorphism $\Z/M\Z \times \Z/N\Z \hookrightarrow E(F)$.  
\\ \\
In this section we compute $T^{\circ}(\OO,M,N)$ for all $\OO$, $M$ and $N$.  In \S 4 we computed the analogous quantity $T(\OO,M,N)$ with $\Q(\ff)$ replaced by $K(\ff)$.  In all cases we have 
\[ T^{\circ}(\OO,M,N) \in \{ T(\OO,M,N), \ 2\cdot  T(\OO,M,N) \}. \]

\begin{remark} 
\label{REMARK8.1}
\textbf{} \\
a) The quantity $T^{\circ}(\OO,1,N) = T^{\circ}(\OO,N)$ is computed for all $\OO$ and $N$ in \S 7. \\
b) By \cite[Lemma 3.15]{BCS17}, if $M \geq 3$ and $E_{/F}$ is an $\OO$-CM elliptic curve defined over a number field, 
then $F(E[M]) \supset K$ and $T^{\circ}(\OO,M,N) = 2 \cdot T(\OO,M,N)$. \\
c) Similarly, if $\Delta$ is odd, then by \cite[Lemma 3.15]{BCS17} and $\S2.5$, if $E_{/F}$ is an $\OO$-CM elliptic curve defined over a number 
field, then $F(E[2]) \supset K$ and $T^{\circ}(\OO,2,N) = 2 \cdot T(\OO,2,N)$. 
\end{remark}
\noindent
Thus it remains to consider the case in which $M = 2$, $N = 2N'$ and $\Delta$ is even.

\subsection{Preliminary lemmas} For the computation of $T^{\circ}(\OO,2,2N')$ we need two preliminary results.  The first 
concerns orbits of the level $\ell^2$ Cartan subgroup $C_{\ell^2}(\OO) = (\OO/\ell^2 \OO)^{\times}$ on points of $\OO/\ell^2 \OO$.

\begin{lemma} 
\label{ABBEYLEMMA2}
Let $\ell \mid \ff$.  Then there are $\ell+1$ orbits of $C_{\ell^2}(\OO)$ on order $\ell^2$ points of $\OO/\ell^2 \OO$: 
one of size $\ell^3(\ell-1)$ and $\ell$ of size $\ell(\ell-1)$.
\end{lemma}

\begin{proof}
Since $\# C_{\ell^2}(\OO) = \ell^2 \# C_{\ell}(\OO) = 
\ell^3(\ell-1)$, every orbit has size a divisor of $\ell^3(\ell-1)$.  Moreover, since the subgroup $(\Z/\ell^2 \Z)^{\times}$ acts faithfully 
on the points of order $\ell^2$ of $\OO/\ell^2 \OO$, the size of each orbit is divisible by $\varphi(\ell^2) = \ell(\ell-1)$.  It follows 
that each orbit has size $(\ell-1)\ell^c$ for $1 \leq c \leq 3$. Since there 
are $\ell^4-\ell^2$ points of order $\ell^2$, there are natural numbers $A,B,C$ such that 
\[ A\ell(\ell-1) + B \ell^2 (\ell-1) + C \ell^3(\ell-1) = \ell^4 - \ell^2 = \ell^2(\ell-1)(\ell+1), \]
and thus 																																									
\begin{equation}
\label{ABBEYLEMMA2EQ}
 A + B \ell + C \ell^2 = \ell(\ell+1). 
\end{equation}
The orbit of $1 \in C_{\ell^2}(\OO) = (\OO/\ell^2 \OO)^{\times}$ is $C_{\ell^2}(\OO)$, so has size $\ell^3(\ell-1)$.  Thus 
$C \geq 1$, and (\ref{ABBEYLEMMA2EQ}) shows that $C = 1$, giving 
\[ A + B\ell = \ell. \]
Since $\ell \mid \ff$, by Theorem \ref{BIGMNTHM} there is a $C_{\ell^2}(\OO)$-orbit of size $\ell(\ell-1)$, i.e., $A \geq 1$.  It follows 
that $A = \ell$ and $B = 0$, completing the proof.
\end{proof}

\begin{remark}
It is straightforward to determine the sizes and multiplicities of $C_{\ell^2}(\OO)$-orbits on order $\ell^2$ points of $\OO/\ell^2 \OO$ in all cases.  However, the analysis for $C_{\ell^3}(\OO)$ is more complicated.  
\end{remark}
\noindent
The next result computes the $2$-torsion field of an $\OO$-CM elliptic curve when the discriminant $\Delta$ is even and different from $-4$.

\begin{lemma}
\label{ABBEYLEMMA4}
Let $\OO$ have even discriminant $\Delta \neq -4$.  Let $E_{/\Q(\ff)}$ be an $\OO$-CM elliptic curve.  Then 
\[ \Q(\ff)(E[2]) = \Q(\ff)(\mathfrak{h}(E[2])) = \Q(2\ff). \]
\end{lemma}
\begin{proof}
Let $E_{/\Q(\ff)}$ be an $\OO$-CM elliptic curve.  By Lemma \ref{LITTLEPROJPROP}, we have 
\begin{equation}
\label{AL4EQ1}
\Q(\ff)(E[2]) \supset \Q(2\ff).
\end{equation}
Moreover we have (cf. \cite[Cor. 7.24]{Cox89}) that 
\begin{equation}
\label{AL4EQ2}
 [\Q(2\ff):\Q(\ff)] = [K(2\ff):K(\ff)] = 2. 
\end{equation}
By Theorem \ref{THM5.5}, every $\OO$-CM elliptic curve $E_{/\Q(\ff)}$ has a $\Q(\ff)$-rational point of order $2$, so 
\begin{equation}
\label{AL4EQ3}
 [\Q(\ff)(E[2]):\Q(\ff)] \leq 2. 
\end{equation}
Combining (\ref{AL4EQ1}), (\ref{AL4EQ2}) and (\ref{AL4EQ3}) gives the result.
  \end{proof}

\subsection{$T(\OO,2,2^b)$ when $\Delta<-4$ is even}
As in \S 6.3, let $m$ be the maximum of all $b \in \Z$ such that there is an $\OO$-CM elliptic curve $E_{/\Q(\ff)}$ with a $\Q(\ff)$-rational cyclic
$2^b$-isogeny. Since $\Delta$ is even, we have $m \geq 1$. Let $M$ be the supremum over all $b \in \Z$ such that there is an $\OO$-CM elliptic curve $E_{/K(\ff)}$ with a $K(\ff)$-rational
cyclic $2^b$-isogeny.

\begin{thm} \label{Thm8.5}
Suppose $\Delta < -4$ and $2 \mid \Delta$, and let $b \in \Z^+$.  The least $d$ such that there is a number field $F \supset \Q(\ff)$ with $[F:\Q(\ff)]=d$ and an $\OO$-CM elliptic curve $E_{/F}$ with $\Z/2\Z \times \Z/2^b\Z \hookrightarrow E(F)$ is as follows:
\begin{enumerate}
\item If $b \leq m$, then the least degree is $T(\OO,2, 2^b)$.
\item If $b = 2$ and $m = 1 < M$, then the least degree is $2 \cdot T(\OO,2,2^b)$.  
\item If $m+1 <b \leq M$, then the least degree is $2\cdot T(\OO,2, 2^b)$.
\item If $3 \leq m+1 = b \leq M$, then the least degree is $T(\OO,2,2^b)$.
\item If $b>M=m \geq 1$, then the least degree is $T(\OO,2,2^b)$.
\item If $b>M>m \geq 1$, then the least degree is $2 \cdot T(\OO,2,2^b)$.
\end{enumerate}
For explicit descriptions of $m$ and $M$, see Proposition \ref{IsogenyProp}.
\end{thm}

\begin{proof}
(a) Suppose $b \leq m$. By Theorem \ref{THM6.5}, there is an extension $L/\Q(\ff)$ of degree $T(\OO,2^b)$ and an $\OO$-CM elliptic curve $E_{/L}$ with a point of order $2^b$.  By Lemma \ref{ABBEYLEMMA4} and (\ref{AL4EQ2}) we have 
$\Q(\ff)(\mathfrak{h}(E[2]))=\Q(2\ff)$ and $[\Q(2\ff):\Q(\ff)] = 2$.  So there is an extension $F/\Q(\ff)$ of degree at most $2\cdot T(\OO,2^b)=T(\OO,2,2^b)$ such that $\Z/2\Z \times \Z/2^b\Z \hookrightarrow E(F)$. Thus equality holds.\\
\noindent
(b) Seeking a contradiction, we assume that there is a field $F/\Q(\ff)$ of degree $T(\OO,2,4)=2$ and an $\OO$-CM elliptic curve $E_{/F}$ with $\Z/2\Z \times \Z/4\Z \hookrightarrow E(F)$.  It follows that $F$ does not contain $K$.  Let $P \in E(F)$ be a point of order 4.  Lemma \ref{ABBEYLEMMA2} gives $[K(\ff)(\mathfrak{h}(P)):K(\ff)]=1$. Since $\Q(\ff)(\mathfrak{h}(P))$ does not contain $K$, it follows that $\Q(\ff)(\mathfrak{h}(P))=\Q(\ff)$, contradicting Theorem \ref{THM6.5}.
\\ 
(c) Let $F/K(\ff)$ be an extension of degree $T(\OO,2,2^b) = \varphi(2^b)$, and let $E_{/F}$ be an $\OO$-CM elliptic curve such that $\Z/2\Z \times \Z/2^b\Z \hookrightarrow E(F)$. Let $P \in E(F)$ be a point of order $2^b$. Suppose for the sake of contradiction that \[[K(\ff)(\mathfrak{h}(2P)):K(\ff)]> T(\OO,2^{b-1}) = \varphi(2^{b-1})/2. \] Then by \cite[$\S 7.4$]{BC18} we have \[[K(\ff)(\mathfrak{h}(P)):K(\ff)]> 2^2 \cdot \varphi(2^{b-1})/2=\varphi(2^b), \] which is a contradiction since $K(\ff)(\mathfrak{h}(P)) \subset F$ and $[F:K(\ff)]=\varphi(2^b)$. Thus 
\[[K(\ff)(\mathfrak{h}(2P)):K(\ff)]= \varphi(2^{b-1})/2. \]

Seeking a contradiction, we assume that there is a field $F/\Q(\ff)$ of degree $T(\OO,2,2^b) = \varphi(2^b)$ and an $\OO$-CM elliptic curve $E_{/F}$ with $\Z/2\Z \times \Z/2^b\Z \hookrightarrow E(F)$.  It follows that $F$ does not contain $K$.  Let $P \in E(F)$ be a point of order $2^b$. Since the compositum $FK(\ff)$ is an extension of $K(\ff)$ of degree $\varphi(2^b)$, the previous paragraph implies $[K(\ff)(\mathfrak{h}(2P)):K(\ff)]= \varphi(2^{b-1})/2$. Since $\Q(\ff)(\mathfrak{h}(2P))$ does not contain $K$, it follows that $[\Q(\ff)(\mathfrak{h}(2P)):\Q(\ff)]=\varphi(2^{b-1})/2$, which contradicts Theorem \ref{THM6.5}. 
\\
(d) Let $D = \frac{\Delta}{4}$ and $t = 2^{\ord_2(D)}$.  Then by \cite[Thm. 5.6.4]{Halter-Koch} we have that $I = [t,\sqrt{D}]$ is a real, 
primitive proper $\OO$-ideal.    We claim that there is an $\OO$-CM elliptic curve $E_{/\Q(\ff)}$ and a field embedding $\Q(\ff) \hookrightarrow \R$ such that $E_{/\R} \cong_{\R} \C/I$.  Indeed, there is a unique embedding $\Q(\ff) = \Q[j]/(H_{\Delta}(j)) \hookrightarrow \R$ under which $j$ maps to $j(\C/I)$, so let $E_0$ be any elliptic curve defined over $\Q(\ff)$ with $j(E_0) = j(\C/I)$.  
Then $(E_0)_{/\R}$ need not be $\R$-isomorphic to $\C/I$, but if not there is some quadratic (since $\Delta < -4$) twist $E_{/\R}$ of 
$(E_0)_{/\R}$ such that $E_{/\R} \cong_{\R} \C/I$.  However the square classes in $\R^{\times}$ are represented by $\pm 1$, so the map $\Q(\ff)^{\times}/\Q(\ff)^{\times 2} \ra \R^{\times}/\R^{\times 2}$ is surjective, and thus $E$ is the base extension of an 
elliptic curve defined over $\Q(\ff)$.  \\ \indent
As in the proof of \cite[Thm. 4.12]{BCS17} there is a $\Z_2$-basis $\widetilde{e_1}$, $\widetilde{e_2}$ of the Tate module $T_2(E)$ 
with respect to which the image of the Cartan subgroup $(\OO \otimes \Z_2)^{\times}$ of $\GL_2(\Z_2)$ is 
\[ \left\{ \left[ \begin{array}{cc} \alpha & \beta (\frac{D}{t}) \\ \beta t & \alpha \end{array} \right] \bigg{|} \ \alpha^2 - D \beta^2 \in \Z_2^{\times} \right\}, \]
and if $c \in \gg_{\Q(\ff)}$ is the complex conjugation induced by the embedding $\Q(\ff) \hookrightarrow \R$, then its image in the 
$2$-adic Galois representation is 
\[ \rho_{2^{\infty}}(c) = \left[ \begin{array}{cc} 1 & 0 \\ 0 & -1 \end{array} \right]. \]
Let $P = \widetilde{e_1} \pmod{2^b}$.  Then $P$ lies in $E(\R)$ and has order $2^b$.  Since $b-1 = m \geq 2$, Theorem 
\ref{KWONTHM} gives $\ord_2(D) \geq b-1$.  Thus $\beta t \equiv 0 \pmod{2^{b-1}}$, so $2P$ lies in a Cartan orbit of size $\varphi(2^{b-1})$ 
and $[K(\ff)(\hh(2P)):K(\ff)] = \frac{\varphi(2^{b-1})}{2}$.  Since $P \in E(\R)$, also $2P \in E(\R)$, and thus $\Q(\ff)(\hh(2P))$ 
does not contain $K$. So
\[ [\Q(\ff)(\hh(2P)):\Q(\ff)] = \frac{\varphi(2^{b-1})}{2} \]
and thus by Lemma \ref{LITTLEDIVLEMMA} we have
\[ [\Q(\ff)(\hh(P)):\Q(\ff)] \leq 2^2 [\Q(\ff)(\hh(2P)):\Q(\ff)] = \varphi(2^b) = T(\OO,2,2^b). \]
Let $F \coloneqq \Q(\ff)(\hh(P))$.  Since $F$ does not contain $K$, Theorem \ref{GENKWONTHM} gives $\Q(2\ff) \subset F$.  So there is an $\OO$-CM elliptic curve $E_{/F}$ with an $F$-rational point of order $2^b$, and since $\Q(2\ff) \subset F$, by Lemma \ref{ABBEYLEMMA4} every $\OO$-CM elliptic curve $E_{/F}$ has $E[2] \subset E(F)$.  Thus $\Z/2\Z \times \Z/2^b \Z \hookrightarrow 
E(F)$.  Since $[F:\Q(\ff)] \leq T(\OO,2,2^b)$, we are done.
\\ 
(e) Suppose $b >M=m \geq 1$. Let $c=\ord_2(\ff)$. Proposition \ref{IsogenyProp} gives that $\leg{\Delta_K}{2}=0$ and either (i) $c=0$ or (ii) $c \geq 1$ and $\ord_2(\Delta_K)=3$. In either case, $m=M=2c+1$. By Theorem \ref{THM6.5}, there is an $\OO$-CM elliptic curve $E$ with a point of order $2^b$ in an extension $F/\Q(\ff)$ of degree $T(\OO,2^b)=T(\OO,2,2^b)$.
It follows that $F$ does not contain $K$. Since $b>m$, Theorem \ref{GENKWONTHM} implies $\Q(2\ff) \subset F$. By Lemma \ref{ABBEYLEMMA4}, $E$ has full 2-torsion over $F$, a field of degree $T(\OO,2,2^b)$ over $\Q(\ff)$. \\
(f) Suppose $b>M>m \geq 1$. As above, let $c=\ord_2(\ff)$.  Then we know that one of the following occurs by Proposition \ref{IsogenyProp}: 
\begin{itemize}
\item $2 \mid \ff, \leg{\Delta_K}{2}=-1$, and $2c=M$, or
\item $2 \mid \ff, \leg{\Delta_K}{2}=0, \ord_2(\Delta_K)=2$, and $2c+1=M$.
\end{itemize}
As before, Theorem \ref{BIGMNTHM} implies $T(\OO,2,2^b)=T(\OO,2^b)$. Suppose for the sake of contradiction that there is a field extension $F/\Q(\ff)$ of degree $T(\OO,2,2^b)=T(\OO,2^b)$ and an $\OO$-CM elliptic curve $E_{/F}$ with $\Z/2\Z \times \Z/2^b\Z \hookrightarrow E(F)$.  By Theorem \ref{THM6.5} the least degree of an extension $F/\Q(\ff)$ for which an $\OO$-CM elliptic 
curve $E_{/F}$ has an $F$-rational point of order $2^b$ is $2 \cdot T(\OO,2^b)$, which gives a contradiction.
\end{proof}

\subsection{$T(\OO,2,2^b)$ when $\Delta=-4$}

\begin{prop} \label{Prop8.6}
Let $\Delta=-4$ and let $b \in \Z^+$. Then the least $d$ such that there is a number field $F \supset \Q$ with $[F:\Q]=d$ and an $\OO$-CM elliptic curve $E_{/F}$ with $\Z/2\Z \times \Z/2^b\Z \hookrightarrow E(F)$ is $T(\OO,2,2)$ if $b=1$ and $2 \cdot T(\OO,2,2^b)$ if $b>1$.
\end{prop}

\begin{proof}
Let $\OO = \Z[\sqrt{-1}]$ be the imaginary quadratic order of discriminant $-4$.  Recall that  elliptic curve $E$ defined over a field $F$ of characteristic $0$ has $\OO$-CM iff it is $F$-rationally isomorphic to a curve 
\[ E_A: y^2 = x^3 + Ax \]
for some $A \in F^{\times}$.  \cite[Prop. X.5.4]{Silverman}.  Moreover we have $E_A[2] = E_A[2](F)$ iff $-A \in F^{\times 2}$.  Taking 
$A = -1$, we recover the well known fact that $T(\OO,2,2) = 1$.   From now on we suppose that $b \geq 2$.  
\\ \indent
Seeking a contradiction, we suppose that there is a number field $F$ of degree $T(\OO,2,2^b)$ and an $\OO$-CM elliptic $E_{/F}$ with $\Z/2\Z \times \Z/2^b\Z \hookrightarrow E(F)$. So $F$ does not contain $K$ and $E \cong E_A$ with $-A \in F^{\times 2}$.    
Put $d \coloneqq \frac{2}{\sqrt{-A}} \in F$.  We have an isomorphism
\[
\varphi:  (E_A)_{/F(\sqrt{d})} \stackrel{\sim}{\ra} (E_{-4})_{F(\sqrt{d})}, \ (x, y) \mapsto \left( \frac{x}{d}, \frac{y}{d \sqrt{d}} \right).
\]
Let $P \in E(F)$ be a point of order $2^b$, and let $P' = \varphi(P) = (x_1,y_1) \in E_{-4}(F(\sqrt{d}))$, so $y_1 \in F(\sqrt{d})$ and 
$x_1 \in F$.  Then $2^{b-2} P' = (x_0,y_0)$, a point of order $4$ on $E_{-4}(F(\sqrt{d}))$ and $\Q(x_0) \subset \Q(x_1) \subset F$, as follows e.g. from the duplication formula \cite[\S III.2]{Silverman}; see also \cite[p. 32]{ADS}.  Since $F$ does not contain $K=\Q(\sqrt{-1})$ and the $4$-division polynomial of $E_{-4}$ is \[4y(x^2+4)(x^2-4x-4)(x^2+4x-4), \] we get that $x_0$ is a root of $x^2 \pm 4x-4$.   Then $\Q(\sqrt{2},\sqrt{-1}) = K(x_0)=K(x_0^2)=K(\mathfrak{h}(2^{b-2}P'))$, and  
\[
[K(\mathfrak{h}(2^{b-2}P')):K]=2.
\]
Thus $2^{b-2}P'$ is in a Cartan orbit of size 8, the larger of the two possible orbits.  (See \cite[Lemma 7.6 and Thm. 7.8]{BC18}.) We will use the ideas of \cite[$\S7.3,\S7.4$]{BC18} to show that $P'$ belongs to the larger of the two Cartan orbits of points of order $2^b$.  For an $\OO$-CM elliptic curve $E_{/\C}$ 
and a subset $S \subset E(\C)$, we denote by $\langle \langle S \rangle \rangle$ the $\OO$-submodule of $E(\C)$ generated by $S$.  
Let $\pp$ denote the unique prime of $\OO$ lying over $2$, so $\pp^2 = (2)$.   Moreover, the finite $\OO$-submodules of 
$E[2^{\infty}]$ are precisely $E[\pp^n]$ for $n \geq 0$, and we have $E[\pp^n] \cong_{\OO} \OO/\pp^n$ and thus $\# E[\pp^n] = 
\# \OO/\pp^n = 2^n$, so such a module is determined up to isomorphism by its cardinality.  It follows from the discussion of \cite[$\S 7.3$]{BC18} that the annihilator ideal of $2^{b-2}P'$ is $\mathfrak{p}^4$, so \[\langle \langle 2^{b-2}P' \rangle \rangle \cong_{\OO} \OO/\mathfrak{p}^4  \cong_{\mathbb{Z}} \Z/2^2\Z \times \mathbb{Z}/2^2 \Z.
\]
By the discussion in \cite[$\S7.4$]{BC18}, we see that $2^{b-2} \langle \langle P' \rangle \rangle  = \langle \langle 2^{b-2}P' \rangle \rangle$.  It follows that
\[
\langle \langle P' \rangle \rangle \cong_{\Z} \Z/2^b\Z \times \Z/2^b\Z,
\]
so $\# \langle \langle P' \rangle \rangle = 2^{2b}$. This implies $\langle \langle P' \rangle \rangle \cong \OO/\pp^{2b}$ and thus that the Cartan 
orbit of $P'$ has size $\# (\OO/\pp^{2b})^{\times} = 2^{2b-1}$. Therefore
\[
[K(\mathfrak{h}(P')):K]=2^{2b-3}.
\]
Since $F$ does not contain $K$, it follows that $[\Q(\mathfrak{h}(P')):\Q]=2^{2b-3}$.   Because
\[
\Q(\mathfrak{h}(P'))=\Q(x_1^2) \subset \Q(x_1)\subset F,
\] we get
\[
[F:\Q] \geq 2^{2b-3} =2 \cdot T(\OO,2,2^b),
\]
a contradiction.
\end{proof}

\subsection{$T(\OO, 2, N)$ when $2 \mid \Delta$}

\begin{thm} \label{Thm8.7}
Let $\OO$ be an imaginary quadratic order of discriminant $\Delta$ where $2 \mid \Delta$. Let $N=2^b \ell_1^{b_1} \cdots \ell_r^{b_r}$, where $\ell_i$ are distinct odd primes and $b \geq 1$. Then the least degree over $\Q(\ff)$ in which there is a number field $F/\Q(\ff)$ and an $\OO$-CM elliptic curve $E_{/F}$ with $\Z/2\Z \times \Z/N\Z \hookrightarrow E(F)$ is $T(\OO, 2, N)$ if and only if the following conditions hold:
\begin{enumerate}
\item $T^{\circ}(\OO, 2, 2^b)=T(\OO,2,2^b)$ and
\item $T^{\circ}(\OO,\ell_i^{b_i})=T(\OO,\ell_i^{b_i})$ for all $1 \leq i \leq r$.
\end{enumerate}
Otherwise, the least degree is $2 \cdot T(\OO, 2, N)$.
\end{thm}
\noindent
We first treat the case where $b=1$. Since $2 \mid \Delta$, we have $T^{\circ}(\OO, 2, 2)=T(\OO,2,2)$ by Theorem \ref{Thm8.5} and Proposition \ref{Prop8.6}, so this case is a consequence of the following lemma.

\begin{lemma} \label{Lem8.8}
Let $\OO$ be an imaginary quadratic order of discriminant $\Delta$ where $2 \mid \Delta$. Let $N=2 \ell_1^{b_1} \cdots \ell_r^{b_r}$, where $\ell_1 < \ldots < \ell_r$ are odd primes. Then $T^{\circ}(\OO,2,N)=T(\OO, 2, N)$ if and only if $T^{\circ}(\OO,\ell_i^{b_i})=T(\OO,\ell_i^{b_i})$ for all $1 \leq i \leq r$. Otherwise $T^{\circ}(\OO,2,N) = 2 \cdot T(\OO, 2, N)$.
\end{lemma}
\begin{proof} Let $N'=\ell_1^{b_1} \cdots \ell_r^{b_r}$. The case where $N'=1$ follows from Theorem \ref{Thm8.5} and Proposition \ref{Prop8.6}, so we may assume $N' \geq 3$. First suppose that $T^{\circ}(\OO,\ell_i^{b_i}) = T(\OO,\ell_i^{b_i})$ for all $1 \leq i \leq r$.  Since $\widetilde{T}(\OO,2,2) = 2$, by the work of the previous section and Theorem \ref{BIGMNTHM}, we have
\[ T(\OO,2,2N') = \frac{2}{w} \prod_{i=1}^r \widetilde{T}(\OO,\ell_i^{b_i}) = 
2 \cdot T(\OO,N') = 2 \cdot T^{\circ}(\OO,N'). \]
Any $\OO$-CM elliptic curve over a number field acquires full $2$-torsion over a quadratic extension field -- when $\Delta \neq -4$, 
this follows from \cite[Thm. 4.2]{BCS17}; and when $\Delta = -4$ it is immediate from the Weierstrass equation $y^2 = x^3 + Ax$ -- and thus 
\[ T^{\circ}(\OO,2,2N') \leq 2\cdot T^{\circ}(\OO,N') = 2 \cdot T(\OO,N') = T(\OO,2,2N'). \]
Conversely, suppose there is $1 \leq i \leq r$ such that $T^{\circ}(\OO,\ell_i^{b_i}) = 2 \cdot T(\OO,\ell_i^{b_i})$.  As in \S 7 we have $T^{\circ}(\OO,N') = 2 \cdot T(\OO,N')$. Seeking a contradiction, we suppose that there is a field extension $F/\Q(\ff)$ of degree 
\[ T(\OO,2,2N') = \frac{2}{w} \prod_{i=1}^r \widetilde{T}(\OO,\ell_i^{b_i}) \] 
and an $\OO$-CM elliptic curve $E_{/F}$ with an injective group homomorphism $\iota: \Z/2\Z \times \Z/2N' \Z \hookrightarrow E(F)$.  
Let $Q \coloneqq \iota((0,1))$.  By Lemma \ref{LEMMA4.10} and Proposition \ref{PROP4.11} we have 
\[ [K(\ff)(\hh(2Q)):K(\ff)] = \frac{\widetilde{T}(\OO,N')}{w} = \frac{T^{\circ}(\OO,N')}{2}, \]
which implies that $K \subset F$ and gives a contradiction as in the end of \S 7 above.
\end{proof}

\begin{proof}[Proof of Theorem \ref{Thm8.7}] The case in which $b = 1$ is treated in Lemma \ref{Lem8.8}, so we may assume that $b>1$. First suppose $T^{\circ}(\OO, 2, 2^b)=T(\OO,2,2^b)$ and $T^{\circ}(\OO,\ell_i^{b_i})=T(\OO,\ell_i^{b_i})$ for all $1 \leq i \leq r$. Let $N'=\ell_1^{b_1} \cdots \ell_r^{b_r}$. By assumption, there is a field extension $F/\Q(\ff)$ of degree $T(\OO,2,2^b)$ and an $\OO$-CM elliptic curve $E_{/F}$ such that $\Z/2\Z \times \Z/2^b \Z \hookrightarrow E(F)$. By Theorem \ref{Thm7.1}, there is a point $P' \in E(\overline{F})$ of order $N'$ such that $[\Q(\ff)(\mathfrak{h}(P')):\Q(\ff)]=T(\OO, N')$. Let $L \coloneqq F \cdot \Q(\ff)(\mathfrak{h}(P'))$. Since $P'$ is rational in an extension of $L$ of degree at most $w$, there is an extension of $\Q(\ff)$ of degree at most
\[
w\cdot [L:\Q(\ff)] \leq w\cdot T(\OO,N')\cdot T(\OO,2,2^b)=T(\OO,2,N)
\]
such that the base extension of $E$ to this field has a rational torsion subgroup isomorphic to $\Z/2\Z \times \Z/N\Z$. The conclusion follows. \\ \indent
Conversely, suppose there is $F/\Q(\ff)$ of degree $T(\OO, 2, N)$ and an $\OO$-CM elliptic curve $E_{/F}$ with $\Z/2\Z \times \Z/N\Z \hookrightarrow E(F)$.  If $N'=1$, then the statement is trivial, so we assume $N'>1$. We first consider the case where $\Delta \neq -4$. Let $P,Q \in E(F)$ be such that $P$ has 
order $N$, $Q$ has order $2$ and $\langle P,Q \rangle \cong \Z/2\Z \times \Z/N\Z$.   Since $\widetilde{T}(\OO,N')=\prod_{i=1}^r \widetilde{T}(\OO,\ell_i^{b_i})$, we have
\[
T(\OO,2,N)=2 \cdot T(\OO,N')\cdot T(\OO,2,2^b)=\frac{1}{2}\widetilde{T}(\OO,N') \cdot \widetilde{T}(\OO,2,2^b),
\]
and this is the size of the $\overline{C_N(\OO)}$-orbit on $\overline{(P,Q)}$.  By Lemma \ref{LEMMA4.10} the size of the $C_N(\OO)$-orbit on $(P,Q)$ is 
\[
\widetilde{T}(\OO,N') \cdot \widetilde{T}(\OO,2,2^b).
\]
Propositions \ref{FIRSTCOMPILEPROP} and \ref{PROP4.11} imply that $2^b P$ is in a $C_{N'}(\OO)$-orbit of size $\widetilde{T}(\OO,N')$ and $(N'P,Q)$ has a $C_{2^b}(\OO)$-orbit of size $\widetilde{T}(\OO,2,2^b)$. Thus 
\[
[K(\ff)(\mathfrak{h}(2^bP)):K(\ff)]=T(\OO,N')
\]
and
\[
[K(\ff)(\mathfrak{h}(N'P), \mathfrak{h}(Q)):K(\ff)]=T(\OO,2,2^b).
\]
Since $\Q(\ff)(\mathfrak{h}(2^bP))$, $\Q(\ff)(\mathfrak{h}(N'P), \mathfrak{h}(Q)) \subset F$ and $F$ does not contain $K$, it follows that $T^{\circ}(\OO, 2, 2^b)=T(\OO,2,2^b)$ and $T^{\circ}(\OO,N')=T(\OO,N')$. By Theorem \ref{Thm7.1}, we have $T^{\circ}(\OO,\ell_i^{b_i})=T(\OO,\ell_i^{b_i})$ for all $1 \leq i \leq r$, and the conclusion follows.

Finally, suppose $\Delta =-4$. Since $b>1$, Proposition \ref{Prop8.6} gives $T^{\circ}(\OO, 2, 2^b) \neq T(\OO,2,2^b)$, so the theorem will follow if we can show there is no number field $F/\Q$ of degree $T(\OO, 2, N)$ and an $\OO$-CM elliptic curve $E_{/F}$ with $\Z/2\Z \times \Z/N\Z \hookrightarrow E(F)$. For the sake of contradiction, suppose such an $E_{/F}$ exists.   It follows that $F$ does not contain $K$.  Let $P \in E(F)$ be a point of order $N$ so that $N'P$ has order $2^b$.  We claim that 
\[
[K(\mathfrak{h}(N'P)):K]=2^{2b-3}.
\]
Indeed, in the proof of Proposition \ref{Prop8.6} we begin with a point $P \in E(F)$ of order $2^b$ where $F$ is a number field 
not containing $K = \Q(\sqrt{-1})$ and construct $d \in F^{\times}$ and an isomorphism $\varphi: E_{/F(\sqrt{d})} \ra E'_{/F(\sqrt{d})}$ 
such that $[K(\mathfrak{h}(\varphi(P))):K] = 2^{2b-3}$.  By \cite[Lemma 7.6]{BC18} the $C_{2^b}(\OO)$-orbit on $\varphi(P)$ 
has size $2^{2b-1}$.  However the size of the Cartan orbit on a point does not depend upon the rational model, so the $C_{2^b}(\OO)$-orbit 
on $P$ also has size $2^{2b-1}$.  In our context this applies to show that the $C_{2^b}(\OO)$-orbit on $N'P$ has size $2^{2b-1}$.  
 Since the $C_{\ell_i^{b_i}}(\OO)$-orbit on $\frac{N}{\ell_i^{b_i}}P$ has size at least $\widetilde{T}(\OO,\ell_i^{b_i})$, by Proposition \cite[Prop. 7.7]{BC18} the $C_{N}(\OO)$-orbit on $P$ has size at least
\[
2^{2b-1}\cdot \widetilde{T}(\OO,N').
\]
Thus
\[
[K(\mathfrak{h}(P)):K]\geq \frac{1}{4}\cdot2^{2b-1}\cdot \widetilde{T}(\OO,N')=2^{2b-3}\cdot \widetilde{T}(\OO,N')=2 \cdot T(\OO,2,N),
\]
which implies
\[
[\Q(\mathfrak{h}(P)):\Q] \geq 2 \cdot T(\OO,2,N),
\]
and we have reached a contradiction since $\Q(\mathfrak{h}(P)) \subset F$.
\end{proof}

\end{document}